\documentclass{amsart}
\usepackage{amssymb,amsmath,amsthm,amsfonts,amscd,amstext}
\usepackage[mathscr]{eucal}
\usepackage{array}
\usepackage{graphicx}
\usepackage{cite}

\newtheorem{lem}{Lemma}[section]
\newtheorem{prop}[lem]{Proposition}
\newtheorem{thm}[lem]{Theorem}
\newtheorem{cor}[lem]{Corollary}
\newtheorem{conj}[lem]{Conjecture}
\newtheorem{sublem}[lem]{Sublemma}
\newtheorem{defn}{Definition}[section]
\newtheorem{rem}[lem]{Remark}
\newtheorem{exm}[lem]{Example}

\begin{document}
\title{  Covering monopole map and higher degree  in non commutative geometry}
\author{  Tsuyoshi Kato}


\Large

\date{}

\begin{abstract}

We analyze  the monopole map
over the universal covering space of a compact
 four-manifold.
We induce the property of local properness of the covering monopole map
under the condition of closedness of the Atiyah-Hitchin-Singer (AHS) complex.
In particular we construct a higher degree of the covering monopole map
when the linearized equation is isomorphic. This
  induces a homomorphism between the
  K-group of the group $C^*$ algebras.
We apply  non-linear analysis on the covering space,
which is related to $L^p$ cohomology.
We also obtain various Sobolev estimates on the covering spaces.

By applying  the Singer conjecture on $L^2$ cohomology,
we propose a conjecture of
 an aspherical version of $\frac{10}{8}$ inequality. This
 is satisfied for a large class of four-manifolds,
 including some complex
surfaces of general type.
\end{abstract}

\maketitle

\section{Introduction}
There has been a significant development 
in gauge theory on the
study of smooth structure in four-dimension.
It is based on the construction of a moduli space
that is given by the set of
solutions to some non-linear elliptic partial 
differential equation modulo gauge symmetry. It has been revealed that these
moduli spaces contain deep information on the
topology of the underlying
four-manifolds.

In relation to the Seiberg-Witten (SW) theory,
Bauer and Furuta introduced a
new invariant  \cite{bauer and furuta}.
To 
 explain this invariant, we use 
 the analogy
of the finite-dimensional case. 
Formally speaking, 
the SW moduli space is given by the zero set of a map between 
configuration spaces
that are Hilbert manifolds. If the moduli space is 
zero-dimensional,
then, by definition, 
 its algebraic number is the SW invariant. 
It is fundamental in differential topology that,
 in a case when a map
is given between finite-dimensional compact manifolds, such an algebraic
number can be recovered from the degree of the map through $K$-theory.
One may say that Bauer-Furuta (BF) theory can be considered
 an
 infinite-dimensional version of this degree theory based on the concept of 
finite-dimensional approximation.

In this paper we develop a covering version of the
degree theory and
study the monopole map over the universal covering space of a compact
four-manifold. In particular, it is crucial to induce properness of the
map, in order to apply the framework of algebraic topology to the map.

Later we explain the motivation of such a construction, but  first,
we state our main Theorem. Let $X$ be the universal covering space
of a compact, oriented smooth four-manifold $M$.

\begin{thm}\label{thm 1.1}
Suppose the Atiyah-Hitchin-Singer (AHS) complex has closed
range over the Sobolev spaces on $X$.

Then, the covering monopole map is locally strongly proper.
\end{thm}

We can apply a framework of algebraic topology constructed in
\cite{kato4}.

\begin{cor}\label{cor 1.2}
Suppose the AHS complex has closed range as above. 
Assume moreover the following conditions:

\begin{itemize}

\item
 The Dirac operator over $X$  is invertible, and
 
 \item
The second $L^2$ cohomology of the AHS complex vanishes.
 \end{itemize}
Then the covering monopole map gives a 
$\Gamma$-equivariant
$*$-homomorphism
\[
\tilde{\mu}^*: S {\frak C}(H) \to S {\frak C}_{\tilde{\mu}}(H')
\]
between certain $C^*$-algebras, where 
$\Gamma: = \pi_1(M)$ is the fundamental group of $M$.

In particular the map
 induces a homomorphism on $K$-theory
\[
\tilde{\mu}^*: K^*(C^*(\Gamma))  \to K^*(S {\frak C}_{\tilde{\mu}}(H') \rtimes \Gamma).
\]
\end{cor}

We shall present some examples of four-manifolds whose covering
spaces satisfy these conditions with respect to their spin structures.

We first describe our 
 motivation for introducing such a covering version
of BF theory from some historical perspectives.
Classical surgery theory has revealed that the fundamental group significantly impacts the 
smooth structure on a manifold.
In high dimensions,
  a smooth structure is  reduced to the
  algebraic topology of the
group ring.
Non-commutative geometry created a new framework
that unifies
the  Atiyah-Singer index theorem with coefficients and surgery theory
passing through representation theory \cite{connes}.
This topic  led to the significant development of analysis over the universal covering space. 
Atiyah-Singer index theory has been extensively 
developed over non-compact manifolds. The construction by Gromov and
Lawson is fundamental and revealed a deep relation to the
non-existence of
positive scalar curvature metrics \cite{gromov and lawson}.

The study of smooth structures
 is a core aspect of   both the
fields of non-commutative geometry and gauge theory. 
In gauge theory, the tangent space of a moduli space is given
by the index bundle of the family of elliptic operators parametrized by
the moduli space. 
Thus, 
 the Atiyah-Singer index theorem is the fundamental 
object as the local model of the moduli theory. Hence,
using 
Atiyah-Singer index theory, both 
the fields led to important developments in
differential topology.

It would be quite natural to try to combine both theories
 by introducing a systematic tool to analyze
 a  smooth structure 
 on a four-manifold from the perspective 
 of the fundamental group, and to construct
a gauge theory over non-compact four-manifolds in the framework of
non-commutative geometry. This paper is the 
first step in tackling
 this project by using SW theory and BF theory.
 It aims to construct an infinite-dimensional degree theory in 
 non-commutative geometry. This would also provide
 motivation to develop analysis
of $L^p$ cohomology theory,
 which appears naturally in non-linear analysis
over non-compact spaces.

For a better understanding of our construction, 
we  describe its finite-dimensional version.
Let $\varphi: {\mathbb R}^n \to {\mathbb R}^n$
be a proper map.
We denote  the set of continuous functions
vanishing at infinity as $C_0(X)$.
Later we will also use $C_c(X)$ to denote
 the set of continuous functions with compact support.
It induces a  $K$- theory map
$\varphi_* : K_*({\mathbb R}^n) \to K_*({\mathbb R}^n)$
via the composition of functions 
$f \in C_0({\mathbb R}^n)$ as
$f \circ \varphi \in C_0({\mathbb R}^n)$.
This gives the degree map in a standard sense.
If a discrete group $\Gamma$ acts on ${\mathbb R}^n$ and $\varphi$ is $\Gamma$-equivariant, then 
 then  the following equivariant degree map
 is induced on the equivariant $K$ theory:
$$\varphi_* : K_*^{\Gamma}  ({\mathbb R}^n) :=  K_*({\mathbb R}^n \rtimes \Gamma) 
\to K_*({\mathbb R}^n \rtimes \Gamma).$$
If  $\Gamma$ acts on $ {\mathbb R}^n$  freely, then
we have the induced  map 
$\varphi_* : K_*({\mathbb R}^n / \Gamma) \to K_*({\mathbb R}^n / \Gamma)$
over the classifying space.
 The homotopy class of 
 $\varphi: {\mathbb R}^n / \Gamma \to {\mathbb R}^n / \Gamma$
 is determined by the induced group homomorphism
$\varphi_* : \Gamma  \to \Gamma$,
where $\Gamma=  \pi_1( {\mathbb R}^n / \Gamma) $.
Note that a  straightforward  analogue of the degree in an  infinite-dimensional case 
does not exist,
 because the
  infinite-dimensional unitary group is contractible.
 
Higson, Kasparov, and Trout constructed 
a $C^*$-algebra  which is a kind of  an  infinite-dimensional Clifford algebra,
and 
 induced  an infinite-dimensional version of Bott periodicity
  between Hilbert spaces in $K$-theory  \cite{hkt}.
In this paper, we combine the  constructions 
of the BF degree theory with Higson-Kasparov-Trout 
Bott periodicity, and introduce the $K$-theoretic degree of  the covering monopole map. 
Our main aim 
here is to construct a covering monopole operator that is 
given by an equivariant $*$-homomorphism between two Clifford  $C^*$-algebras, 
which we call the {\em higher degree}  of the covering monopole map.
It induces a homomorphism between 
the equivariant $K$-groups.

To achieve this,  
we require  some analytic conditions.
The first   is the closedness of  the AHS complex
which consists of a part of the linearized operator of the covering monopole map.
This type of  property
has been studied  deeply  with respect to
$L^2$ cohomology theory,
and we can find plenty of instances of four-manifolds whose covering spaces satisfy such property \cite{gromov1}.
In this paper we construct the higher degree when the linearized map is  isomorphic.
We also present examples of four-manifolds  
 satisfying  this type of  property.
General cases will be considered  in another  papers.
We  also include some basic analysis on the covering monopole map
over general  four-manifolds.
Note that  we do not assume isomorphism
 of the linearized map until Section \ref{sec6}.

\subsection{Review of SW theory and BF theory}
\label{sec1.1}
Let us recall  the construction of the SW moduli space. Let
  $M$ be an oriented closed  four-manifold equipped with a spin$^c$ structure,
   and let $S^{\pm}$ and $L$ be the associated  rank $2$ Hermitian bundles
and their determinant bundle respectively.
The Clifford multiplication $T^*M \times S^{\pm} \to S^{\mp}$ 
induces a linear map
$\rho: \Lambda^2 \to End_{\mathbb C} (S^+)$
whose  kernel  is the sub-bundle of anti-self-dual $2$-forms and 
the image is the sub-bundle of trace free skew-Hermitian endomorphisms.

The configuration space
for the SW map consists of the set of 
$U(1)$ connections over $L$ and sections of positive spinors.
The map as
associates as:
$$F(A, \phi) := (D_A(\phi), F^+(A) - \sigma(\phi)),$$
where 
$F^+(A)$ is the self-dual part of the curvature of $A$, and
$A$ induces  a connection over the spinor bundles, which gives the associated Dirac operator.
Then, 
$\sigma(\phi)$ is given as the trace-free endomorphism:
$$ \phi \otimes \phi^* - 1/2 |\phi|^2 id$$ 
which is regarded as a self-dual $2$-form on $M$ via $\rho$.

The gauge group acts on the  configuration space, which is the set of automorphisms
on the principal spin$^c$ bundle that cover the identity on the frame bundle.
It  is given by a map from $M$ to the center $S^1$ of Spin$^c(4)$.

The SW map $F$ is equivariant with respect to the $U(1)$ gauge group actions ${\frak G}(L)$, and its
 moduli space is given by  the total
 set of solutions divided by the gauge group action:
$${\frak M}(M):  = \{ (A,\phi) : F(A,\phi)=0 \} / {\frak G}(L).$$

Now recall a basic differential topology.
Let $M$ and $N$ be  two compact oriented manifolds 
both with dimension $n$, and consider a smooth map
$f: M \to N$.
There are  two ways to extract the degree of $f$.
The first   is to count the algebraic number of the inverse
image of a generic point of $f$. The second is to use
 the   multiplication number of  the pull-back
$f^*: H^n(N:{\mathbb Z}) \to H^n(M:{\mathbb Z})$.
In general, both numbers coincide
 and the value is called the degree of $f$.
Let us  consider the case when  the SW moduli space
has  zero dimension, and
 apply  the two different interpretations of the degree to the  SW map.
The SW invariant corresponds to  the first way.
The degree construction  of the  map 
by the algebro-topological method
 is the basic idea of BF theory,
 which corresponds to the second way.

One of the  key differences from  the finite-dimensional case
is that the spaces are Sobolev spaces which  are   locally non-compact.
Hence,  more functional analytic ideas are required.
Let us recall a part of  the construction of BF theory, which 
 is based on a rather abstract formalism 
of homotopy theory on infinite-dimensional spaces by Schwarz \cite{schwarz}.
Let $H', H$ be two separable Hilbert spaces and
$F =l+c: H' \to H$ be a Fredholm map between them
such that the linearized map  $l$ 
is Fredholm and its non-linear part $c$  is compact on each bounded set.
More precisely, $c$ maps a bounded set to a relatively compact subset in $H$.
Then,
  the restrictions of $F$ on `large' finite-dimensional linear subspaces $V' \subset H'$
    composed with the projections to the image of $l$
become `asymptotically proper'  in some sense
as follows:
 $$\text{pr} \circ F: V'  \cap B_r \to V=l(V'),$$
 where $B_r \subset H'$ is
  the open ball 
 with radius $r$
 for sufficiently large $r>1$. This gives a well-defined element in the stable cohomotopy group from $ F.$

BF theory applies the above  framework to the monopole  map, which
is a modified version of the SW map, 
since the SW  map  is not proper.
The monopole map $\mu$ is defined for  the quadruplet
$(A, \phi, a,f)$,
 where $A$ is  a spin$^c$ connection, $\phi$ is a positive spinor (section of $S^+$), 
and  $a$ and $f$  are  a $1$-form and a  locally constant function, respectively.
Let $Conn$ be the set of spin$^c$-connections.
Then,
\begin{align*}
\mu: & Conn \times   (\Gamma(S^+)  \oplus \Omega^1(M) \oplus H^0(M)) \to \\
&  \qquad \qquad \qquad Conn \times (\Gamma(S^-)  \oplus \Omega^+(M) \oplus \Omega^0(M) \oplus  H^1(M))  
& \\
&(A, \phi, a,f) \to (A, D_{A+a}\phi, F^+_{A+a} - \sigma(\phi), d^*(a)+f, a_{harm})
\end{align*}
where 
$a_{harm}$ is the harmonic projection of $a$. 
The map
$\mu$ is equivariant with respect to the  action  by the gauge group ${\frak G} = map(M, {\bf T})$.

The subspace $A+ ker (d) \subset Conn$ is invariant
under the free action of the based gauge group ${\frak G}_0 \subset {\frak G}$, where the based gauge group 
consists of all automorphisms of the bundle whose 
  values
are the identity at a base point.
Its quotient is isomorphic to
the space of equivalent  classes of flat connections
$Pic(M) = H^1(M; {\mathbb R}) / H^1(M; {\mathbb Z})$.
Let us consider   the quotient spaces
\begin{align*}
& {\frak A} := (A+ ker(d)) \times_{{\frak G}_0}
  (\Gamma(S^+)  \oplus \Omega^1(M) \oplus H^0(M)) , \\
&  {\frak C} :=  (A+ ker(d)) \times_{{\frak G}_0}
 (\Gamma(S^-)  \oplus \Omega^+(M) \oplus \Omega^0(M) \oplus  H^1(M))  .
 \end{align*}
 The monopole map descends to the fibered map 
 $$\mu  : {\frak A} \to {\frak C}$$
 over $Pic(M)$.
 In general a Hilbert bundle over a compact space admits trivialization such that 
the bundle  isomorphisms 
 ${\frak A} \cong H' \times Pic(M)$ and $ {\frak C} \cong H \times Pic(M)$ hold by Kuiper's Theorem.
 Let us consider the composition with the projection
 $$\text{ pr } \circ \mu : {\frak A} \to H.$$
 Let $b^+(M)$ be the dimension of the space of self-dual harmonic $2$-forms.
 \begin{thm}\cite{bauer and furuta}\label{BF}
 Let $M$ be a compact oriented smooth four-manifold.
 The monopole map over $M$ defines an element in the stable cohomotopy group.
 
 If $b^+(M) \geq b^1(M)+1$, then
 the group  admits a natural homomorphism to the 
 group of integers, 
 and the image of  the element   coincides with the 
 SW invariant.
 \end{thm}

In this paper,  we 
 use the Clifford $C^*$-algebras $S{\frak C}(H)$.
We now state a special case of our construction.

\begin{prop}
Let $M$ be as above with $b^1(M)=0$. 
Suppose 
  the Fredholm index of $l$ is zero.
Then,  $\mu$ induces a $*$-homomorphism:
$$\mu^*: S{\frak C}(H) \to S{\frak C}(H').$$

Moreover, 
the induced map:
$$\mu^*: K(S{\frak C}(H)) \cong {\mathbb Z} \to K(S{\frak C}(H')) \cong {\mathbb Z}$$
is given by   multiplication by  the  SW invariant.
\end{prop}

Our aim is to extend the construction of the $*$-homomorphism 
over the universal covering space of a compact oriented smooth four-manifold
 equivariantly
 with respect to the fundamental group action.

 \subsection{Covering monopole map}
Let $M$ be a compact oriented smooth Riemannian  four-manifold,
 and $X = \tilde{M}$ be its universal covering space
equipped with the lift of the metric.
We denote the fundamental group $\pi_1(M)$ by $\Gamma$.
Let us fix a spin$^c$ structure on $M$.
 We assume  there exists a 
a solution
$(A_0, \psi_0)$
 to the SW equations over $M$.
Then, we denote
their  respective  lifts by $\tilde{A}_0, \tilde{\psi}_0$ over $X$.
Note that both $\tilde{A}_0$ and $\tilde{\psi}_0$ 
cannot be  in $L^2$, if they are non-zero.
Furthermore, 
 we have to choose a solution as a base point.
Otherwise any solution over the universal covering space
cannot be in $L^2$.
At this moment,
 it is not necessary to require (ir)reducibility of the base point.

In this paper we shall introduce the 
 covering monopole map
 $\tilde{\mu}= \tilde{\mu}_{A_0, \psi_0}$
  at the base $(A_0, \psi_0)$ given by 
  \begin{align*}
\tilde{\mu}: &  \ L^2_k(X;  \tilde{S}^+  \oplus  \Lambda^1 \otimes  i {\mathbb R}) \\
&  \to  
   L^2_{k-1}(X;  \tilde{S}^-  \oplus  (\Lambda^2_+ \oplus \Lambda^0) \otimes  i {\mathbb R})
\oplus  H^1_{(2)}(X)   & \\
\end{align*}
where $H^1_{(2)}(X)$ is  the first $L^2$ cohomology group with respect to the induced metric.

In general, the de Rham differential does not have closed range between  the Sobolev spaces
over a  non-compact manifold. 
This  leads to two different cases of $L^2$ cohomology theory,
reduced or unreduced ones.
If  the AHS complex has closed range over $X$, then  these $L^2$ cohomology groups
coincide and  it is  uniquely defined.
Hereinafter, we assume closedness of the AHS complex over $X$.

Let
\[
{\frak G}_{k+1}(\tilde{L}) : 
 =\exp(L^2_{k+1}(X; i \mathbb R))
 \]
 be the Sobolev gauge group.
 It is well known that this space admits
 structure of  a Hilbert manifold and 
 is a group for the pointwise multiplication
 for $k \geq 2$
 (see \cite{morgan}, page $59$).
 Later on we always assume $k \geq 2$ (see Subsection
 \ref{2.3}).
 
With respect to the gauge group action,
we  verify that the
  covering monopole map admits the $\Gamma$-equivariant global  slice
  \begin{align*}
\tilde{\mu}:  \  L^2_k(X;  \tilde{S}^+
\oplus  &   \Lambda^1 \otimes  i {\mathbb R})  \cap \text{ Ker } d^* \\
 & \to  
   L^2_{k-1}(X;  \tilde{S}^-  \oplus  \Lambda^2_+ \otimes  i {\mathbb R})
\oplus  H^1_{(2)}(X)   .
\end{align*}
The linearized operator
of the covering monopole map over  $X$ 
 is $\Gamma$-equivariant at the base point, 
 and the Atiyah's $\Gamma$-index coincides with:
\begin{align*}
\dim_{\Gamma}  \ d \tilde{\mu} & 
= \text{ ind } D - \chi_{AHS}(M) - \dim_{\Gamma} H^1_{(2)}(X) \\
& =  \text{ ind } D  - \dim_{\Gamma} H^+_{(2)}(X).
\end{align*}
where $\chi_{AHS}(M) =b_0 (M) - b_1(M) + b_2^+(M) $.
\begin{rem}
Let us consider the Kernel of
 \[
 d : L^2_{k+1}(X:  \Lambda^1 \otimes i {\mathbb R})
 \to
 L^2_{k}(X:  \Lambda^2 \otimes i {\mathbb R})
 \]
 as a closed linear subspace
 $Ker d
 \subset L^2_{k+1}(X:  \Lambda^1 \otimes i {\mathbb R})$,
and set
${\bf A}_0=
  \tilde{A}_0 + \text{ Ker } d $.
 A covering version of the BF formalism is the 
 ${\frak G}_{k+1} (L) \rtimes \Gamma$ equivariant  monopole map:
 \begin{align*}
\tilde{ \mu}: & {\bf A}_0 \times   L^2_k(X; \tilde{S}^+  \oplus \Lambda^1 \otimes  i{\mathbb R})  \to \\
&  \qquad  {\bf A}_0 \times  [ \ L^2_{k-1}(X; \tilde{S}^-
 \oplus (\Lambda^0  \oplus \Lambda^2_+) \otimes  i{\mathbb R}) \oplus  H^1_{(2)}(X) \ ].
\end{align*}
The quotient space by the gauge group is fibered over the first $L^2$ cohomology group:
$${\bf A}_0 \times_{{\frak G}_{k+1} (L)}  L^2_k(X; \tilde{S}^+  \oplus \Lambda^1 \otimes  i{\mathbb R}) 
\cong H^1_{(2)}(X) \times  L^2_k(X; \tilde{S}^+  \oplus \Lambda^1 \otimes  i{\mathbb R})  $$
and the latter space is similar.
By projecting to the fiber, 
we obtain the
 ${\frak G}_{k+1} (L) \rtimes \Gamma$ equivariant  monopole map:
 \begin{align*}
\tilde{ \mu}: & H^1_{(2)}(X) \times   L^2_k(X; \tilde{S}^+  \oplus \Lambda^1 \otimes  i{\mathbb R})  \to \\
&  \qquad  L^2_{k-1}(X; \tilde{S}^-
 \oplus (\Lambda^0  \oplus \Lambda^2_+) \otimes  i{\mathbb R}) \oplus  H^1_{(2)}(X)
\end{align*}
The $\Gamma$-index of the liberalized map is given by:
$$\dim_{\Gamma}  \ d \tilde{\mu} 
= \text{ ind } D - \chi_{AHS}(M)$$
which is a topological invariant of the base manifold $M$. 
However, we encounter difficulty in analyzing this space,
becuase  it is  not proper
whenever $H^1_{(2)}(X)$ does not vanishes.
We  will not use this version of the map in the remainder of this paper.
\end{rem}

Let $F = l+c: H' \to H$ be a smooth map between Hilbert spaces,
where $l$ is its linear part. 
We set $W = l(W') \subset H$
for  a finite dimensional linear subspace $W' \subset H'$.
Consider $\text{ pr} \circ F : W' \to W$ which is 
the restriction of $F$ composed with the orthogonal  projection to $W$.
Then we obtain  the induced homomorphism
$$(\text{ pr} \circ F)^*  : C_0(W) \to C_0(W')$$
if it 
is proper.
The basic  idea is to regard  this as an approximation of  the original map
$F: H' \to H$.
When $F$ is Fredholm,  such finite dimensional 
 restriction works effectively if we choose sufficiently large dimension of $W$.
 
 In our case of the covering monopole map,
 we must construct an `induced map' between
function spaces over infinite-dimensional linear spaces,
by using  a family of approximations as above.
Our approach is to use  the infinite-dimensional Clifford $C^*$ algebras $S{\frak C}(H)$
  \cite{hkt},  which are defined  through a kind of limit
of $C_0(W, Cl(W))$ over all finite-dimensional linear subspaces $W \subset H$.
To induce a $*$-homomorphism from $S{\frak C}(H)$, we construct  
another $C^*$-algebra
$S{\frak C}_F(H')$.
In the case of the monopole map over a compact 
four-manifold,
$F^*: S{\frak C}(H) \to S{\frak C}(H')$ is really constructed and 
$S{\frak C}_F(H')$ is given by the image
$F^*(S{\frak C}(H))$.

Throughout this paper, 
 a compact subset of a manifold refers to a compact 
submanifold of codimension zero possibly with smooth boundary.
Let $E \to X$ be a vector bundle  and
$H': =L^2_k(X; E)$ be  the Sobolev space with  the open $r$-ball denoted by 
$B_r \subset H'$.
For a compact subset $K \Subset X$, let $L^2_k(K; E)_0$ be the closure of 
$C^{\infty}_c(K;E)$ by the Sobolev $L^2_k$ norm, where the latter is 
the set of smooth functions whose supports lie in the interior of  $K$.

\begin{defn}\label{def1.1}
Let $F: H' \to H$ be a smooth map between Hilbert spaces.
It is strongly proper if  $(1)$ the preimage of a bounded set is contained in 
some bounded set, and 
$(2)$ the restriction of $F$ on  any ball $B_r $
 is proper.
 
Suppose the Hilbert spaces consist of Sobolev spaces over $X$.
Then, the  map
$F$ is locally strongly proper if it is strongly proper over
the restriction on $L^2_k(K; E)_0$   for any compact subset $K \Subset X$. 
\end{defn}
An important case of strongly proper map is given by
the monopole map 
between Sobolev spaces
over a closed four-manifold. 
In our case these two  properties hold locally since the base space is non-compact. 

\vspace{3mm}

Let
$$F =l+c : H' :=  L^2_k(X;E') \  \to \ H: = L^2_{k-1}(X; E)$$
be a $\Gamma$-equivariant
 locally strongly proper map, where $l$ is a first-order elliptic differential operator and $c$ is
pointwise and locally compact on each bounded set
(see Subsection \ref{cpt}).

In variation $(B)$
below Definition \ref{weak-f.appr}, we  introduce adaptedness for some finite-dimensional approximation
of a Sobolev space, which is a kind of compatibility condition with respect to an exhaustion of $X$ by compact subsets.

In the case when $\Gamma$ acts on $X$, we will introduce 
a weakly finite $\Gamma$-approximation in variation $(A)$,
which requires that the intersection of a weakly finite 
approximation
with its $\Gamma$-translation also
approximates the Sobolev space.

\begin{prop}\label{prop 1.6}
Suppose  $l$ is  isomorphic.
Then, there is an adapted  family of finite-dimensional linear subspaces $\{W_i'\}_i $
which finitely $\Gamma$-approximates  $F$.
\end{prop}
This is verified in  Corollary \ref{as-uni-family}, and it
follows from  Proposition \ref{kato4-fin.app} that 
 $F$ induces a $\Gamma$-equivariant  $*$-homomorphism:
$$F^*: S{\frak C}(H) \to S{\frak C}_F(H').$$
By applying the infinite-dimensional Bott periodicity by \cite{hkt}, $F^*$
 induces   a  homomorphism on the $K$-theory of the full group $C^*$-algebra:
$$F^* :  K(C^*(\Gamma))
 \to K(S{\frak C}_F(H') \rtimes \Gamma) .$$

This is a general framework and we obtain Corollary \ref{cor 1.2} 
by applying Proposition \ref{prop 1.6} to
 Theorem \ref{thm 1.1}.

\begin{rem}
$(1)$ Asymptotic morphism is a notion between 
$C^*$ algebras that
is weaker than  the usual $*$-homomorphism \cite{connes and higson}, but
still   induces a homomorphism in $K$-theory.
One might expect that 
there is a way to construct 
 a $\Gamma$-equivariant asymptotic morphism 
$\tilde{\mu}^*$ from $S{\frak C}(H) $ to $ S{\frak C}(H')$
over some classes of  covering monopole maps.

$(2)$ So far we have assumed  the condition that the linearized operator gives
an isomorphism. In general, non-zero kernel or co-kernel subspaces
 are  both infinite-dimensional if the 
 fundamental group is infinite.
To eliminate this condition, we must use some method to stabilize 
these infinite-dimensional spaces, such as 
 Kasparov's KK-theory for general construction.
\end{rem}

\subsection{Higher $\frac{10}{8}$ conjecture}\label{1.3}
    Furuta verified the following constraint on topology of smooth four-manifolds \cite{furuta}.
    In fact a stronger estimate  $b^2(M)  \geq \frac{10}{8}|\sigma(M)| +2$ is given there.
    \begin{thm}
    Let $M$ be a compact smooth spin four-manifold. 
    Then, the inequality 
    $$b^2(M)  \geq \frac{10}{8}|\sigma(M)|$$
    holds, 
    where  $\sigma(M) $ is the signature of $M$.
    \end{thm}

       The  proof uses the type of  finite-dimensional approximation described here, 
   with representation-theoretic observation over the ${\mathbb H}$ and ${\mathbb R}$
   as Pin$_2$ modules.
    More specifically, the monopole
map is reduced to a $G: = $ Pin$_2$-equivariant map 
    $\mu': V' \to V$ 
    between
finite-dimensional Pin$_2$ modules $V'$ and $V$.
Then the induced map on
the equivariant $K_G$-theory is computed as multiplication by the
degree of the monopole map:
  \[
  \alpha = 2^{b^+(M)+ \frac{\sigma(M)}{8}-1}(1-c)
  \]
  where $c$ is  part of a generating set of the representation ring $R(G)$.
  Then, the above inequality 
   follows from the positivity that
  $b^+(M)+ \frac{\sigma(M)}{8}-1 \geq 0$
  with some elementary observation. Thus, computation of the induced
map is a core part of the induction of the inequality.

    We  propose a higher version of $\frac{10}{8}$ inequality:  
            \begin{conj}\label{conj1.9}
   Suppose $M$ is a compact 
    aspherical smooth spin
   four-manifold.
   Then, the inequality: 
      $$\chi(M) \geq  \frac{10}{8}|\sigma(M)| $$
      holds,
where $\chi(M)$ is the Euler characteristic of $M$.
\end{conj}
{\em Strategy:}
    (1) We  propose
     the covering version of the inequality.
    Let $X = \tilde{M}$ be the universal covering space.
     Then, the covering $10/8$  inequality: 
    $$b^2_{\Gamma}(X) \geq \frac{10}{8}|\sigma_{\Gamma}(X)| = \frac{10}{8}|\sigma(M)| $$
    holds,
where $b^2_{\Gamma}(X)$ is the second $L^2$ betti number, and:
$$\sigma_{\Gamma}(X) = \dim_{\Gamma} H^+_{(2)}(X) - \dim_{\Gamma} H^-_{(2)}(X)$$
is the $\Gamma$-signature, which
  is equal to the signature of $M$ by  the Atiyah's $\Gamma$-index theorem.

        (2) The Singer conjecture states that if $M$ is aspherical of even dimension $2m$, then
        the $L^2$ betti numbes vanish except middle dimension
        $$b^i_{\Gamma}(X)=0 \qquad i \ne m$$
        The Singer conjecture is known to be  true for K\"ahler hyperbolic manifolds   according to Gromov \cite{gromov1}.
            This is a stronger version of the Hopf conjecture, which states
    that, under the same conditions, the non-negativity: 
    $$(-1)^m \chi(M) \geq 0$$ holds.
    The Hopf conjecture is true for
     four-dimensional hyperbolic manifolds 
     according to Chern \cite{chern}.
 This supports the Singer conjecture in four-dimensions.
 
 Suppose 
  the Singer conjecture is true for an aspherical four-manifold $M$. 
 Then, we have the equalities:
 \begin{align*}
 \chi(M) & = \chi_{\Gamma}(X) = 
 b^0_{\Gamma} (X)  - b^1_{\Gamma}(X) +b^2_{\Gamma}(X) - b^3_{\Gamma}(X) +b^4_{\Gamma}(X) \\
&  = b^2_{\Gamma}(X).
\end{align*}
 
   (3) We 
   have the inequality in conjecture \ref{conj1.9},  if we 
    combine 
    $(1)$ and $(2)$ above. 
    To verify $(1)$, we compute the `degree' of
the covering monopole map.
        
    \vspace{3mm}

     Let us check that, in the case of non positive signature values, 
     the above inequality 
        $\chi(M) \geq  \frac{10}{8}|\sigma(M)| $ follows from  another inequality
        $$\chi_{\text{AHS}}(M) \geq  -\frac{1}{8}   \sigma(M) $$
        where $\chi_{\text{AHS}}(M) = b^0(M) - b^1(M) + b^+(M)$.
    Assume that $\sigma(M)  \leq 0$ 
    is non-positive. Then, the equalities: 
    \begin{align*}
  \chi(M)  &  + \frac{10}{8}   \sigma(M) \\
 & = b^0(M) - b^1(M) + b^2(M) - b^3(M)+b^4(M) + b+ -b^- (M)+ \frac{1}{4}  \sigma(M)  \\
 & = 2( b^0(M) - b^1(M)  + b^+ (M))+  \frac{1}{4}  \sigma(M) \\
 & = 2 \chi_{\text{AHS}}(M) +  \frac{1}{4}  \sigma(M) 
 \end{align*}
 hold, according to
    the  Poincar\'e duality.

    \subsubsection{Concrete cases}
    So far, we have obtained 
     various affirmative estimates that
     support  the conjecture.
    \begin{lem}
   $ (1) $
    If $M$ is an ashperical surface bundle, then the inequality 
    $\chi(M) \geq 2 |\sigma(M)|$ holds \cite{kotschick}.

   $ (2) $
    Suppose 
     the intersection form of $M$ is even, 
     and its  fundamental group is
     amenable or realized by $\pi_1$ of a closed hyperbolic manifold of $\dim \geq 3$. 
    Then,   the inequality 
    $\chi(M) \geq \frac{10}{8} |\sigma(M)|$ holds \cite{bohr}.
    \end{lem}
     In $(1)$, one can replace $2$ with $3$, if moreover $M$ 
   admits  a complex structure.
    The proof is rather different from our approach.
    Note that
    $M$ is aspherical,   if both the base and the fiber surfaces have 
    their genus values $\geq 1$.
    
    \begin{exm}
    Atiyah constructed complex algebraic surfaces 
    $Z_g$
    with non-zero signatures \cite{atiyah2}. They admit the structure of fiber bundles whose base and fiber
    spaces are both Riemann surfaces
    of genus $\geq 3$, and hence, they are 
    aspherical.
     It is well known that the total space of
     a fiber bundle is aspherical, if the base and the fiber spaces 
     are both aspherical. One can check this by using the standard homotopy exact sequence of the fibration.

    Their signatures and Euler characteristics are respectively
    given by
    $\sigma(Z_g)= (g-1)2^{2g+1}$ 
    and $\chi(Z_g)=(g-1)2^{2g+2}(2g+1)$. 
      Then, surely the inequality
    \[
  \chi(Z_g) = 2(2g+1)
   | \sigma(Z_g) | > \frac{10}{8}  | \sigma(Z_g) |
   \]
   holds.  
    \end{exm}
    
    For $(2)$, Bohr developed  a very interesting argument 
    that replies on some group-theoretic   properties.
    In Bohr's case, $M$ is not necessarily assumed to be aspherical.
   \begin{lem}
   Suppose that $M$ is a complex surface of general type with $c_1^2 \geq 0$.
   Then, the inequality $\chi(M) \geq  \frac{12}{8}|\sigma(M)| $ holds.
   \end{lem}
   The condition of $c_1^2 \geq 0$ holds if it is minimal, or 
   $c_1^2 =3c_2$ holds.
   The latter case is given by the unit ball in ${\mathbb C}^2$
   divided by a discrete group action  by Yau.
   
   \vspace{3mm}
   
\begin{proof}
   Recall the formulas
   $\sigma(M) = \frac{1}{3}(c_1^2 - 2c_2)$ with $\chi(M) =c_2(M)$.
   Moreover positivity $ c_2 >0$ holds.
   
   We consider two cases, and suppose $\sigma >0$ holds.
   Then the strict inequality:
   $$\frac{10}{8}|\sigma(M)| = 
   \frac{10}{8} \frac{1}{3}(c_1^2 -2c_2)  \leq \frac{10}{24}c_2 =\frac{10}{24} \chi(M)$$
   holds according to the  Miyaoka-Yau inequality
   $c_1^2 \leq 3c_2$.
   
   Suppose $\sigma(M) \leq 0$ holds. Then 
    $$\frac{10}{8}|\sigma(M)| = 
   \frac{10}{8} \frac{1}{3}(2c_2-c_1^2 )  \leq \frac{10}{12}c_2 = \frac{10}{12} \chi(M)$$
   holds by non-negativity  of the Chern number.
\end{proof}   
   
   \vspace{3mm}

    \begin{prop}
    The covering $\frac{10}{8}$ inequality
     $$b^2_{\Gamma}(X) \geq \frac{10}{8}|\sigma(M)| $$
    holds for  spin four-manifolds with a residually finite fundamental group.
    \end{prop}
    
\begin{proof}
Let $\Gamma \supset \Gamma_1 \supset  \Gamma_2 \supset  \dots$
be the tower by normal subgroups of finite indices with $\cap_i \ \Gamma_i =1$.
    Let $M_i $ be $\Gamma / \Gamma_i$ spin coverings of $M$ with $\pi_1(M_i)=\Gamma_i$,
    and $X = \tilde{M}$ be the universal covering of $M$ with $\pi_1(M) =\Gamma$.
  
  Note that any covering space of a spin manifold is also spin.
    By Furuta, the inequalities hold:
    $$b_2(M_i) \geq \frac{10}{8}|\sigma(M_i)| +2.$$
    Denote $m_i = |\Gamma / \Gamma_i|$ and divide both sides by $m_i$ as
        $$\frac{b_2(M_i) }{m_i} \geq \frac{10}{8}| \frac{\sigma(M_i)}{m_i}| +\frac{2}{m_i}.$$
First,
 the equality $ \frac{\sigma(M_i)}{m_i} = \sigma(M)$ holds because the signature is multiplicative
under finite covering. 
It follows   that:
$$\lim_i \ \frac{b_2(M_i) }{m_i}  = b^2_{\Gamma}(X)$$
converges to the $L^2$-Betti number
(see Theorem $13.49$ in \cite{luck}).
Then, the conclusion holds because $ \frac{2}{m_i} \to 0$.
(The argument was suggested by Yosuke Kubota.)
       \end{proof}

    It is believed that most word hyperbolic groups are residually finite.
    For example, the fundamental groups of hyperbolic manifolds are residually finite.
    
    \begin{cor}
    Suppose  a four-manifold $M$ is aspherical and spin.
    Moreover, assume that $\pi_1(M)$ is residually finite and K\"ahler-hyperbolic.
    
    Then, the aspherical $\frac{10}{8}$ conjecture holds.
    \end{cor}

    \begin{exm}
$(1)$ 
  A K\"ahler manifold 
  is K\"ahler hyperbolic, if it is homotopy equivalent to 
  a manifold which admits a metric of negative curvature.
  
  A K\"ahler manifold $M$
  is K\"ahler hyperbolic, if $\pi_1(M)$ is word-hyperbolic and $\pi_2(M)=0$ hold.
  In particular an aspherical
  K\"ahler manifold $M$
  is K\"ahler hyperbolic, if $\pi_1(M)$ is word-hyperbolic.
\cite{gromov1}, see \cite{luck}.
  \vspace{2mm}

  $(2)$  
  An  irreducible symmetric Hermitian space
    of non-compact type
  $G/K$ is K\"ahler hyperbolic, 
  where  $G$
  is a connected non-compact simple adjoint Lie group and 
  $K$ is a maximal connected and compact Lie subgroup of G with the center $S^1$. \cite{wolf}.
  
  A complex manifold is K\"ahler hyperbolic, 
  if it is biholomorphic to a bounded symmetric
  domain in the complex plane.
    \cite{gromov1}.
  \vspace{2mm}
  
  $(3)$ 
   The product space of two K\"ahler hyperbolic manifolds
     is also   K\"ahler hyperbolic.
     
     In particular a
 compact manifold is K\"ahler hyperbolic, 
   if its universal covering space is a symmetric Hermitian space   of non-compact type.
   In fact, it is  a product of 
   irreducible symmetric Hermitian spaces   of non-compact type.

     \vspace{2mm}

   $(4)$
   A finitely generated  linear group $\Gamma$
   is residually finite.
   Let $G/K$ be the case of $(2)$. If 
    $M:=\Gamma \backslash G/K$ is a compact 
    K\"ahler manfiold,
    then $M$
    or  product spaces of $M$
    are all  K\"ahler hyperbolic such that
    $\pi_1(M)$ is residually finite,
     where $\Gamma \subset G$ is 
    a co-compact discrete subgroup.
    \end{exm}

       It is one of  approaches
       to atatck  the aspherical 10/8-inequality,
        to seek for spin four-manifolds
        with two conditions of 
         Kaehler hyperbolicity and 
          residually finiteness of the fundamental group.          
         The aspherical 10/8-inequality has been verified for such a class based on using a family of normal coverings of finite index. 
         Our ultimate goal to is to eliminate the second condition of residually finiteness and to present a more straight method by developing the analysis on the universal covering spaces of compact spin four-manifolds.
   
   \begin{rem}
   It is known that an oriented and definite
    four-manifold 
   must have a diagonal and hence odd-type  intersection form 
    (\cite{donaldson}, \cite{donaldson2}, see   \cite{bauer and furuta}).
      An aspherical four-manifold with definite form
 is not  expected to  satisfy the inequality   conjecture \ref{conj1.9}, if it exists.
   In fact in the case the inequality is given by:
   $$2(1-b_1) +b_2 \geq \frac{10}{8}b_2$$
   which is equivalent to
   $8(1-b_1)  \geq b_2  \geq 0$.
  Then, we have a contradiction to the above inequality if $b_1 >1$ holds.
   Thus,  the $b_1=1$ (and hence $b_2=0$) and $b_1=0$ cases might survive.
    Both cases  seem rare for aspherical four-manifolds.
    \end{rem}

  \vspace{3mm}

The author would like to appreciate Professor Kasparov and Professor Furuta for numerous
discussions and useful comments.
He also thanks to the referee to have spent 
for a long time to read this paper.

\section{Monopole map}\label{sec.2}
In Section \ref{sec.2},  we  briefly review 
the SW and BF theories over compact four-manifolds.
Then, we extend their constructions over universal covering spaces
of compact four-manifolds.
The $L^2$ cohomology of their fundamental groups
plays an important role in their extensions.

The setting of Sobolev spaces over the covering space $X$
is  explained in Section \ref{sec.3}.

\subsection{Clifford algebras}
Let $V$ be a real four-dimensional Euclidean space, and
consider the ${\mathbb Z}_2$-graded  Clifford algebra $Cl(V) =Cl_0(V) \oplus Cl_1(V)$. 

Let $S$ be  the unique complex four-dimensional irreducible representation of 
$Cl(V)$.
The complex involution is defined by:
$$\omega_{\mathbb C} = - e_1e_2e_3e_4$$
where $\{e_i\}_i$ is any orthonormal basis.
The involution decomposes $S$ into its eigen bundles as $S=S^+ \oplus S^-$, and
 induces the eigenspace decomposition:
\begin{align}\label{9}
Cl_0(V) \otimes {\mathbb C} \cong 
(Cl_0(V) \otimes {\mathbb C})^+ \oplus (Cl_0(V) \otimes {\mathbb C})^-
\end{align}
via  left multiplication.
It turns out that the following isomorphisms hold:
$$(Cl_0(V) \otimes {\mathbb C})^{\pm} \cong \text{ End}_{\mathbb C} (S^{\pm}).$$
Passing through the vector space isomorphism
$Cl_0(V) \cong \Lambda^0 \oplus \Lambda^2 \oplus \Lambda^4$,
the first component of the right-hand side of (\ref{9})  corresponds as: 
$$(Cl_0(V) \otimes {\mathbb C})^+ \cong
 {\mathbb C} (\frac{1+\omega_{\mathbb C}}{2}) \oplus (\Lambda^2_+(V) \otimes {\mathbb C})$$
where the self-dual form corresponds
 to the  trace-free part.
Then, for any vector $v \in S^+$, $v \otimes v^* \in \text{End}(S^+)$
 can be regarded as an element of a self-dual $2$-form:
$$\sigma(v) := v \otimes v^* - \frac{ |v|^2}{2} \text{ id } \in \ \Lambda^2_+(V) \otimes i  {\mathbb R}$$
if its trace part is extracted.

\subsection{Sobolev spaces}\label{sobolev}
In this subsection, we include some basic materials on Sobolev spaces over a complete Riemannian manifold.
Let $(X; g) $ be an $n$-dimensional complete Riemannian manifold of
bounded geometry in the sense that the
 injectivity radius is uniformly positive at any point, and 
the $C^k$-norm of  the curvature is
uniformly bounded at any point for any $k \geq 0$. 
Then, $(X; g) $ admits a uniform
local chart  $\varphi_x: D \hookrightarrow X$ with  $\varphi(0)=x$, 
  where $D \subset \mathbb R^n$   is the unit disk.
A family of smooth functions  $\{a_i\}_i$ on $D$   is called uniformly bounded,
if they admit uniformly bounded  $C^k$-norms:
\[
\sup_i \ ||a_i||_{C^k(D)} \ < \  \infty
\]
  for any $k \geq 0$.

Let $E$ be a Euclidean vector bundle on $X$, and take a connection $A_0$
on $E$. One may assume that there is a trivialization:
\[
\psi_x: E|\varphi_x(D) \ \cong \ \varphi_x(D) \times \mathbb R^m
\]
for each $x \in X$.
One may further assume that $\psi_x$  preserves their metrics,
when we equip with the standard inner product on $\mathbb R^m$. Then, we can obtain
  $\psi_x: E|\varphi_x(D)  \cong D \times \mathbb R^m$ via $\varphi_x$.
By pulling back $A_0$ on $S$  by using $\varphi_x$, it can be expressed as:
\[
\varphi_x^*(A_0) = d +a_x,
\]
where $a_x $ is a matrix-valued $1$-form on $D$.
One can choose $A_0$ such that the family $\{a_x\}_{x \in X}$  is
uniformly bounded as above. We call such a connection also uniformly
bounded. Let $\chi_i$
be a  partition of unity subordinate to a covering
$B_i := \varphi_{x_i}(D) \subset X$ for some $x_i \in X$.
Then, we write:
\[
\nabla_{A_0}(s) : = \sum_i \ \psi_i^{-1}(d(\psi_i(\chi_i s)))
\]
where $\psi_i: = \psi_{x_i}$. Note that
$d(\psi_i(\chi_is)) \in \Omega^1_c(D, \mathbb R^m)$.

There are many other connections that are uniformly bounded. In
fact, for any uniformly bounded element $a \in \Omega^1(X, \text{End} E)$, the sum
$A: =A_0+a$  is also uniformly bounded, where $\text{End} E : = P \times_{O(n)} \frak o(n)$ 
and $P$ is the frame bundle.

The Levi-Civita connection induces a connection $\nabla$ on the tensor
power $\otimes \Omega^1(X)$.
 Then, the pair $(\nabla_{A_0}, \nabla)$ gives a connection on
$\Omega^*(X, \text{End}E)$,
which we also denote by $\nabla_{A_0}$. One can operate $\nabla_{A_0}$
successively on a section $s \in \Omega^0(X, \text{End}E)$, and then obtain:
\[
\nabla^l_{A_0}(s) \in \Omega^l(X, \text{End}E).
\]

\begin{defn}
The Sobolev space of $E$ is given by completion of
$\Omega^0_c(X, \text{End}E)$  by the norm:
\[
||s||^2_{L^2_k} : = \sum_{i=0}^k \ |\nabla_{A_0}^i(s)|^2(x) vol.
\]
\end{defn}

One can verify straightforwardly that the equivalent class of the
norms is independent of the choice of $A_0$ and $g$ in the sense that, for two
choices of such pairs, the corresponding Sobolev norms $|| \quad ||', ||\quad ||$ are
equivalent as:
\[
c^{-1}||\quad||' \leq ||\quad || \leq c||\quad||'
\]
for some constant $c>0$. 
The equivalent class of the norm is
independent if we use any uniformly bounded elliptic operators instead
of a uniformly bounded connection, where the former means that the
coefficients of the differential operator are uniformly bounded.

\subsection{Monopole map  over compact four-manifolds}
\label{2.3}
Let $M$ be an oriented compact  Riemannian four-manifold 
 equiped with a spin$^c$ structure. Let 
 $S^{\pm}$ and $L$ be the Hermitian rank-$2$ bundles
and the determinant bundle, respectively.

Let $A_0$ be a smooth $U(1)$ connection on $L$. With a Riemannian metric on $M$,
$A_0$ induces a spin$^c$ connection and the associated Dirac operator $D_{A_0}$ on $S^{\pm }$.
We set  a large $k \geq 2$ and consider  the configuration space:
$${\frak D}= \{(A_0 +   a, \psi) : a \in L^2_k(M; 
\Lambda^1 \otimes  i {\mathbb R}), \ \psi \in L^2_k(M;S^+)\}.$$
Then we have the SW map:
\begin{align*}
& F : {\frak D} \to L^2_{k-1}(M;  S^- \oplus \Lambda^2_+ \otimes  i {\mathbb R}), \\
& (A_0 + a, \psi) \to (D_{A_0+a}(\psi) , F^+_{A_0 +a} - \sigma(\psi)).
\end{align*}
Note that the space of connections is independent of
the choice of $A_0$ as long as $M$ is compact.

There is   symmetry ${\frak G}_*  := L^2_{k+1}(M; S^1)_*$
that acts on ${\frak D}$
by the group of based  automorphisms 
with the identity at $* \in M$ 
on the spin$^c$ bundle.
The action of the gauge group $g$ on
the spinors is the standard one  and on
 $1$-form is given by:
$$a \mapsto a + g^{-1}dg.$$
It  trivial for both $0$ and self-dual $2$-forms.

It follows that $F$ is equivariant with respect to the gauge group action, and hence
the gauge group acts on the  zero set: 
$$\tilde{\frak M} = \{ \ (A_0+a, \psi) \in {\frak D}  : F(A_0+a, \psi) =0  \ \}.$$
Moreover, the 
quotient space ${\frak B} := {\frak D} / {\frak G}_*$ is Hausdorff.

\begin{defn}
The based SW moduli space is given by the quotient space:
$${\frak M}= \tilde{\frak M} / {\frak G}_*.$$
\end{defn}

Any connection $A_0 + a$ with $a \in L^2_k(M; \Lambda^1 \otimes i {\mathbb R})$ 
can be assumed to satisfy Ker $d^*(a)=0$ after gauge transform.
Such gauge group element  is unique, since it is based.
Therefore,  locally constant functions cannot appear.
The slice  map  is given by the restriction:
$$SW:  \  L^2_k(M; S^+) 
\oplus   (A_0 + \text{ Ker } d^*)
\to L^2_{k-1}(M; S^- 
\oplus  \Lambda^2_+ \otimes  i {\mathbb R}) $$
whose 
 zero set 
 consists of the based moduli space
 equipped with  natural $S^1$ action
 $$\tilde{\frak M}  / {\frak G}_* =
 SW^{-1}(0) \ \subset  \
 (A_0 + \text{ Ker } d^*)\oplus
 L^2_k(M; S^+)
  .$$

\vspace{3mm}

We now list some of the remarkable properties of the SW moduli space.

(1) The infinitesimal model of the moduli space is given by the elliptic complex
\begin{align*}
0 \to L^2_{k+1}(M; i {\mathbb R})  & \to 
L^2_k(M; \Lambda^1 \otimes i{\mathbb R} \oplus S^+) \\
&  \to 
L^2_{k-1}(M; \Lambda^2_+ \otimes i{\mathbb R} \oplus S^-)  \to 0
\end{align*}
where the first map is given by  $a \to (2da, - a \psi)$ and the second is given by:
$$\begin{pmatrix}
d^+ & -D\sigma(\psi) \\
\frac{1}{2} \psi & D_{A}
\end{pmatrix}$$
at $(A, \psi)$, 
where $D\sigma(\psi) $ is the differential of $\sigma$ at $\psi$.
Thus,  the formal dimension is given by
$$\text{ ind } SW  = - \chi_{\text{AHS}} + \text{ ind } D_{A_0} $$
where the right-hand side is the sum of  the negative Euler characteristic of the  AHS complex and 
the index of the Dirac operator.

(2) The moduli space consists of only the trivial solution 
when $M$ admits a positive scalar curvature.

(3) The moduli space is always compact or empty.

(4) A part of the linearized map
$$0 \to L^2_{k+1}(M)   \to 
L^2_k(M; \Lambda^1 ) 
  \to 
L^2_{k-1}(M; \Lambda^2_+ )  \to 0$$
is called the AHS complex. Let us compute the first cohomology group.

\begin{lem}\label{lem2.1}
Let $M$ be a closed four-manifold.
The first cohomology  of the AHS complex:
$$H^1 (M)= \text{ Ker } d^+ / \text{ im } d$$
is isomorphic to the 1st de Rham cohomology.
\end{lem}
\begin{proof}
There is a canonical  linear map
$H^1_{dR}(M) \to H^1(M)$.

Suppose $d^+(a)=0$ holds. 
We verify  that $d(a)=0$ also holds. 
By the Stokes' theorem,
 we have the equalities:
\begin{align*}
0 & = \int_M d(a) \wedge d(a) = \int_M d^+(a) \wedge d^+(a) + \int_M d^-(a) \wedge d^-(a) \\
&= \int_M d^-(a) \wedge d^-(a) = -||d^-(a)||_{L^2}^2.
\end{align*}
Therefore,  the inverse linear map
$H^1(M) \to H^1_{dR}(M)$ exists.
\end{proof}

\vspace{3mm}

\vspace{3mm}

The monopole map  is given by
\begin{align*}
\mu: & \ Conn  \times    [ \ 
\Gamma(S^+)  \oplus \Omega^1(M) \otimes i {\mathbb R} 
\oplus H^0(M; i {\mathbb R})  \ ] \to \\
&  \qquad  Conn \times [ \ \Gamma(S^-)  \oplus (\Omega^+(M) \oplus \Omega^0(M)) \otimes i {\mathbb R} 
 \oplus 
 H^1(M;i {\mathbb R}) \ ]
& \\
&(A, \phi, a,f) \to (A, D_{A+a}\phi, F^+_{A+a} - \sigma(\phi), d^*(a)+f, a_{harm})
\end{align*}
where $Conn$ is the space of spin$^c$ connections and
$\mu$ is equivariant with respect to the gauge group action.
The subspace $A+ ker (d) \subset Conn$ is invariant
under the free action of the based gauge group, and its  quotient is isomorphic to
the isomorphism  classes of the flat bundles
$$Pic(M) = H^1(M; {\mathbb R}) / H^1(M; {\mathbb Z}).$$

Now we denote the quotient spaces
\begin{align*}
& {\frak A} := (A+ ker(d)) \times_{{\frak G}_*}  [ \ \Gamma(S^+)  \oplus \Omega^1(M) \otimes i {\mathbb R} 
 \oplus H^0(M; i {\mathbb R} ) \ ]  , \\
&  {\frak C} :=  (A+ ker(d)) \times_{{\frak G}_*}  \\
 & \qquad \qquad
 [ \ \Gamma(S^-)  \oplus (\Omega^+(M) \oplus \Omega^0(M) )\otimes i {\mathbb R} \oplus  H^1(M; i {\mathbb R} ) \ ] .
 \end{align*}
 Then the monopole map descends to the fibered map 
 $$\mu = \tilde{\mu}/ {\frak G}_* : {\frak A} \to {\frak C} $$
 over $Pic(M)$.
 For a fixed $k$, consider the fiberwise $L^2_k$ Sobolev completion and denote it as ${\frak A}_k$.
 Similarly, we define  ${\frak C}_{k-1}$.
 Then, the monopole map  extends to  a smooth map
$\mu: {\frak A}_k \to {\frak C}_{k-1}$ over $Pic(M)$.
It is a well known fact that a Hilbert bundle over a compact space admits trivialization,
which follows from a fact that the unitary group of an infinite dimensional separable Hilbert space
is contractible.
Given a trivialization ${\frak C}_{k-1} = Pic(M) \times H_{k-1} $,
one obtains 
the monopole map  can be obtained 
by composition with the projection:
$$\mu: {\frak  A}_k \to {\frak C}_{k-1} = Pic(M) \times H_{k-1} \to H_{k-1}.$$

There 
is a finite dimensional reduction of a strongly proper map in our sense,
which allows us to define  {\em degree}.
BF theory verifies  that the 
monopole map 
$\mu: {\frak  A}_k \to  H_{k-1}$ is strongly proper
when the underlying  four-manifold is compact.
 Thus, we can  define
a degree as an element  in   the  $S^1$ equivariant stable co-homotopy group.
There is a natural homomorphism from  the $S^1$ equivariant stably co-homotopy group to integer.
The image of the degree coincides with  the SW invariant if $b^+ > b^1+1$ holds 
\cite{bauer and furuta}.

\subsection{Seiberg-Witten map on the universal covering space}
 Let $(M,g)$ be a smooth and closed Riemannian 
 four-manifold equipped with a spin$^c$ structure.
 Denote its universal covering space and fundamental group by $X = \tilde{M}$
 and $\Gamma= \pi_1(M) $ respectively. We equip with the lift of the metric  on $X$.
 The spinor  bundle $S=S^+ \oplus S^-$ over $M$ is also lifted 
 as $\tilde{S} = \tilde{S}^+ \oplus \tilde{S}^-$ over $X$, which are all $\Gamma$-invariant.
 
 Later on we will assume that
  a universal covering space
 $X$ is non-compact
  (and hence its fundamental group is infinite).

 \begin{rem}\label{exp-decay}
 Among various classes of non-compact manifolds, 
 gauge theory  has been  intensively developed
 over cylindrical four-manifolds.
 For cylindrical four-manifolds, 
 a standard  analytical approach has been established,
 which works both for SW theory and Yang-Mills theory (for example, \cite{taubes}).
A striking analytic property of 
 gauge theory over a cylindrical four-manifold 
 is the exponential decay phenomenon for solutions
 under a mild condition on the slice three-manifold.
  It follows from 
 this propertry that
 the $L^2$ norm of its curvature is integral
 in the case of an instanton.
 On the other hand, for more general classes of non-compact four-manifolds such as the hyperbolic four-plane,
  it is not so difficult to have 
 an instanton whose  $L^2$ norm of the 
 curvature takes a non-integral value.
 As a result, it cannot satisfy 
 the exponential decay estimate.
 It seems impossible to obtain a
 gauge theoretic method of analysis
 that can  work for a general class of 
 complete Riemannian non-compact  four-manifolds.
 \end{rem}

Recall Sobolev spaces in \ref{sobolev}. 
Let us fix a connection  
over  $M$ and lift it  on $X$  such that it is 
$\Gamma$-invariant. It is  uniformly bounded on $X$  in the sense defined  in \ref{sobolev},
so it can be used  with the lifted metric  to introduce Sobolev spaces
$L^2_k(X)$.
 Their norms are also $\Gamma$-invariant. For a Euclidean bundle $E \to M$, 
we consider its lift
$\tilde{E} \to X$. 
Then, we  obtain the Sobolev spaces
$L^2_k(X, \tilde{E})$ 
with the coefficient. These are also equipped with $\Gamma$-invariant
metrics. Later on we assume this property.

 Let $A_0$ be a spin$^c$ connection over $M$, and 
 $(A_0, \psi_0)$ be a smooth solution to the SW equations such that 
 the following equalities hold:
\[
 \begin{cases}
&  D_{A_0}(\psi_0)=0, \\
&  F^+_{A_0} - \sigma(\psi_0)=0.
\end{cases}
\]
Later on we will assume that there exists a solution 
$(A_0, \psi_0)$  over $M$ as above. 
This is, of course a non trivial condition.

 Let us denote its lift by $(\tilde{A}_0,\tilde{\psi}_0)$ over $X$, and put
$$ \begin{cases}
& D_{\tilde{A}_0 , \tilde{\psi}_0}(\psi,a) = D_{\tilde{A}_0 +a}(\tilde{\psi}_0 +\psi)  - D_{\tilde{A}_0 }( \tilde{\psi}_0) = 
a (\tilde{\psi}_0 +\psi) + D_{\tilde{A}_0}(\psi), \\
&  \sigma(\tilde{\psi}_0, \psi) := \sigma(\tilde{\psi}_0+ \psi) -\sigma( \tilde{\psi}_0) .
\end{cases}$$
 
 \begin{lem} \label{lem2.2}
 For  $k \geq 1$, 
 they extend to the continuous maps
  \begin{align*}
  &     D_{\tilde{A}_0 , \tilde{\psi}_0 } :   L^2_k((X,g);  \tilde{S}^+
  \oplus  \Lambda^1 \otimes i {\mathbb R} )
   \to   L^2_{k-1}((X,g);  \tilde{S}^-)  , \\
 &  \sigma(\tilde{\psi}_0, \quad ) : L^2_k((X,g);  \tilde{S}^+)  
   \to   L^2_{k-1}((X,g); \Lambda^2_+ \otimes i {\mathbb R}) .
 \end{align*}
 
 If $k \geq 3$, then the second map
  defines the continuous map:
 $$ \sigma(\tilde{\psi}_0, \quad ) : L^2_k((X,g);  \tilde{S}^+)  
   \to   L^2_k((X,g); \Lambda^2_+ \otimes i {\mathbb R}) .$$
\end{lem}
 \begin{proof}
 Note the equality:
 \begin{align*}
 \sigma(\tilde{\psi}_0, \psi) = \tilde{\psi}_0 \otimes  \psi^* + \psi \otimes \tilde{\psi}_0^* - 
 <  \tilde{\psi}_0 ,   \psi > \text{ id} + \sigma(\psi).
 \end{align*}
Since $\psi_0 \in C^{\infty}(M; S^+)$ is smooth,  there is
 a constant $C$ such that
 the following estimates hold:
 $$||\sigma( \tilde{\psi}_0 ,  \psi )||_{L^2_{k-1}} \leq C 
 || \psi ||_{L^2_{k-1}}  + || \sigma(\psi)||_{L^2_{k-1}}.$$
 Let us consider the last term. Note the estimate
 \begin{align*}
 ||\nabla^l(\psi \otimes \psi^*)||_{L^2} & = 
 || \Sigma_{\alpha +\beta =l} \nabla^{\alpha}(\psi) \otimes 
 \nabla^{\beta}(\psi^*)||_{L^2} \\
 & \leq  \Sigma_{\alpha +\beta =l}
 || \nabla^{\alpha}(\psi) \otimes  \nabla^{\beta}(\psi^*)||_{L^2} 
 \end{align*}
  with $l \leq k-1$.
  It  follows from Lemma \ref{sob-emb} $(1)$  below
  that:
 \begin{align*}
    || \nabla^{\alpha}(\psi) \otimes  \nabla^{\beta}(\psi^*)||_{L^2(X)}^2 
&  = \Sigma_{\gamma \in \Gamma} 
   || \nabla^{\alpha}(\psi) \otimes  \nabla^{\beta}(\psi^*)||_{L^2(\gamma(K))}^2 \\
   & \leq \Sigma_{\gamma \in \Gamma} 
    || \nabla^{\alpha}(\psi)||_{L^4(\gamma(K))}^2 ||  \nabla^{\beta}(\psi^*)||_{L^4(\gamma(K))}^2 \\
&   \leq C \Sigma_{\gamma \in \Gamma} 
    || \nabla^{\alpha}(\psi)||_{L^2_1(\gamma(K))}^2 ||  \nabla^{\beta}(\psi^*)||_{L^2_1(\gamma(K))}^2 \\
&   \leq C ||\psi||_{L^2_k(X)}^4,
   \end{align*}
where $K \subset X$ is a fundamental domain of the covering.

In particular, we obtain the estimate:
 $|| \sigma(\psi)||_{L^2_{k-1}} 
 \leq C||\psi||^2_{L^2_k}$, and so
  we obtain the estimate:
 $$||\sigma(\tilde{\psi}_0. \psi)||_{L^2_{k-1}} \leq C ( ||\psi||_{L^2_{k-1}} +   ||\psi||^2_{L^2_k}).$$
 This implies that $\sigma(\tilde{\psi}_0, \quad)$ is
 continuous
 from $L^2_k$ to $L^2_{k-1}$.
 
 The estimate for $  D_{\tilde{A}_0 , \tilde{\psi}_0 }$ is obtained in  the same way.

 Now suppose $k \geq 3$, and consider 
  $\nabla^{\alpha}(\psi) \otimes 
 \nabla^{\beta}(\psi^*)$ with $\alpha +\beta =l \leq k$.
 Suppose both  $\alpha$ and $\beta$ are less than or equal to $k-1$.
 Then, by  the same argument as above  the $L^2$ norm is bounded by 
 $C ||\psi||_{L^2_k(X)}^4$.
 Next suppose $\alpha=k \geq 3$ and hence $\beta =0$. 
 Then, there is a constant $C$ with $||\psi||_{C^0(\gamma(K))} \leq C ||\psi||_{L^2_k(\gamma(K))}$
 by Lemma $3.2(2)$ below.
 So we obtain the estimates:
 \begin{align*}
    || \nabla^{k}(\psi)  \otimes \psi^*||_{L^2(X)}^2 
 &  = \Sigma_{\gamma \in \Gamma} 
   || \nabla^{k}(\psi) \otimes  \psi^*||_{L^2(\gamma(K))}^2 \\
   &  \leq \Sigma_{\gamma \in \Gamma} 
    || \psi||_{L^2_k(\gamma(K))}^2 ||  \psi^*||_{C^0(\gamma(K))}^2 \\
    & \leq C
      || \psi||_{L^2_k(\gamma(K))}^2 ||  \psi^*||_{L^2_k(\gamma(K))}^2 \leq C
     ||\psi||_{L^2_k(X)}^4.
   \end{align*}
\end{proof}

 \vspace{3mm}
 
 Later,  we will choose a large $k >> 1$, unless otherwise
  stated.
 
 \begin{defn}
  The covering SW map 
  at the base $(A_0, \psi_0)$ is given by:
 \begin{align*}
  F_{\tilde{A}_0 , \tilde{\psi}_0} : L^2_k((X,g);  \tilde{S}^+ & \oplus  \Lambda^1 \otimes  i {\mathbb R}) 
   \to   L^2_{k-1}((X,g);  \tilde{S}^-  \oplus \Lambda^2_+ \otimes  i {\mathbb R}) \\
  (\psi,a) &  \mapsto  (D_{\tilde{A}_0 + a}(\tilde{\psi}_0 +\psi) , F^+_{\tilde{A}_0+a} 
  -  \sigma(\tilde{\psi}_0+ \psi)    ) \\
  & \qquad = ( D_{\tilde{A}_0 , \tilde{\psi}_0 }(\psi,a), d^+(a) -  
  \sigma(\tilde{\psi}_0, \psi) )
 \end{align*}
 \end{defn}
 The fundamental group $\Gamma$ acts equivariantly on the covering SW map.

 The gauge group ${\frak G}_{k+1}$ over the spin$^c$ bundle is defined by
 $${\frak G}_{k+1} = \exp(L^2_{k+1}(X; i {\mathbb R})),$$
 which is based at infinity, and
 admits the
 structure of a Hilbert commutative Lie group for $k \geq 3$
 (see Corollary $3.3(1)$ below).
 The action is given by:
 $$(A, \psi) \mapsto ((\det g)^*(A), g^{-1}\psi)$$
 for $g \in {\frak G}_{k+1}$.
 Note the equality $\det \sigma = \sigma^2$ for  $\sigma: X \to S^1 \subset {\mathbb C}$.
 These two groups are combined into ${\frak G}_{k+1} \rtimes \Gamma$,
 and 
 the covering SW map is 
 equivariant with respect to this new  group.
 Here, the homoomrphism 
 $\Gamma \to \text{Aut}({\frak G}_{k+1} )$ is given by
 the pull-back via deck transformations.

 \begin{lem}
 The gauge group acts  equivariantly
 on the covering SW map at the base $(A_0, \psi_0)$.
  \end{lem}
\begin{proof}
We must  verify that a gauge-transformed connection 
has to be in $\tilde{A}_0 + L^2_k$.
Let $g =\exp(f)$ with $f \in L^2_{k+1}(X; i {\mathbb R})$. Then, the
 conclusion follows from the equality:
$$(\det g)^*(\tilde{A}_0 ) -\tilde{A}_0  = 2df
\in L^2_k(X; \Lambda^1 \otimes i {\mathbb R}).$$
\end{proof}


\subsubsection{Reducible case}
The covering SW map becomes simpler if we use  a reducible solution
 $(A_0, 0)$ as the base, 
 in which case $A_0$ satisfies the anti-self-dual
 (ASD) equation
 $F^+_{A_0}=0$.
Then, the map   is given by:
 \begin{align*}
  F_{\tilde{A}_0} : L^2_k((X,g);  \tilde{S}^+  \oplus  \Lambda^1 \otimes i {\mathbb R})
 &  \to   L^2_{k-1}((X,g);  \tilde{S}^-  \oplus  \Lambda^2_+\otimes i {\mathbb R} ) \\
  (\psi,a) &  \to  (D_{\tilde{A}_0 + a}(\psi) , d^+(a) - \sigma(\psi)).
 \end{align*}
Note that the moduli space of $U(1)$ ASD connections consists of 
the space of harmonic ASD $2$-forms on $M$ whose cohomology class coincides with 
the first Chern class of the $U(1)$ bundle.

 \subsubsection{Equivariant gauge fixing}
  Let us say that AHS complex is {\em closed}, if the differentials:
$$0 \to L^2_{k+1}(X) \to  L^2_k(X; \Lambda^1) \to L^2_{k-1}(X; \Lambda^2_+) \to 0$$
 have closed range.

\begin{lem}
Suppose the AHS complex is closed.. Then the first  cohomology group
$H^1 (X)= \text{ Ker } d^+ / \text{ im } d$
is isomorphic to the  $L^2$ first   de Rham cohomology.
\end{lem}
\begin{proof}
This follows by the same argument as Lemma \ref{lem2.1}.

\end{proof}

  
 \begin{rem}
 Later, we see several classes of universal covering spaces
 whose AHS complexes are closed.
In many cases this property depends only on 
 the large scale analytic property of their fundamental groups. 
 \end{rem}

 Suppose the AHS complex is  closed,
 and consider the space of  $L^2$ harmonic one-forms:
 \begin{align*}
 {\frak H}^1  = 
 \text{ Ker  }  [ \ d^* \oplus d^+: 
 L^2_{k} & ((X,g);      \Lambda^1  \otimes    i {\mathbb R}) \\
 & \to 
 L^2_{k-1}((X,g);      ( \Lambda^0 \oplus   \Lambda^2_+ ) \otimes   i {\mathbb R})  \ ] 
 .
 \end{align*}
 For any $a \in 
 L^2_{k}((X,g);     \Lambda^1  \otimes    i {\mathbb R}) $,
 consider the orthogonal decomposition
 $a =a_1 \oplus a_2 \in {\frak H}^1 
  \oplus ({\frak H}^1 )^{\perp}$.
 We denote $a_1$ by  $a_{harm} \in  {\frak H}^1$.

Let us state the equivariant gauge fixing.
\begin{prop}\label{prop2.6}
 Suppose the AHS complex is closed.
 
 Then there is a global and $\Gamma$-equivariant gauge fixing
such that   the covering SW  map is 
restricted to the slice
 \begin{align*}
  F_{\tilde{A}_0 , \tilde{\psi}_0} : L^2_k(X;  \tilde{S}^+) \oplus 
&   L^2_k(X;     \Lambda^1 \otimes    i {\mathbb R})  \cap \text{Ker } d^*  \\
&       \to   L^2_{k-1}((X,g);  \tilde{S}^-  \oplus \Lambda^2_+ \otimes  i {\mathbb R}) .
 \end{align*}
 
 More strongly, the following holds.
 For any  $A = \tilde{A}_0+a$ with $a \in L^2_k(X; \Lambda^1 \otimes  i{\mathbb R})$,
 there is $\sigma \in {\frak G}_{k+1}$ such that
 $(\det \sigma)^*(A) := \tilde{A}_0 +a'$ satisfies
 the equality
 $d^*(a')=0$ with the estimate
 $$||a'||_{L^2_k(X)} \leq C (||d^+(a)||_{L^2_{k-1}(X)} +||a_{harm}||) $$
 for some constant $C$ independent of $A_0$.

 \end{prop}
 Compare this with Lemma $5.3.1$ in \cite{morgan}.

 \begin{proof}
 {\bf Step 1:}
 Take two elements 
 $A:=\tilde{A}_0+a$ and 
 $A':=\tilde{A}_0+a'$ with $d^*a=d^*a'=0$.
 Suppose $A'= (\det \sigma)^*(A)$ could hold
  for some $\sigma =\exp(f) \in {\frak G}_{k+1}$.
 Then the equality
 \[
 a' = a +2df
 \]
 should hold. Applying $d^*$ on both sides, 
 we obtain $d^*df=0$. Then
 \[
 0= <d^*df,f>_{L^2} = ||df||^2_{L^2},
 \]
 which gives  $df=0$ and the equality $A=A'$.
 Moreover $f\equiv 0$ holds, 
 because $X$ is non-compact.

 This implies that the quotient map:
 \begin{align*}
 L^2_k(X;  \tilde{S}^+)   \oplus 
   L^2_k(X;     \Lambda^1&  \otimes    i {\mathbb R})  \cap \text{Ker } d^*  \hookrightarrow \\
& L^2_k(X;  \tilde{S}^+) \oplus 
   L^2_k(X;     \Lambda^1 \otimes    i {\mathbb R})  \ / \ 
 {\frak G}_{k+1}
 \end{align*}
 is injective.
 
 \vspace{2mm}
 
 {\bf Step 2:}
  Note that when one 
 finds such  a constant for some    $A_0$, then it holds for any choice of the base,
 since the SW moduli space is compact over the base compact manifold $M$
 (see Subsection \ref{2.3}).

 It follows from the assumption that there is a bounded linear map
 $$\Delta^{-1}: d^*(L^2_k(X;\Lambda^1 \otimes  i {\mathbb R})) \to 
 L^2_{k+1}(X; i {\mathbb R}) $$
 that  inverts the Laplacian.
 Let us set:
 $$s_0 = - \frac{1}{2} \Delta^{-1} (d^*(a)) \in L^2_{k+1}(X; i {\mathbb R})$$
 and $\sigma_0 = \exp(s_0) \in {\frak G}_{k+1}$.
 For $a' = a + 2\sigma_0^{-1} d \sigma_0$, we have:
 $$\det (\sigma_0)^*A = \tilde{A}_0 + a'$$
 with  the equality 
 $d^*(a')=0$.
 
 {\bf Step 3:}
 Let us consider
 $$d^* \oplus d^+: L^2_k(X;\Lambda^1 \otimes  i {\mathbb R})) \to 
 L^2_{k-1}(X;  \Lambda^0 \oplus \Lambda^2_+) \otimes i {\mathbb R} $$
 Its kernel  is the space of harmonic $1$-forms.
 We
 decompose $a' = h+b$, where $h =a'_{harm}$ is the harmonic form and 
 $b$ lies in the   orthogonal subspace.
 Then it follows from closedness  that there is a bound
 $$||b||_{L^2_k} \leq C(||d^*(b)||_{L^2_{k-1}} + ||d^+(b)||_{L^2_{k-1}}) = C||d^+(b)||_{L^2_{k-1}}.$$
 Moreover $d^+(b) = F^+_A - F^+_{\tilde{A}_0} =  d^+(a')$ holds.
 Thus,  we obtain
  $$||a'||^2_{L^2_k} \leq C ( ||d^+(a')||_{L^2_{k-1}} + ||a'_{harm}||).$$
 Since both the equalities
 $d^+(a')=d^+(a)$ and $a'_{harm}=a_{harm}$
 hold, this concludes the proof.
 \end{proof}


 \subsubsection{Covering SW moduli space}
 Let us consider the closed subset:
  \begin{align*}
 {\frak M} (A_0,  \psi_0)  & := 
  F_{\tilde{A}_0 , \tilde{\psi}_0}^{-1}(0)   \\
  & \subset  \ L^2_k(X;  \tilde{S}^+) \oplus 
  L^2_k(X;     \Lambda^1  \otimes    i {\mathbb R})  \cap \text{Ker } d^* .
  \end{align*}
  It is non empty since $[(0,0)]$ is an element in it.
 If $ F_{\tilde{A}_0 , \tilde{\psi}_0}$ has a regular  value at $0$
 such that its differential is surjective on ${\frak M} (A_0,  \psi_0)$,
 then it  is a regular manifold
 equipped with the induced $\Gamma$ action.
 Its $\Gamma$-dimension 
  is equal to:
   $$\text{ ind  } D_{A_0} - \chi_{AHS}$$
   where ind $D_{A_0}$ is the index of $D_{A_0}$ and $\chi_{AHS}$ is the  AHS-Euler characteristic on $M$.
   
 Note that if an element $g \in \Gamma$ is infinite cyclic, then the $g$-action is free except at the origin 
 $[(0,0)]$.

 Choose any $x ,y \in  F_{\tilde{A}_0 , \tilde{\psi}_0}^{-1}(0) $, and consider its differential:
 \begin{align*}
 d (F_{\tilde{A}_0 , \tilde{\psi}_0})_x: 
 L^2_k(X;  \tilde{S}^+) & \oplus 
  L^2_k(X;     \Lambda^1  \otimes    i {\mathbb R})  \cap \text{Ker } d^*  \\
&      \to   L^2_{k-1}((X,g);  \tilde{S}^-  \oplus \Lambda^2_+ \otimes  i {\mathbb R}) .
\end{align*}
We  denote $d (F_{\tilde{A}_0 , \tilde{\psi}_0})_x$   by $dF_x$ for brevity.

 \begin{lem}
 (1)  For $x= ( \psi,a)$, the following formula holds:
 $$dF_x(c, \xi) = (D_{\tilde{A}_0 +a}( \xi) +c(\tilde{\psi}_0 + \psi )  ,
 d^+(c) - (\tilde{\psi}_0+ \psi)\otimes   \xi^* -\xi\otimes (\tilde{\psi}_0+\psi)^* ).$$

 (2) Let $k \geq 3$.
 The difference 
 $dF_x - dF_y$
   is  compact.
 \end{lem}
 \begin{proof}
 We have:
 \begin{align*}
 dF_x & (c, \xi) = \frac{d}{dt}(D_{\tilde{A}_0 +a+ tc}(\tilde{\psi}_0 + \psi + t \xi) ,
 d^+(a+tc) - \sigma(\tilde{\psi}_0, \psi + t \xi))_{t=0}  \\
 & = (D_{\tilde{A}_0 +a}( \xi) +c(\tilde{\psi}_0 + \psi )  ,
 d^+(c) - (\tilde{\psi}_0+ \psi)\otimes   \xi^* -\xi\otimes (\tilde{\psi}_0+\psi)^* ).
   \end{align*}
 
 Let  $x=(\psi, a)$ and $y =(b, \phi)$.
 Their difference is given by:
  $$(dF_x - dF_y) (c,\xi)= (a-b)\xi +c(\psi - \phi) , - (\psi-\phi) \otimes \xi^*  - \xi \otimes  (\psi-\phi)^* )$$
  Since all $a,b, c, \phi, \psi,  \xi \in L^2_k$,  their products all 
  lie in $L^2_k$ by Corollary $3.3$.
    If $a,b, \psi,\phi$ all  have compact support, then 
  compactness follows from the Sobolev multiplication  with Rellich's Lemma.
  In general, they
  can be approximated by compactly supported smooth functions as the space of compact operators is a closed set in the space of bounded operators.
Hence, the difference is still compact.
  \end{proof}

 \vspace{3mm}

 \subsection{\Large \bf Covering monopole map}
 Let $M$ be a compact oriented four-manifold equipped with a spin$^c$ structure,
 and $X=\tilde{M}$ be its universal covering space with 
$\pi_1(M) = \Gamma$.
 Let
 $H^1_{(2)}(X) \ (\bar{H}^1_{(2)}(X))$ be the first  (reduced) $L^2$ cohomology group.
The  $L^2$ cohomology  groups coincide with each other,
 that is, $ \bar{H}^*_{(2)}(X)= H^*_{(2)}(X)$ when 
 the AHS complex is closed.

 Let $(A_0, \psi_0)$ be a solution to the SW equations over $M$, 
 and denote its lift by $(\tilde{A}_0, \tilde{\psi}_0)$ over $X $.

 \begin{defn}
  The covering monopole map at the base $(A_0, \psi_0)$ is the
   ${\frak G}_{k+1}  \rtimes \Gamma$ equivariant  map
  given by:
  \begin{align*}
\tilde{\mu}: &  \ 
 L^2_k((X,g);  \tilde{S}^+  \oplus  \Lambda^1 \otimes  i {\mathbb R})  \to  \\
  & \qquad \qquad 
    L^2_{k-1}((X,g);  \tilde{S}^-  \oplus  (\Lambda^2_+ \oplus \Lambda^0) \otimes  i {\mathbb R})
\oplus  \bar{H}^1_{(2)}(X)   & \\
& \quad ( \phi, a) \mapsto (  F_{\tilde{A}_0 , \tilde{\psi}_0},(\phi,a) , d^*(a), [a]) 
\end{align*}
where 
$[\quad ]$ is the orthogonal projection to the reduced cohomology group.
 \end{defn}
 
 \begin{rem}
 Even if the first de Rham cohomology group $H^1_{dR}(X; {\mathbb R})=0$ vanishes,
  the first reduced cohomology group may survive. 
 For an element in the latter cohomology,
  there associates a `gauge-group action'
 that can eliminate it. 
 Clearly such an action 
 does not lie in the $L^2$ Sobolev space.
 Its  behavior at infinity
 appears complicated such that they will `move' quite `slowly' at infinity.
 \end{rem}

 Suppose the AHS complex is closed. Then 
 the covering monopole map restricts on the slice:
  \begin{align*}
\tilde{\mu}: & \  L^2_k((X,g);  \tilde{S}^+) 
\oplus  L^2_k((X,g); \Lambda^1 \otimes  i {\mathbb R})  \cap \text{ Ker } d^*
\to  \\
  & \qquad \qquad 
   L^2_{k-1}((X,g);  \tilde{S}^-  \oplus  \Lambda^2_+ \otimes  i {\mathbb R})
\oplus  H^1_{(2)}(X)   & \\
& \quad ( \phi, a) \mapsto (  F_{\tilde{A}_0 , \tilde{\psi}_0}(\phi,a) ,  [a]).
\end{align*}
which is  a $\Gamma$-equivariant map
(see Proposition \ref{prop2.6}).

 \begin{lem}
The $\Gamma$-index of the linearized map
is given by:
\begin{align*}
\dim_{\Gamma}  \ d \tilde{\mu} & 
= \text{ ind } D - (b_0 (M) - b_1(M) + b_2^+(M) ) - \dim_{\Gamma} H^1_{(2)}(X) \\
& =  \text{ ind } D  - \dim_{\Gamma} H^+_{(2)}(X)
\end{align*}
where ind $D$ is the index of the Dirac operator over $M$.
\end{lem}
  \begin{proof}
  This follows from 
  Atiyah's $\Gamma$-index Theorem.
\end{proof}

\begin{rem}\label{rem2.10}
The $\Gamma$-dimension is a topological invariant of the base manifold $M$ when
   one of $H^1_{(2)}(X)=0$ or $H^+_{(2)}(X)=0$ holds.

If $M$ is compact and aspherical, then the
Singer conjecture 
states that the $L^2$ cohomology should vanish except for the middle dimension,
where in our case of  four-manifolds, only 
 the second $L^2$ cohomology is able to survive and  $H^1_{(2)}(X)$  should vanish.
This result has been verified for many classes of compact aspherical manfolds
whose fundamental groups have `hyperbolic' structure \cite{gromov1}.
\end{rem}

\section{$L^p$ analysis and   estimates on Sobolev spaces}
\label{sec.3}
\subsection{Sobolev spaces over covering spaces}
Let $E',E$ be vector bundles over a compact  Riemannian manifold $M$, and
$l: C^{\infty}(M;E') \to C^{\infty} (M;E)$ be a first-order  elliptic differential operator.

Let us  lift them over the universal  covering space $X = \tilde{M}$, and introduce 
the  lift of the $L^2$ inner product 
 $$< u, v> = \int_X (u(x), v(x)) vol $$  
 over $X$,
which is $\Gamma$-invariant.
 Let
 $l^*$ be the  formal adjoint operator over $X$.
We will use the Sobolev norms on sections of $E' \to X$  by:
\begin{align*}
& <u,v>_{L^2_1} = <u,v> + <l(u), l(v)>, \\
& <u,v>_{L^2_2} = <u,v> + <l(u), l(v)> + <l^*l(u), l^*l(v)>,  \\
& \dots
\end{align*}
whose spaces are given by taking the closure of   $C^{\infty}_c(X;E')$.
In other words, the inner products 
can be written as:
$$ <u,v>_{L^2_k} 
= \sum_{j=0}^k \ <(l^*l)^j (u),v>.$$

Similarly on $E \to X$, we equip with the Sobolev norms by use of:
\begin{align*}
& <w,x>_{L^2_1} = <w,x> + <l^*(w), l^*(x)>, \\
& <w,x>_{L^2_2} = <w,x> + <l^*(w), l^*(x)> + <ll^*(w), ll^*(x)>,  \\
& \dots
\end{align*}

\begin{lem}\label{self-adjoint}
$<l(v), w>_{L^2_k } = <v, l^*(w)>_{L^2_k}$ hold for all $k \geq 0$.
\end{lem}
\begin{proof}
It holds for $k =0$. 
Suppose it holds up to $k-1$.
Since the equalities: 
 \begin{align*}
 <l(u), w>_{L^2_k}  & =  <l(u),w> + <l^*l(u), l^*(w))>_{L^2_{k-1} } \\
&  =   <u,l^*(w)> + <l(u), l l^*(w))>_{L^2_{k-1} }  \\
& = <u,l^*(w)>_{L^2_k}
\end{align*}
 hold by induction,
the conclusion also holds for $k$.
\end{proof}

In  the case when $l: L^2_k (X) \cong L^2_{k-1}(X)$ gives a linear isomorphism, 
we can  replace the norms by:
$$<u,v>'_{L^2_k} = <(l^*l)^ku, v>_{L^2}.$$
Then $l : L^2_k(X) \cong  L^2_{k-1}(X)$ is unitary with respect to this particular norm.
These norms are equivalent:
$$C^{-1} || \quad ||_{L^2_k} \leq|| \quad ||'_{L^2_k}  \leq 
C || \quad ||_{L^2_k}$$
 to the above Sobolev norms
for some $C \geq 1$.
This follows from the fact that there is a positive $\delta >0$ with 
the bound:
$$ \delta ||u||_{L^2} \leq \ ||l(u)||_{L^2}, \quad
 \delta ||w||_{L^2} \leq  ||l^*(w)||_{L^2}$$
where $\delta $ is independent of $u$ and $w$.

\vspace{3mm}

For convenience we recall the local Sobolev estimates over four-dimensional manifolds. 
By local compactness we mean that it is compact
on $L^2_k(K)_0$ that is a restriction of the Sobolev spaces with support on $K$. Here $K$  is
a closure of an open and bounded subset in $X$. More precisely $L^2_k(K)_0$ 
is a Sobolev closure of $C_c^{\infty}(\text{int} K)$.

Hereinafter,
we  assume that a compact subset $K$  is a compact
smooth submanifold of codimension zero 
so it should have a smooth
boundary in $X$ if $X$ is non-compact.

\begin{lem}\label{sob-emb}
(1) The continuous embeddings
$L^p_k \subset L^q_l$
hold locally, if both  $k \geq l$ and $k - \frac{4}{p} \geq l - \frac{4}{q}$ hold.
They are also 
locally compact, if   the stronger inequalities
  $k > l$ and $k - \frac{4}{p} > l - \frac{4}{q}$ hold.

(2) The continuous  embeddings
$L^p_k \subset C^l$
hold locally if $k - \frac{4}{p}  > l$ hold.
\end{lem}
In particular, it is  convenient for us to check the embeddings
$ L^2_k \subset L^4_{k-1}  $.

\begin{proof}
We refer \cite{ebin}, \cite{eichhorn} for the proof.
We also refer \cite{gilberg and trudinger} for more detailed analysis of  Sobolev spaces.
\end{proof}

\begin{cor}\label{sob-mult}
The following local multiplications are continuous locally:

$(a)$ $L^2_k \times L^2_k \to L^2_k$ for $ k \geq 3$ and

$(b)$  $L^2_k \times L^2_k \to L^2_{k-1}$ for $k \geq 1$.
\end{cor}
\begin{proof}
Let us take $u,v \in L^2_k$. For $k' \leq k$,
$$\nabla^{k'}(uv) = \sum_{a+b=k'} \nabla^a(u) \nabla^b(v)$$
hold.
 If $0  \leq k' <k$, then the estimates:
\begin{align*}
||\nabla^a(u) \nabla^b(v)||_{L^2_{loc}} &
\leq ||\nabla^a(u)||_{L^4_{loc}}|| \|   \nabla^b(v)||_{L^4_{loc}} \\
& \leq C  ||\nabla^a(u)||_{(L^2_1)_{loc}}|| \|   \nabla^b(v)||_{(L^2_1)_{loc}}
\end{align*}
hold by Lemma \ref{sob-emb} $(1)$.
Thus,   we obtain:
$$||uv||_{(L^2_{k'})_{loc}} \leq C ||u||_{(L^2_{k})_{loc}}||v||_{(L^2_{k})_{loc}}.$$
This verifies $(b)$.

Let us verify $(a)$.
Suppose $3 \leq k'=k$. Then we obtain the estimate
\begin{align}\label{12}
||\nabla^k(u)v||_{(L^2)_{loc}} \leq C ||v||_{C^0}||u||_{(L^2_{k})_{loc}}
\end{align}
by Lemma \ref{sob-emb} $(1)$.
Combining this result with (\ref{12}), we have verified $(a)$.
\end{proof}

\vspace{3mm}

\subsection{$L^p$ cohomology}
Let $(X,g)$ be a complete Riemannian manifold.
For $p >1$, let $L^p_k (X; \Lambda^m)$ be the Banach space of 
$L^p_k$ differential $m$-forms on $X$, and
$d$ be the exterior differential whose domain is  $C^{\infty}_c(X; \Lambda^m)$.

Let us recall the following notions:

\vspace{2mm}

(1) The (unreduced) $L^p$ cohomology $H^{m,p}(X)_k$
is given by
\begin{align*}
 \text{ Ker } & \{d: L^p_k(X,\Lambda^m) \to L^p_{k-1}(X,\Lambda^{m+1}) \} \\
& \qquad  \qquad
\  / \  \text{ im } \{d : L^p_{k+1}(X,\Lambda^{m-1}) \to L^p_k(X,\Lambda^m) \}.
\end{align*}
 
 \vspace{2mm}

 (2) The reduced $L^p$ cohomology $\bar{H}^{m,p}(X)_k$ is given by
 \begin{align*}
 \text{ Ker } & \{ d: L^p_k(X,\Lambda^m) \to L^p_{k-1}(X,\Lambda^{m+1}) \} \\
 & \qquad  \qquad
 \  /  \ \overline{\text{im}  \ \{ d : L^p_{k+1}(X,\Lambda^{m-1}) \to L^p_k(X,\Lambda^m) \} }
 \end{align*}
 where $\overline{\text{im}}$ is the closure of the image.
 
 \vspace{2mm}

There is a canonical surjection
$H^{m,p}(X)_k \to \bar{H}^{m,p}(X)_k$,
and its kernel
$$T^{m,p}_k := \text{ Ker } \{ \ H^{m,p}(X)_k \to \bar{H}^{m,p}(X)_k \ \} $$
is called the torsion of $L^p$ cohomology. 
The differential
$d$ has closed range if and only if the torsion  $T^{m,p}_k=0$ vanishes.

\begin{defn} The space ${\frak H}^{m,p}(X)$
of 
$L^p$ harmonic $m$-forms 
   is given by
$$\text{ Ker }   \{ \ (d \oplus d^*) : L^p_k(X,\Lambda^m)  \to L^p_{k-1}(X,\Lambda^{m+1} \oplus \Lambda^{m-1})  \ \}.$$
\end{defn}
Note that ${\frak H}^{m,p}(X)$ is independent of the choice of $k$.
It  is well known that 
the  space $\bar{H}^{m,2}(X)_k$ is isomorphic to $L^2$ harmonic $m$-forms, which are 
 independent of $k $.

For our case of the AHS complex, 
 the second cohomology  involves $d^+$ rather than $d$. 
 Let $\dim X=4$.

 \begin{lem}\label{lem3.4}
 Suppose 
  $d: L^2_k(X; \Lambda^i) \to L^2_{k-1}(X; \Lambda^{i+1})$ have closed range for any $k \geq 1$ and 
  $i=0,1$.
 Then, the composition with the self-dual projection:
 $$d^+: L^2_k(X: \Lambda^1) \to L^2_{k-1}(X: \Lambda^2_+)$$
 also has closed range for any $k \geq 1$.
 \end{lem}
 \begin{proof}
 {\bf Step 1:}
 Let $H$ be a Hilbert space and 
 $W \subset H$ be a closed linear subspace.
 If a sequence $w_i \in W$ weakly converges to some $w \in H$, then $w \in W$.
 In fact $<w, h> = \lim_i <w_i, h> =0$ for any $h \in W^{\perp}$.

 Let $H_1$ and $H_2$ be both Hilbert spaces, and  
 $W \subset H_1 \oplus H_2$ be a closed linear subspace.
 Let us consider the projection 
 $P:H_1 \oplus H_2 \to H_1$, and take
 a sequence $w_i =v_i^1 +v_i^2 \in W \subset H_1 \oplus H_2$.
Suppose the sequence $P(w_i) = v_i^1 \in H_1$ converges 
to some $v_1 \in H_1$,  and  the weak limit of $w_i $ does not lie on $W$.
Then $||v_i^2|| \to \infty$ must hold.
In fact if $||v_i^2||$ could be bounded, then $v_i^2$ weakly converge 
to some $v_2 \in H_2$. In particuar $w_i$ weakly converges to 
$v_1 +v_2 $ which should lie in $W$ as we have verified.

 {\bf Step 2:}
 Let us verify the conclusion for $k=1$. It follows from Stokes' theorem that,
 for $\alpha \in L^2_1(X; \Lambda^1)$:
 $$ 0 = \int_X d(\alpha) \wedge d(\alpha) = ||d^+(\alpha)||_{L^2}^2 - ||d^-(\alpha)||^2_{L^2}$$
 Thus, we have the equality 
 $ ||d^+(\alpha)||_{L^2}^2 = ||d^-(\alpha)||^2_{L^2}$.
  
 Suppose 
  a sequence $\alpha_i \in L^2_1(X; \Lambda^1)$ 
 converges as
  $d^+(\alpha_i) \to a_+ \in L^2(X; \Lambda^2_+)$.
 Then $\{ d(\alpha_i)\}_i$ is a bounded sequence
 in $L^2$,
 since  the equality
 $||d(\alpha_i)||^2_{L^2}=||d^+(\alpha_i)||^2_{L^2}+||d^-(\alpha_i)||^2_{L^2}$ holds.
 Hence, the bounded sequence has a weak limit
 $w$-$\lim_i \ d(\alpha_i) \to a \in L^2(X; \Lambda^2)$
such that the self-dual part of $a$ coincides with $a_+$.
 
 Let us apply Step $1$ to 
 $W:= \text{ im } d(L^2_1(X; \Lambda^1))$,
 $v^1_i : = d^+(\alpha_i) \in H_1:=L^2(X; \Lambda^+)$ and 
 $v^2_i : = d^-(\alpha_i) \in H_2:=L^2(X; \Lambda^-)$.
 Then $a =d(\alpha) \in W$ for some $\alpha \in L^2_1(X; \Lambda^1)$;
 otherwise, $d^-(\alpha_i)$ should diverge in the $L^2$ norm.
 In particular, 
 $a_+ =d^+(\alpha)$ holds, and so $d^+$ has closed range.

 {\bf Step 3:}
 Let us verify $k=2$ case, and assume $\alpha_i \in 
 L^2_2(X; \Lambda^1)$ satisfies the  
 convergence $d^+(\alpha_i) \to a \in L^2_1(X; \Lambda^2_+)$.
 Then there is some $\alpha \in L^2_1(X; \Lambda^1)$ with
 $d^+(\alpha) =a $ by Step $2$.
 
 We may assume $d^*(\alpha) =0$ since 
  $d: L^2_{k+1}(X) \to L^2_k(X; \Lambda^1)$ has closed range.
  Then  the elliptic estimate tells 
 $\alpha \in L^2_2(X; \Lambda^1)$ and 
 so the  $k=2$ case follows.
 
 We can proceed by induction such that the conclusion holds.
 \end{proof}
 

\subsection{Examples of zero-torsion $L^2$ cohomology}
There are several instances of zero-torsion $L^p$ cohomology.
See  [P] for the $p \ne 2$ case.

Let us consider the case  $p=2$. Using the 
 Hilbert space structure, 
there have been  many examples  with zero torsion discovered,
some of which we present below.

\subsubsection{K\"ahler hyperbolic manifolds}
Let $(M, \omega)$ be a compact K\"ahler manifold,
and assume that the lift of the K\"ahler 
form $\tilde{\omega}$
over the universal covering space $X$ 
represents zero in the second de Rham cohomology $H^2(X;{\mathbb R})$
such that it can be given as $\tilde{\omega} = d(\eta)$ for some
$\eta \in C^{\infty}(X; \Lambda^1)$. 
Note that $\eta$ cannot be $\Gamma = \pi_1(M)$ invariant,
since then $\omega$ would be an exact form on $M$.

\begin{lem}\cite{gromov1}
Suppose $||\eta||_{L^{\infty} (X)}< \infty$ is finite.
Then the $L^2$ de Rham differentials have closed range.

Moreover, $(M, \omega)$  satisfies the Singer conjecture.
\end{lem}
See Section \ref{1.3}  and also Remark \ref{rem2.10}
on the Singer conjecture.

\subsubsection{Zero torsion with positive scalar curvature}
Let us present 
 four-manifolds with positive scalar curvature whose
universal covering spaces have zero torsion.

\begin{lem}\label{lem3.6}
Let $X$  and $Y$ be complete Riemannian manifolds
of dimension $2$, where $X$ is non-compact and 
 $Y$ is compact.
Suppose the de Rham differentials have closed range
on $X$.

Then the AHS complex over $X \times Y$ also has closed range.
\end{lem}
The following argument is quite straightforward, and can be applied to more general cases.

\begin{proof}
{\bf Step 1:}
It follows from Lemma \ref{lem3.4} that 
 $d^+$ also has closed range if $d$ is the case on 
 $1$-forms.
Thus,  it is enough to verify closedness of $d$ on both $0$ and $1$-forms.

Note that $C^{\infty}_c(X \times Y) \subset L^2 (X \times Y)$
is dense and $L^2 (X \times Y) = L^2(X) \otimes L^2(Y)$ holds,
where the right-hand side is the Hilbert space tensor product.

Let $\{ \ f_{\lambda} \ \}_{\lambda}$ and 
$\{ \ u_{\lambda} \ \}_{\lambda}$ be the spectral decompositions of the Laplacian on the
$L^2$ forms of degrees $0$ and $1$
over $Y$ respectively, where 
both $f_{\lambda}$ and $u_{\lambda}$ have
 the eigenvalues $ \lambda^2$.
 
{\bf Step 2:}
 Let $\Delta = d^* \circ d$ be the Laplacian acting on 
 the space of $L^2$ functions on $X$.
 Then
 by the open mapping theorem, 
  its spectrum is contained in $[ \epsilon, \infty)$
 for some positive $\epsilon >0$, 
 because the de Rham differential has 
closed range over $X$, and 
$X$
admits no nonzero $L^2$ harmonic  function.
Note that $X$ is assumed to be non-compact.
This implies that there is 
a positive constant $C > 0$
such that the uniform lower bound
$ ||d g||_{L_j^2} \geq C    ||g||_{L^2_{j+1}}$ holds
for any $g$.
In fact, 
$ ||d g||^2_{L^2}  =<\Delta g, g>_{L^2}
\geq \epsilon     ||g||^2_{L^2}$ holds.
Hence, 
$||g||^2_{L^2_1} = ||g||^2_{L^2}+||dg||^2_{L^2}
\leq (1+\epsilon^{-1}) ||dg||^2_{L^2}$ holds.
Then, 
\[
  ||g||^2_{L^2_{j+1}} = 
  ||dg||^2_{L^2_{j}} +||g||^2_{L^2}
  \leq  ||dg||^2_{L^2_{j}} + \epsilon^{-1} ||dg||^2_{L^2}
  =(1+ \epsilon^{-1})  ||dg||^2_{L^2_{j}}
  \]
  holds.

{\bf Step 3:}
Let us consider the case of $0$-forms.
Suppose $\alpha = d (F) \in L^2_k(X \times Y; \Lambda^1)$ lies in the image
\[
d(L^2_{k+1}(X \times Y))  \subset (d(L_{k+1}^2(X)) \otimes L^2_{k+1}(Y) )\oplus 
(L^2_{k+1}(X) \otimes d(L_{k+1}^2(Y)))
\]
where the right-hand side is the Hilbert space tensor product, which is 
defined as both $d(L^2_{k+1}(X))$ and $d(L^2_{k+1}(Y))$ are closed in $L^2_k$.

Decompose  $F = \sum_{\lambda} \ g_{\lambda} \otimes f_{\lambda}$.
Note that
$||d f_{\lambda}||_{L^2}^2 = \lambda^2 ||f_{\lambda}||_{L^2}^2$.
Moreover there is a positive constant $C > 0$
such that the uniform lower bound
$ ||d g_{\lambda}||_{L_j^2} \geq C    ||g_{\lambda}||_{L^2_{j+1}}$ holds by Step $3$.

Then, 
we have the equalities
\begin{align*}
||\alpha||_{L^2_k}^2 & = || d \sum_{\lambda} g_{\lambda} \otimes f_{\lambda} ||_{L^2_{k}}^2 \\
&
 = \sum_{\lambda} \ 
 \sum_{j=0}^k \ (
 ||dg_{\lambda}||_{L^2_j}^2 ||f_{\lambda}||_{L^2_{k-j}}^2 
+ ||g_{\lambda}||_{L^2_j}^2 ||df_{\lambda}||_{L^2_{k-j}}^2 ) .
\end{align*}
 By using the spectral weights for the Sobolev spaces,
it is equal to 
\begin{align*}
 & \sum_{\lambda} \sum_{j=0}^k  \
 \lambda^{2(k-j)} \ (
 ||dg_{\lambda}||_{L^2_j}^2 ||f_{\lambda}||_{L^2}^2 
+ ||g_{\lambda}||_{L^2_j}^2 ||df_{\lambda}||_{L^2}^2 ) \\
& =  \sum_{\lambda}   \sum_{j=0}^k \ 
 \lambda^{2(k-j)} \ (
 ||dg_{\lambda}||_{L^2_j}^2 ||f_{\lambda}||_{L^2}^2 
+ \lambda^2  ||g_{\lambda}||_{L^2_j}^2 ||f_{\lambda}||_{L^2}^2 ) \\
& =  \sum_{\lambda} \ 
 \sum_{j=0}^k \ (
 ||dg_{\lambda}||_{L^2_j}^2 ||f_{\lambda}||_{L^2_{k-j}}^2 
+   ||g_{\lambda}||_{L^2_j}^2
 ||f_{\lambda}||_{L^2_{k-j+1}}^2 ) .
 \end{align*}
Then we have the estimates
\begin{align*}
& \geq 
C \sum_{\lambda} \ 
 \sum_{j=0}^k \ 
 ||g_{\lambda}||_{L^2_{j+1}}^2 ||f_{\lambda}||_{L^2_{k-j}}^2  
   + \sum_{\lambda \ne 0 } \ 
 \sum_{j=0}^k \ 
  ||g_{\lambda}||_{L^2_j}^2 ||f_{\lambda}||_{L^2_{k-j+1}}^2 \\
  & \geq C' \sum_{\lambda} \sum_{j=0}^{k+1} \ 
 ||g_{\lambda}||_{L^2_{j}}^2 ||f_{\lambda}||_{L^2_{k+1-j}}^2  
 = C'  ||\sum_{\lambda} g_{\lambda} \otimes f_{\lambda} ||_{L^2_{k+1}}^2
\end{align*}
for another positive constant $C' >0$.
This verifies that $d$ has closed range on $0$-forms over $X \times Y$.

{\bf Step 4:}
Let us verify a general fact.
Let $ d: H \to W = W_1 \oplus W_2$ be a linear map between Hilbert spaces,
and suppose that both  the compositions with the projections $d_i : H \to W_i$ have  
closed range for $i=1,2$.
Then we claim that  $d$ itself has closed range.

To see this, we replace  both $W_1$ and $W_2$ with the images $d_1(H)$ and $d_2(H)$
respectively, because the image of $f$ is contained in $d_1(H) \oplus d_2(H)$.
Hence,  $d_1$ and $d_2$ can be assumed to be surjective.

Let $V: = \text{ker} d_1 \cap \text{ker} d_2 \subset H$ be the intersection of their kernels.
Then, we can restrict on the orthogonal complement $V^{\perp} \subset H$.
Note that  there is a  positive constant $C>0$
such that any element:
\[
u =u_1+u_2 \in (V^{\perp} \cap \text{ker} d_1) \oplus  (V^{\perp} \cap \text{ker} d_2)
\]
admits a lower bound $||d(u)||\ \geq C||u||$
for some positive constant $C>0$,
since $d(u)=d(u_1)+d(u_2)$ with 
$||d_1(u_2)|| \geq C ||u_2||$ and $||d_2(u_1)|| \geq C||u_1||$.
Hence, we obtain the bound:
\[
||d(u)||^2  = ||d_1(u_2)||^2 +||d_2(u_1)||^2 \geq C (||u_2||^2+||u_1||^2) =C||u||^2.
\]
This implies that any element $u \in V^{\perp}$ also admits a lower bound 
$||d(u)||^2 \geq C||u||^2$
for some   positive constant $C>0$.
 This verifies that $d$ has closed range.

{\bf Step 5:}
Let us consider the case of $1$-forms.
To check closedness of the differential, 
we
may assume that
$u \in d(L^2_{k+1}(X \times Y))^{\perp}$
by using the inner product in Lemma \ref{self-adjoint}
for $l = d \oplus (d^+)^*$.

Let us take an element
$u \in L^2_{k}(X \times Y ; \Lambda^0_X \oplus \Lambda^1_Y)$. 
We decompose:
\begin{align*}
u &  =  \sum_{\lambda} g_{\lambda} \otimes u_{\lambda} 
 = \sum_{\lambda} g_{\lambda} \otimes ( \alpha_{\lambda}+ d^* \omega_{\lambda})
\end{align*}
where $\alpha_{\lambda} $ is a closed $1$-form.

\begin{sublem}
$\alpha_{\lambda}$ is a harmonic form.
\end{sublem}

\begin{proof}
In fact the inner product:
\begin{align*}
0 & = <u, d(g \otimes \alpha)>_{L^2_k}
= <d^*u, g \otimes \alpha>_{L^2_k}  \\
& = \sum_{\lambda} 
 <g_{\lambda} \otimes d^*(u_{\lambda}), g \otimes \alpha>_{L^2_k} 
 = \sum_{\lambda} 
 <g_{\lambda}, g>_{L^2_k} \cdot
 < d^*(u_{\lambda}),  \alpha>_{L^2_k}  \\
 & = 
  \sum_{\lambda} 
 <g_{\lambda}, g>_{L^2_k} \cdot
 < d^*(\alpha_{\lambda}),  \alpha>_{L^2_k}
\end{align*}
vanishes for any $g \otimes \alpha$.
Hence, $d^*(\alpha_{\lambda})=0$ vanishes.
\end{proof}

 Then:
\begin{align*}
du  & =  \sum_{\lambda} dg_{\lambda} \otimes u_{\lambda} +
\sum_{\lambda} g_{\lambda} \otimes du_{\lambda} \\
& \in L^2_{k-1}(X \times Y ; \Lambda^1_X \otimes \Lambda^1_Y) \oplus 
L^2_{k-1}(X \times Y ; \Lambda^0_X \otimes \Lambda^2_Y) \\
 & = \sum_{\lambda} dg_{\lambda} \otimes  \alpha_{\lambda} 
  + \sum_{\lambda} dg_{\lambda} \otimes  d^* \omega_{\lambda} +
\sum_{\lambda} g_{\lambda} \otimes dd^* \omega_{\lambda} .
\end{align*}

It is sufficient  to check  closedness of the differential  on each term above from Step $4$.

Let us consider the projection
 of the differential to the first term:
\begin{align*}
d^1 :  u = \sum_{\lambda}  & g_{\lambda} \otimes \alpha_{\lambda} \in 
L^2_{k}(X \times Y ; \Lambda^0_X \otimes \Lambda^1_Y)  \\
&  \to 
d^1 u:  = \sum_{\lambda} d g_{\lambda} \otimes \alpha_{\lambda} 
\in L^2_{k-1}(X \times Y ; \Lambda^1_X \otimes \Lambda^1_Y) 
\end{align*}
Let $ \mathcal H^1(Y) $ be the space of harmonic $1$-forms on $Y$. 
Note that 
the restriction
$d^1: L^2_k(X) \otimes \mathcal H^1(Y) 
\to L^2_{k-1}(X; \Lambda_X^1) \otimes  \mathcal H^1(Y)
$ has closed range, because
$ \mathcal H^1(Y) $ is finite dimensional and the de Rham differential on $X$ is assumed to have closed range.

{\bf Step 6:}
Let us verify  closedness of  the differential on the rest case:
\begin{align*}
d^2  & :   u =    \sum_{\lambda}  
g_{\lambda} \otimes d^* \omega_{\lambda}  
 \in 
L^2_{k}(X \times Y ; 
\Lambda^0_X \otimes \Lambda^1_Y)  \\
&  \to 
d^2 u: = \sum_{\lambda} 
(dg_{\lambda} \otimes d^* \omega_{\lambda} +
g_{\lambda} \otimes dd^* \omega _{\lambda}  ) \\
& \qquad
\qquad
\in L^2_{k-1}(X \times Y ; \Lambda^1_X \otimes \Lambda^1_Y  \oplus 
\Lambda^0_X \otimes \Lambda^2_Y )
\end{align*}
Note the equalities:
$$||d d^* \omega_{\lambda}||_{L^2}^2 = <dd^* \omega_{\lambda}, dd^* \omega_{\lambda}>
= <d^*dd^* \omega_{\lambda},  d^* \omega_{\lambda}> 
= \lambda^2 || d^* \omega_{\lambda}||_{L^2}^2.$$
Then we have the estimates:
\begin{align*}
& ||d u||_{L^2_{k-1}}^2 
 = \sum_{\lambda} \sum_{j=0}^{k-1}  \ 
 \lambda^{2(k-1-j)}\ 
  ||d g_{\lambda}||_{L^2_j}^2 ||d^* \omega_{\lambda}||_{L^2}^2  \\
&   + \sum_{\lambda} \sum_{j=0}^{k-1} \ 
 \lambda^{2(k-1-j)} \ 
  ||g_{\lambda}||_{L^2_j}^2 ||dd^* \omega_{\lambda}||_{L^2}^2  \\
  &
 = \sum_{\lambda} \sum_{j=0}^{k-1} \ 
 \lambda^{2(k-1-j)} \ 
  ||d g_{\lambda}||_{L^2_j}^2 ||d^* \omega_{\lambda}||_{L^2}^2   
   + \sum_{\lambda} \sum_{j=0}^{k-1} \ 
 \lambda^{2(k-j)} \ 
  ||g_{\lambda}||_{L^2_j}^2 ||d^* \omega_{\lambda}||_{L^2}^2  \\
& \geq 
 C \sum_{\lambda} \sum_{j=0}^{k-1} \ 
 \lambda^{2(k-1-j)}  \ 
  ||g_{\lambda}||_{L^2_{j+1} }^2 ||d^* \omega_{\lambda}||_{L^2}^2  
   + \sum_{\lambda} \sum_{j=0}^{k-1} \ 
 \lambda^{2(k-j)}\sum_{j=0}^{k-1} \ 
  ||g_{\lambda}||_{L^2_j}^2 ||d^* \omega_{\lambda}||_{L^2}^2  \\
  & \qquad \qquad  \geq C' ||u||_{L^2_k}^2
  \end{align*}
for some positive constants $C,C' >0$.

{\bf Step 7:}
Let us consider the case:
 $$u = 
  \sum_{\lambda}   v_{\lambda} \otimes f_{\lambda}\in 
  L^2_{k}(X \times Y ; \Lambda^1_X \otimes \Lambda^0_Y).$$
 Let us decompose $v_{\lambda} = \alpha_{\lambda} + d^* \omega_{\lambda}$
 with  $d \alpha_{\lambda}=0$.
 
  By a similar argument as Step $5$,
  $\alpha_{\lambda}$
  is a harmonic form.
  Moreover we may assume
  $\alpha_{0} =0$, since
  $f_0=1$ is the constant function and so
  $d(\alpha_0 \otimes f_0)=0$ holds.
  
  By the assumption, there is a positive constant $C >0$ such that
  the estimate
  $||dd^* \omega||_{L^2_j} \geq 
  C||d^* \omega||_{L^2_{j+1}}$
  holds.
  Then
  \begin{align*}
  ||du ||_{L^2_{k-1}}^2 & = || \sum_{\lambda} \ 
  dd^* \omega_{\lambda} \otimes f_{\lambda}+
  v_{\lambda} \otimes df_{\lambda}||_{L^2_{k-1}}^2 \\
  & =
   \sum_{\lambda} \ ||
  dd^* \omega_{\lambda} \otimes f_{\lambda}||_{L^2_{k-1}}^2+ ||
  v_{\lambda} \otimes df_{\lambda}||_{L^2_{k-1}}^2 \\
  & = 
  \sum_{\lambda} \sum_{j=0}^{k-1} \ ||
  dd^* \omega_{\lambda} ||_{L^2_{j}}^2 \cdot || f_{\lambda}||_{L^2_{k-1-j}}^2+ ||
  v_{\lambda} ||_{L^2_{j}}^2 \cdot
  || df_{\lambda}||_{L^2_{k-1-j}}^2 \\
  & \geq C
  \sum_{\lambda} \sum_{j=0}^{k-1} \
   ||
  d^* \omega_{\lambda} ||_{L^2_{j+1}}^2 \cdot || f_{\lambda}||_{L^2_{k-1-j}}^2+ 
   \sum_{\lambda \ne 0} \sum_{j=0}^{k-1} \
 \lambda^{2(k-j)} ||
  v_{\lambda} ||_{L^2_{j}}^2 \cdot
  || f_{\lambda}||_{L^2}^2 \\
  & \geq C  || \sum_{\lambda}
  \ v_{\lambda} \otimes f_{\lambda}||_{L^2_k}^2
  =  C||u||_{L^2_k}^2.
  \end{align*}

     {\bf Step 8:}
     Let us consider the final case as follows,
      which is a linear combinations of Steps $5,6$ and $7$:
 \begin{align*}
 u = 
  \sum_{\lambda}   \ &  v_{\lambda} \otimes f_{\lambda}+  g_{\lambda} \otimes u_{\lambda} \\
&  \in 
  L^2_{k}(X \times Y ; \Lambda^1_X \otimes \Lambda^0_Y \oplus
  \Lambda^0_X \otimes \Lambda^1_Y).
  \end{align*}
 Again, we obtain closedness of the differential by checking the property
 for each degree of the differential forms on $\Lambda^*(X) \otimes \Lambda^*(Y)$
 by Step $4$.
 The only remaining case to be checked   is closedness of the image in 
 $\Lambda^1(X) \otimes \Lambda^1(Y)$.

Let us consider  the differential:
\begin{align*}
d^1:
 u = 
 &  \sum_{\lambda}   \  ( \beta_{\lambda} + 
 d^* \omega_{\lambda} )
  \otimes f_{\lambda}+  g_{\lambda} \otimes ( \alpha_{\lambda}  + d^* \mu_{\lambda})
  \to \\
& 
 \sum_{\lambda}   \ 
 (  \beta_{\lambda} +d^* \omega_{\lambda} )\otimes df_{\lambda}+  dg_{\lambda} 
 \otimes (\alpha_{\lambda} + d^* \mu_{\lambda}) ,\\
 \end{align*}
where both $ \alpha_{\lambda}$ and $ \beta_{\lambda}$ 
are harmonic $1$-forms.
Then it follows from the equalities:
\begin{align*}
& <\beta_{\lambda}, d^* \omega_{\lambda}>_{L^2_j}=
<\beta_{\lambda}, dg_{\lambda}>_{L^2_j}=
<d^* \omega_{\lambda}, dg_{\lambda}>_{L^2_j}=0, \\
& <\alpha_{\lambda}, d^* \mu_{\lambda}>_{L^2_j}
=
<\alpha_{\lambda}, df_{\lambda}>_{L^2_j}
= <df_{\lambda}, d^* \mu_{\lambda}>_{L^2_j}=0
\end{align*}
that: 
\[
||d^1u||_{L^2_{k-1}}^2 =
\sum_{\lambda} \
||(\beta_{\lambda} +d^* \omega_{\lambda})\otimes df_{\lambda}||_{L^2_{k-1}}^2
+ ||dg_{\lambda} \otimes (\alpha_{\lambda}+d^* \mu_{\lambda})||_{L^2_{k-1}}^2
\]
holds.
The first term is bounded as:
\[
\sum_{\lambda} \
||(\beta_{\lambda} +d^* \omega_{\lambda})\otimes df_{\lambda}||_{L^2_{k-1}}^2
 \geq C 
 \sum_{\lambda} \
||(\beta_{\lambda} +d^* \omega_{\lambda})\otimes f_{\lambda}||_{L^2_{k}}^2
\]
for some positive constant $C>0$ by Step $7$.
The second term is bounded as:
\[
\sum_{\lambda} \
||dg_{\lambda} \otimes (\alpha_{\lambda}+d^* \mu_{\lambda})||_{L^2_{k-1}}^2
\geq C \sum_{\lambda} \
 ||g_{\lambda} \otimes (\alpha_{\lambda}+d^* \mu_{\lambda})||_{L^2_{k}}^2
\]
for some positive constant $C>0$ by Steps $5$ and $6$.

Since the equality:
\begin{align*}
& \sum_{\lambda} \
||(\beta_{\lambda} +d^* \omega_{\lambda})\otimes f_{\lambda}||_{L^2_{k}}^2
+
\sum_{\lambda} \
 ||g_{\lambda} \otimes (\alpha_{\lambda}+d^* \mu_{\lambda})||_{L^2_{k}}^2 \\
& =
\sum_{\lambda} \
||(\beta_{\lambda} +d^* \omega_{\lambda})\otimes f_{\lambda}
+
g_{\lambda} \otimes (\alpha_{\lambda}+d^* \mu_{\lambda})||_{L^2_{k}}^2 
 = ||u||_{L^2_k}^2
\end{align*}
holds,  this 
completes the proof  for all  cases.
\end{proof}

\vspace{3mm}

The universal covering space of $\Sigma_g$  for $g \geq 2$ 
 is the upper half plane ${\bf H}^2$ equipped with the hyperbolic metric. It is well known that the
differential $d$ on ${\bf H}^2$  has closed range on any degree
\cite{donnelly}. Then by Lemma \ref{lem3.6},
$\Sigma_g \times S^2$
satisfies  two conditions that the AHS complex over their universal overing space
 has closed range. Moreover 
 the Dirac operator is  invertible
 because
  $S^2$ admits a metric of positive scalar curvature.

Let us compute the $\Gamma$-index of the AHS complex over
${\bf H}^2 \times S^2$ with  $\Gamma= \pi_1(\Sigma_g)$ actions.
We denote by $H^*_{\Gamma}(\Sigma_g \times S^2)$ 
as the $L^2$ cohomology group $H^*_{(2)}({\bf H}^2 \times S^2)$
equipped with  $\Gamma= \pi_1(\Sigma_g)$ action.

\begin{lem} For any $g \geq 2$, the
 $L^2$ cohomology group
 $H^*_{\Gamma}(\Sigma_g \times S^2)$
is zero for $*=0,2$ and satisfies 
$\dim_{\Gamma}  H^1_{\Gamma}(\Sigma_g \times S^2)= 2g-2.$
\end{lem}
\begin{proof}
{\bf Step 1:}
Any $L^2$  harmonic function over  a complete 
non-compact manifold is zero,
and so $H^0_{\Gamma}(\Sigma_g)=0$. 
By the Hodge duality, any  $L^2$ harmonic $2$-forms on ${\bf H}^2$ are also zero.
It follows from the Atiyah's $\Gamma$-index theorem that
$$\dim_{\Gamma} \ H^1_{\Gamma}(\Sigma_g) = 2g-2.$$
Since $H^1_{\Gamma}(\Sigma_g) $ is isomorphic to the space of $L^2$ harmonic $1$-forms,
we have the estimate:
$$\dim_{\Gamma} \ H^1_{\Gamma}(\Sigma_g \times S^2) \geq 2g-2.$$

{\bf Step 2:}
We claim that the above estimate is actually equal.
Let $\alpha \in L^2({\bf H}^2 \times S^2; \Lambda^1)$ be an $L^2$ harmonic $1$-form, and 
decompose:
$$\alpha = \alpha_1  + \alpha_2$$
with respect to $\Lambda^1_{{\bf H}^2 \times S^2}
\cong \Lambda^1_{{\bf H}^2 }  \otimes \Lambda^0_{S^2}
\oplus \Lambda^0_{{\bf H}^2} \otimes \Lambda^1_{S^2}$.
Note that each component lies in $\alpha_i \in L^2_k({\bf H}^2 \times S^2; \Lambda^1)$ 
 for any $k \geq 0$. It follows from 
 $d(\alpha)= d_1(\alpha) + d_2(\alpha) =0$ that: 
\begin{align*}
& 0 =
d_1(\alpha_1)  \in 
L^2_k({\bf H}^2 \times S^2; \Lambda^2_{{\bf H}^2} \otimes \Lambda^0), \\
& 0 =  d_2(\alpha_2)
\in 
L^2_k({\bf H}^2 \times S^2; \Lambda^0_{{\bf H}^2} \otimes \Lambda^2),
\end{align*}
hold,
 where $d_1$ and $d_2$
  are the differentials with respect to 
  ${\bf H}^2$ and $S^2$-coordinates respectively.

Note the isomorphism:
\[
L^2({\bf H}^2 \times S^2; \Lambda^0_{{\bf H}^2} \otimes \Lambda^1_{S^2})
 \cong 
L^2({\bf H}^2) \otimes
L^2(S^2; \Lambda^1).
\]
Let $\{ f_{\lambda}\}_{\lambda}$ be the spectral decomposition of $L^2(S^2)$, as in the 
proof of Lemma \ref{lem3.6} with $Y=S^2$.
We can decompose as:
\[
\alpha_2= \sum_{\lambda} \ k_{\lambda} \otimes d_2 f_{\lambda}
\]
since
 $H^1(S^2)=0$ holds.
 
 Next we decompose $\alpha_1$ as 
 $\alpha_1= \sum_{\lambda} \ a_{\lambda} \otimes f_{\lambda}$, 
 where each $a_{\lambda} \in L^2({\bf H}^2; \Lambda^1)$.
 Since $d_1(\alpha_1)=0$ vanishes, we also have $d_1 a_{\lambda}=0$. 
 Thus,  we can write $a_{\lambda} = h_{\lambda} +d_1 g_{\lambda}$, where
 $h_{\lambda}$ is an $L^2$ harmonic $1$-form. Then, we have:
 \[
 \alpha_1 =  \sum_{\lambda} \ (h_{\lambda}  +d_1g_{\lambda} ) \otimes f_{\lambda}.
 \]
 It follows from 
 $d_2 \alpha_1+d_1 \alpha_2=0$ that:
  \begin{align}\label{lambda}
  \sum_{\lambda} \ 
\{ \ - h_{\lambda} - d_1 g_{\lambda} + d_1 k_{\lambda} \ \} \otimes d_2 f_{\lambda}=0.
\end{align}
 
 We claim that the equality:
\begin{align}\label{lambda2}
\{ \ - h_{\lambda} - d_1 g_{\lambda} + d_1 k_{\lambda} \ \} \otimes d_2 f_{\lambda}=0
\end{align}
holds. 
In fact, by applying $d_2^*$ on both sides of 
(\ref{lambda}), we obtain:
\begin{align*}
 0 & =
 \sum_{\lambda} \ 
\{ \ - h_{\lambda} - d_1 g_{\lambda} + d_1 k_{\lambda} \ \} \otimes d_2^* d_2 f_{\lambda} \\
& = 
\sum_{ \lambda} \ 
\{ \ - h_{\lambda} - d_1 g_{\lambda} + d_1 k_{\lambda} \ \} \otimes   \lambda \cdot f_{\lambda}.
\end{align*}
Since $\{ f_{\lambda}\}_{\lambda}$ consists of an
orthonormal basis of $L^2(S^2)$, 
we obtain the equality (\ref{lambda2}).

This implies $h_{\lambda}=0$ and $d_1(g_{\lambda} - k_{\lambda}) =0$ for $\lambda \ne 0$.
Hence we can assume $g_{\lambda} =k_{\lambda}$ in the expression of $\alpha_1$.

For the $\lambda=0$ case, $f_0$ is clearly constant. 
Hence $\alpha$ has the form:
\[
\alpha = 
(h+d_1g) \otimes 1 + \sum_{\lambda} \ d(g_{\lambda} \otimes f_{\lambda})
\]
where $h$ is an $L^2$ harmonic $1$-form.

Since $\alpha$ is harmonic, both the equalities:
\begin{align*}
&   d_1g=0, \\
& \sum_{\lambda} \ d(g_{\lambda} \otimes f_{\lambda}) = d \sum_{\lambda} g_{\lambda} \otimes f_{\lambda}=0
\end{align*}
should hold. This implies that any $L^2$ harmonic $1$-form on ${\bf H}^2 \times S^2$ can be given 
by tensoring an  $L^2$ harmonic $1$-form on ${\bf H}^2$ with a constant on $S^2$.
\end{proof}

\vspace{3mm}

\subsection{Some estimates over non-compact 
four-manifolds}
In order to apply  $L^p$ estimates over non-compact spaces,
let us induce some basic inequalities.
Let $M$ be a compact four-manifold and $X = \tilde{M}$ be the universal covering space with
$\Gamma = \pi_1(M)$.

\begin{lem}\label{lem3.8}
For $p \geq 2$, the global Sobolev embeddings hold:
$$L^p_{i+1}(X) \subset L^{2p}_i(X).$$
\end{lem}
\begin{proof}
Let $K \subset X$ be a fundamental domain. Then the local Sobolev estimate
gives the embedding $L^p_{i+1}(K) \subset L^{2p}_i(K)$ in Lemma $3.2(1)$.

Now, we take $a \in L^p_{i+1}(X)$. Then we have the estimate:
$$||a||_{L^p_{i+1}(\gamma(K))} \geq c ||a||_{L^{2p}_i(\gamma(K))}$$
where $c$ is independent of $\gamma \in \Gamma$.
Thus,  we have the estimates:
\begin{align*}
||a||_{L^{2p}_i(X)}^{2p} & = \Sigma_{\gamma \in \Gamma}
||a||_{L^{2p}_i(\gamma(K))}^{2p} \leq c \Sigma_{\gamma \in \Gamma} ||a||_{L^p_{i+1}(\gamma(K))}^{2p} \\
& \leq c (\Sigma_{\gamma \in \Gamma}  ||a||_{L^p_{i+1}(\gamma(K))}^p)^2 = c  ||a||_{L^p_{i+1}(X)}^{2p}
\end{align*}
See [Ei], Chapter $1$ Theorem $3.4$.
\end{proof}

\vspace{3mm}

\begin{cor} Let $p =2^l \geq 2$.

(1) The embeddings hold:
$${\frak H}^{m,p}(X)  \ \supset  \ {\frak H}^{m,2}(X)$$
between the
$L^p$ and $L^2$ harmonic $m$-forms.
\end{cor}
\begin{proof}
 Let $p =2^l$. It follows from Lemma $3.8$ that the embeddings hold:
 $$L^2_{i+l-1}(X) \subset L^{2^2}_{i+l-2}(X) \subset \dots \subset L^p_i(X) $$
Then the conclusion holds since  $L^p$ harmonic forms have finite Sobolev norms in all $L^p_k$.
\end{proof}

\begin{rem}
One may consider the converse embedding. 
So far,  there has not been significant 
development of analysis, 
even though it is a quite basic subject.
\end{rem}

\vspace{3mm}

\subsection{$L^p$ closedness}
We assume that the de Rham differential has closed
range on $L^2$ such that it admits the $L^2$ harmonic projection.

Let us take $a \in L^2_k(X; \Lambda^1)$ for some large $k \gg 1$.
It follows from Lemma \ref{lem3.8} that $a \in L^p_1(X; \Lambda^1) \cap L^{2p}_1(X; \Lambda^1)$
for 
$p =2^l$ with $ l  \leq k-1$.
Suppose the following conditions:
$$(1) \ ||a||_{L^p_1(X)} \leq C, \quad (2) \  ||a||_{L^p(K)} \geq \epsilon_0, 
\quad \text{ and } \quad (3) \ d^*(a)=0$$
 hold  for some constants $C$ and $\epsilon_0 >0$,
and a compact subset $K \subset X$.
Let us denote  the 
$L^2$ harmonic projection of $a$
 by $a_{harm} \in {\frak H}^{1,2}$ and
 its $L^2_1$ norm by $ ||a||_{harm}$.

 \begin{lem} \label{lem3.11}
 $(\alpha)$
Assume  the above three conditions  $(1),(2)$, and $(3)$.
 Then,  at least one of the following criteria holds:
 
 \begin{itemize}
 \item
  The  following estimates hold:
$
 ||a||_{L^{2p}_1(X)} \leq c  \  ( ||d^+a||_{L^{2p}(X)} + ||a||_{harm})
$
 for some $c>0$ independent of $a$, 
 or
 
 \item
 There is a sequence $\{a_i\}_i$ as above  that they weakly converges 
 in $L^p_1 \cap L^{2p}_1$
 to a non-zero element in ${\frak H}^{1,p} \cap {\frak H}^{1,2p}$, but not in 
 ${\frak H}^{1,2}$.
 \end{itemize}
 
 $(\beta)$
 Assume the above  condition  $(3)$.
 Then, the  estimate holds:
 $||a||_{L^{2p}_1(X)} \leq c \  ( \max \{ ||d^+a||_{L^{2p}(X)} , \ 
  ||a||_{L^p_1(X)}  \ \} +  ||a||_{harm}).$
 \end{lem}
 \begin{proof}
 Let us verify $(\alpha)$.
Assume that a family $\{a_i\}_i$ with the above conditions,
 satisfies the following property:
 $$||a_i||_{L^{2p}_1(X) } =
  \delta_i^{-1} (  ||d^+(a_i)||_{L^{2p}(X)}
 +  ||a||_{harm})$$
 for some $\delta_i \to 0$.

 Let us divide this situation  into two cases:
 
 (a) Suppose that
 $||a_i||_{L^{2p}_1(X)} $ are uniformly bounded. 
  Then,  $\{a_i\}_i$ weakly converges to some $a \in L^{2p}_1(X; \Lambda^1) \cap L^p_1(X; \Lambda^1)$
 by  condition $(1)$
 and the standard local elliptic estimate.
  Moreover,
 the equalities  $d^+(a)=0$,
   $a_{harm} =0$ and $d^*(a)=0$ hold
   by  condition $(3)$
   since 
 both the convergences
 $||d^+(a_i)||_{L^{2p}(X)} 
, ||a_i||_{harm} \to 0$ hold.
 Note that $L^p$ spaces are reflective for $1<p < \infty$.
 
The restriction $a_i|K$ strongly converges to $a|K$ in $L^p$.
 It follows from  condition $(2)$ above that $a$ is non zero and, hence, gives
 a non-trivial element in ${\frak H}^{1,p} \cap {\frak H}^{1,2p}$,  but not in
 ${\frak H}^{1,2}$  since the $L^2$ harmonic part of $a$ is zero.

 (b) Assume $||a_i||_{L^{2p}_1(X)} \to \infty$.
 Then the following estimates hold:
 \begin{align*}
 ||a_i||_{L^{2p}_1(X)} & \leq c(||d^+(a_i)||_{L^{2p}(X)} +||a_i||_{L^{2p}(X)}) \\
&  \leq c \{ \ \delta_i ||a_i||_{L^{2p}_1(X)}  + ||a_i||_{L^{2p}(X)} \ \},
\end{align*}
 where the first inequality comes from  the elliptic estimate.
 In particular the following inequality holds:
 $$||a_i||_{L^{2p}_1(X)} \leq c'  ||a_i||_{L^{2p}} $$
 It follows from Lemma \ref{lem3.8} that: 
 the estimate must hold:
 \begin{align*}
 ||a_i||_{L^{2p}_1(X)} &  \leq 
 c'  ||a_i||_{L^{2p}}
  \leq c''   ||a_i||_{L^p_1} .
    \end{align*}
  The left-hand side diverges while $||a_i||_{L^p_1}$ are uniformly bounded
  by  condition $(1)$.
  Therefore,  case $(b)$ does not happen.

Next, we consider $(\beta)$.
If the former conclusion of $(\alpha)$ holds, then we are done.
Otherwise we can take a decreasing sequence $\delta_i \to 0$ as in the above proof.
Then the same estimates as above give the inequality:
$$||a_i||_{L^{2p}_1(X)} \leq c \  
  ||a_i||_{L^p_1(X)}   .$$

 The conclusion is just a combination of these cases.
 \end{proof}

\begin{rem}
 For the purpose of our analysis of the covering monopole map in Section \ref{sec4},
any $p >2$ suffices, but the $p=2$ case is not sufficient.
Hereinafter,  we will use $\beta$ only.
 \end{rem}


\subsection{Multiplication  estimates}
Let $X =\tilde{M}$ be the universal covering space of a compact four-manifold $M$
with $\pi_1(M)=\Gamma$, and
$K \subset X$ be a fundamental domain.

\begin{lem} \label{lem3.13}
The multiplication
$$L^2_k(X) \otimes L^2_k(X) \to L^2_m(X)$$
is bounded for 
$m \leq  k $ with $ k \geq 3 $, or $m <k$ with $k \geq 1$.
\end{lem}
\begin{proof}
For the proof, we refer \cite{eichhorn} Chapter $1$, Theorem $3.12$.
We include the proof for convenience.

Let us  take $a,b \in L^2_k(X)$
with $a= \sum_{\gamma \in \Gamma} a_{\gamma}$ with $a_{\gamma} \in L^2_k({\gamma(K)})$.
By Corollary \ref{sob-mult}, the local Sobolev multilication gives the estimates:
$$ ||a_{\gamma}b_{\gamma}||_{L^2_m} \leq C ||a_{\gamma}||_{L^2_k} ||b_{\gamma}||_{L^2_k},$$
where $C$ is independent of $\gamma \in\pi_1(M)$.
Therefore:
\begin{align*}
||ab||_{L^2_m(X)}^2 & = \sum_{\gamma}  ||a_{\gamma}b_{\gamma}||_{L^2_m}^2 \\ 
& \leq 
 C \sum_{\gamma}  ||a_{\gamma}||_{L^2_k}^2 ||b_{\gamma}||_{L^2_k}^2 
 \leq C (\sum_{\gamma}  ||a_{\gamma}||_{L^2_k})^2(\sum_{\gamma}  ||b_{\gamma}||_{L^2_k})^2   \\
&=   ||a||_{L^2_k(X)}^2||b||_{L^2_k(X)}^2.
 \end{align*}
\end{proof}


\begin{lem}\label{lem3.14} 
Let $m >2$ and choose $0 < \epsilon <1$. 
Then  there is a
 constant $C =C_K$ independent of $\epsilon >0$
such that
if two elements  $a,b \in L^2_{2m}(X)$ with $||a||_{L^2_{2m}(X)}=||b||_{L^2_{2m}(X)}=1$ satisfy
 uniform  estimates:
$$||a||_{L^2_{2m} (\gamma K)} , ||b||_{L^2_{2m} (\gamma K)}
 < \epsilon$$
 for all $\gamma \in \Gamma$, 
then
the following estimate holds:
$$||ab||_{L^2_{2m} (X)} <C  \epsilon.$$
\end{lem}
\begin{proof}
It follows from the 
local Sobolev embedding $L^2_{2m} \hookrightarrow C^m$ in Lemma \ref{sob-emb} $(2)$ that 
the estimates hold for all $\gamma \in \Gamma$:
$$||a||_{C^m(\gamma K)}, ||b||_{C^m(\gamma K)} < C \epsilon.$$
This implies the global estimates
$||a||_{C^m(X)},  ||b||_{C^m(X)} < C \epsilon.$

Consider the absolute values of the derivatives:
$$|\nabla^{l}(ab)| \leq \Sigma_{s=0}^{l} \ |\nabla^s(a)| \ |\nabla^{l-s}(b)|$$
 for $l \leq 2m$,
where each component of the right-hand side
satisfies  the property that one of $s$ or $l-s$ 
is less than or equal to $m$.
Suppose $s \leq m$ holds. Then:
\begin{align*}
 |\nabla^s(a)|^2 \ |\nabla^{l-s}(b)|^2 &  \leq 
C^2 \epsilon^2 |\nabla^{l-s}(b)|^2 \\
&  \leq 
C^2 \epsilon^2 (|\nabla^{s}(a)|^2  +|\nabla^{l-s}(b)|^2 ).
\end{align*}
By same argument, 
we can obtain the same estimate when $l-s \leq m$ holds.
Therefore,
in any case,  the following estimate holds:
$$|\nabla^{l}(ab)|^2 \leq C' \epsilon^2 \Sigma_{s=0}^{l} ( |\nabla^s(a)|^2 + |\nabla^s(b)|^2).$$
Now, we obtain the estimate by integration:
$$||ab||_{L^2_{2m}(X)} \leq C \epsilon (||a||_{L^2_{2m}(X)} + ||b||_{L^2_{2m}(X)}).$$
\end{proof}


\begin{cor}\label{cor3.15}
There is a constant $C =C_K$ such that, 
for   two elements  $b \in L^2_m(X)$ and 
$a \in L^2_{2m}(X)$,
if  $a$  satisfies
 the uniformly small  estimate:
$$||a||_{L^2_{2m} (\gamma K)} 
 < \epsilon$$
  for all $\gamma \in \Gamma$, then
the estimate holds:
$$||ab||_{L^2_m(X)} <C  \epsilon ||b||_{L^2_m(X)}.$$
\end{cor}
\begin{proof}
Consider the absolute values of the derivatives:
$$|\nabla^l(ab)| \leq \Sigma_{s=0}^{l} \ |\nabla^s(a)| \ |\nabla^{l-s}(b)|$$
 for $l \leq m$. 
Then, the same argument as above gives  the estimate:
$$|\nabla^{l}(ab)|^2 \leq C' \epsilon^2 \Sigma_{s=0}^{l}  |\nabla^s(b)|^2.$$
Hence  we obtain:
$$||ab||_{L^2_{m}(X)} \leq C \epsilon  ||b||_{L^2_{m}(X)}.$$
\end{proof}


\begin{rem}
Let $K \subset X$ be a fundamental domain, and choose a finite set
$\bar{\gamma} := \{ \gamma_1, \dots, \gamma_m\} \subset \pi_1(M) =\Gamma$.
For a positive constant $ \epsilon >0$, let us set
$$H'( \epsilon, \bar{\gamma}) := \{ \ w \in L^2_k(X;E) 
=H': ||w||_{L^2_k( \gamma K) }< \epsilon , \ \gamma \notin \bar{\gamma}  \ \}.$$

(1) 
For any $r$ and the open ball $B_r \subset H'$ of
 radius $r$, there is some $m$
such that the embedding
$B_r \subset  \ H'(\epsilon, m) := \ \cup_{\bar{\gamma} \in \Gamma^m } \  H'( \epsilon, \bar{\gamma})$ holds.

\vspace{3mm}

(2)
By Lemma \ref{lem3.14}, there is a constant $C$ such that the covering SW map  restricts:
$$F: H'( \epsilon, \bar{\gamma})  \to H( C \epsilon, \bar{\gamma}) .$$
\end{rem}

 \vspace{3mm}

\subsection{Locality of linear operators}
Let 
$K \subset X$ be a compact subset.
Recall the local Sobolev space
$L^2_k(K)_0$ of Definition \ref{def1.1}.
Suppose $l: L^2_k(X) \to L^2_{k-1}(X)$ is 
a first-order differential operator, and consider
its restriction:
$$l: L^2_k(K)_0 \to L^2_{k-1}(K)_0$$
between the Sobolev spaces on a compact subset $K$.
Let us take an element $w \in  l(L^2_k(X))  \cap L^2_{k-1}(K)_0$,
and  ask  when $w$ lies in the image $l(L^2_k(K)_0)$. 
In general this is not always the case.
Later when we consider properness of the covering monopole map, 
we shall use projections to the Sobolev spaces on compact subsets.
Here let us  observe a general analytic property.

Let us introduce a $K$-{\em spill} $e(w) \in [0, \infty)$ by:
$$e(w;K) = \inf_{ v \in  L^2_k(X)} \  \{ \ ||v||_{L^2_k(K^c)} : l(v) =w \ \}.$$

Let $K_0 \subset \subset K_1 \subset \subset K_2 \subset \dots \subset X$ be  exhaustion.
\begin{lem}
Suppose $l: L^2_k(X) \to L^2_{k-1}(X)$ is injective with  closed range.
Choose $w_i \in L^2_{k-1}(K_i)_0 \cap l(L^2_k(X))$, 
which converge to $w  \in L^2_{k-1}(X)$.

Then the spills  go to zero as:
$$e(w_i;K_i) \to 0.$$
\end{lem}
\begin{proof}
By the assumption, there are $v_i \in L^2_k(X)$ that
 converge to $v \in L^2_k(X)$ with $l(v_i) =w_i$ and $l(v)=w$.
For small $\epsilon >0$, there is $i_0 $ such that 
$||w-w_i||_{L^2_{k-1}(X)} < \epsilon$ holds
for $i \geq i_0$.
Hence, 
  the estimates $||v-v_i||_{L^2_k(X)} \leq C \epsilon$
hold for a constant $C$.

Suppose there  exists $\delta >0$ with
$e(w_i; K_i) \geq \delta$. 
Then, we have:
$$||v||_{L^2_k(K_i^c)} \geq ||v_i||_{L^2_k(K_i^c)}  - ||v-v_i||_{L^2_k(K_i^c)} \geq \delta - C \epsilon >0 $$
for all sufficiently large $i $, which cannot happen,
 since the $L^2_k$-norm of $v$ is finite on $X$. Then, 
  its $L^2_k$-norms
on the complements of $K_i$ should go to zero, because $\{ K_i \}_i$ exhausts $X$.
\end{proof}

\vspace{3mm}

\subsection{More Sobolev estimates}
Here, we verify the  Sobolev estimate, 
which  improves the original version to the most general way. 
Note that the estimate will not be used in later sections.

\begin{lem}
Suppose 
$(1) \ k -\frac{4}{p} \geq l - \frac{4}{q}$ with $k \geq l$, and $(2) \ p \leq q$.
Then the embeddings 
$L^p_k(X) \subset L^q_l(X)$
hold over the universal covering space $X = \tilde{M}$
of a compact four-manifold.
\end{lem}
\begin{proof}
For the proof, we refer to \cite{eichhorn}  Chapter $1$, Theorem $3.4$.
We include the proof for convenience.

It follows from the assumption $(1)$ that the  local Sobolev embedding 
$(L^p_k)_{\text{loc}} \subset (L^q_l)_{\text{loc}}$ holds by Lemma \ref{sob-emb} $(1)$.

Take $a \in L^p_k(X)$.
Then we obtain:
\begin{align*}
(||a||_{L^p_k(X)})^q & =( \Sigma_{\gamma \in \Gamma} \ ||a||^p_{L^p_k(\gamma(K))})^{\frac{q}{p}} 
 \geq c  ( \Sigma_{\gamma \in \Gamma} \ ||a||^p_{L^q_l\gamma(K))})^{\frac{q}{p}}. \\
\end{align*}

We want to verify the inequality:
$$( \Sigma_{\gamma \in \Gamma} \ ||a||^p_{L^q_l\gamma(K))})^{\frac{q}{p}} \ \geq \
 \Sigma_{\gamma \in \Gamma} \ ||a||^q_{L^q_l\gamma(K))}$$

The following sublemma completes the proof of Lemma.

\begin{sublem}
Let $\{a_i\}_{i=0}^{\infty}$ 
 be a non negative  sequence.
Then 
the estimates:
$$\sum_i a_i \leq (\sum_i a_i^{t^{-1}})^{t}$$
hold for $ t \geq 1$.
 \end{sublem}

This elementary fact follows from the sub-additivity of the function 
$x \mapsto x^{t^{-1}}$.
\end{proof}

 \vspace{3mm}

\section{Properness of the monopole map}\label{sec4}
A  metrically proper map between Hilbert spaces 
is defined by 
the property that the pre-image of a bounded set is also bounded.
The method of constructing  a finite-dimensional approximation
requires, in addition,  the 
 property 
 that  the restriction on any bounded set is  proper.
 
 It is a characteristic of infinite dimensionality that there 
 exists  a metrically proper map which is not proper on each bounded set. For example, for an infinite dimensional Hilbert space $H$, the  distance function
   $d: H \to  \mathbb R  $ by $x \to ||x||$
  is metrically proper but is not proper on each bounded set,
because the restriction $d: D \to [0,1]$ on the unit disk
$D \subset H$ 
 is not proper.

 A map is called {\em strongly proper}, if it satisfies 
 these
  two properties
  (see Definition \ref{def1.1}).
Both properties are satisfied  for the monopole map over a compact four-manifold,
 because it is  Fredholm.

In our case, the base space is non-compact, and we will verify the locally strong properness 
under the assumption of closedness of the AHS complex. This also works for the construction 
of a finite-dimensional approximation method.

In section $4$, we 
 assume $k \geq 3$ whenever we write $L^2_k$.
Moreover  we continue to assume
 that a compact subset $K \subset X$ is a compact smooth sub-manifold of codimension zero,
 possibly with smooth boundary.
 
 Let $M$ be a compact oriented smooth four-manifold and $X = \tilde{M}$ be the 
 universal covering space with $\pi_1(M) = \Gamma$.
 
Let us fix a  spin$^c$ structure on $M$, and choose
a solution $(A_0, \psi_0)$ to the SW equations over $M$.
Note that the pair $(A_0, \psi_0)$ is smooth
  after gauge transform (see \cite{morgan}, $84$ page, Theorem $5.3.6$).
  Hereinafter, we always assume that 
  the base solution $(A_0, \psi_0)$ is smooth.
We take their lift $(\tilde{A}_0, \tilde{\psi}_0)$ over $X$ as a base point of the covering monopole map.

In this section, we verify the following property:
 \begin{thm}\label{thm4.1}
Suppose the 
 AHS complex has closed range  over $X $.

Then the covering monopole map is
 locally strongly proper in the sense that the map
 \begin{align*}
\tilde{\mu}:  L^2_k(K; \tilde{S}^+)_0  \  \oplus  \ & 
L^2_k(X; \Lambda^1 \otimes i {\mathbb R} )    \cap \text{ Ker } d^*  \\
&  \to 
 L^2_{k-1}(K;\tilde{S}^-)_0
 \oplus L^2_{k-1}(X; \Lambda^2_+\otimes i {\mathbb R} ) \oplus  H^1_{(2)}(X)
& \\
&(\psi, a) \mapsto (D_{\tilde{A}_0+a}\psi, d^+(a) - \sigma(\tilde{\psi}_0, \psi), [a])
\end{align*}
 is strongly proper 
for
 any  compact subset $K \subset X$.
\end{thm}

 If the linearized operator of the covering monopole map gives an isomorphism,
 then the AHS complex has closed range. Hence, in such a case, 
 the conclusion holds. 
 The stronger condition is required when the construction of  Clifford $C^*$-algebra
 is involved.
 
 \begin{rem}
 The statement involves mixture of spaces $X$ and its subset $K$.
 This is because one can controll the
  analytic  behaviour of 
 differential forms by assuming closedness of the differentials,
 however there are no way to cotroll it for the spinors.
 This is the reason why we have to content with the restriction
 of the compactly supported spinors.
 \end{rem}

\begin{proof}
{\bf Step 1:}
Let us consider the properness on the restriction on bounded sets.
The non-linear term is is given by:
$$c: (\psi, a) \mapsto (a \psi, \sigma(\psi)).$$
We claim that this is a  compact mapping such that 
it maps a bounded set into a relatively compact subset.

Let us take another compact subset  $K \subset \subset K'$ 
and let $\varphi : K' \to [0,1]$ be a smooth cut-off function
with $\varphi|K \equiv 1$ that
 vanishes near the boundary of $K'$.
Then, the multiplication by $\varphi$ satisfies the inclusion:
$$\varphi  \cdot \text{ Ker } d^* \subset L^2_k(K'; \Lambda^1 \otimes i {\mathbb R} )_0.  $$

Then, $c$ factors through the multiplication:
\begin{align*}
& ( \text{id}, \varphi ):   L^2_k(K; \tilde{S}^+)_0  \oplus  \text{ Ker } d^* \to 
L^2_k(K; \tilde{S}^+)_0  \oplus L^2_k(K'; \Lambda^1 \otimes i {\mathbb R} ) ,  \quad (*) \\
& c: L^2_k(K; \tilde{S}^+)_0  \oplus L^2_k (K'; \Lambda^1 \otimes i {\mathbb R} )_0  \\
& \qquad \qquad \qquad \qquad  \to
L^2_{k-1}(K;\tilde{S}^-)_0
 \oplus L^2_{k-1}(K'; \Lambda^2_+\otimes i {\mathbb R} )_0. \qquad (**)
 \end{align*}
 The former map $(*)$  is linear and bounded.
 The second map $(**)$  is compact as  it factors through
   the inclusion
 $L^2_k(K')_0 \hookrightarrow L^2_{k-1}(K')_0$
 by Lemma \ref{lem3.13}, and the last maps
  is compact.

 {\bf Step 2:}
 To confirm  properness of $\tilde{\mu}$ on the restriction on a bounded set, 
 it is sufficient to confirm the  metrical properness of its linear term,
 because the non-linear term is compact on the restriction of a bounded set by Step $1$.
 The linear term is given by:
 $$(\psi, a) \mapsto(D_{\tilde{A}_0}\psi, d^+(a) - \sigma'(\tilde{\psi}_0, \psi), [a]),$$
 where $\sigma'(\tilde{\psi}_0, \psi) = \sigma'(\tilde{\psi}_0 , \psi)- \sigma(\psi)$.
 
 It follows from the closedness of $d^+$ that $a \mapsto (d^+(a), [a])$ is injective and metrically proper
 (see the last paragraph of Step $2$ in the proof of Lemma \ref{lem4.7}).
 Since the term $\sigma'(\tilde{\psi}_0, \psi) $ does not involve $a$, 
 it is sufficient  to confirm the  properness of $D_{\tilde{A}_0}\psi$.
 Properness surely holds on any bounded sets over $K$
 as  $D_{\tilde{A}_0}$ is an elliptic operator.

{\bf Step 3:}
Let us consider metrically properness.
This follows from   the combination of Lemma \ref{lem4.7} and  Proposition \ref{prop4.9}
with Lemma \ref{lem4.10} which is a version of Lemma \ref{lem4.4}.
These are all verified later  in  this section.
\end{proof}
\vspace{3mm}

 We  verify  the globally strong  properness for a particular class, which is stronger  than 
 locally strong properness.
 \begin{prop} 
Suppose the AHS complex has  closed range over $X =\tilde{M}$ 
whose second $L^2$ cohomology  $H^+_{(2)}(X)=0$ vanishes. 

If the metric on $M$ has positive scalar curvature, 
then  the covering monopole map is   metrically  proper 
  and locally proper on each bounded set.
 \end{prop}
We present  examples
that satisfy the above conditions. 
Note that  $S^2 \times \Sigma_g$ are the cases
for all $g \geq 2$.

We have already seen the latter property above, 
and hence, need only to verify the
metrically properness.

\vspace{1mm}

Our strategy is  as follows.
Assume the AHS complex has closed range.
Then we verify the following:

\vspace{1mm}

$(\alpha)$  metrical properness for $k=1$
under the assumption of existence of a metric of  
 positive scalar curvature  (Lemma \ref{lem4.4}),
 
 \vspace{1mm}

$(\beta) $ local metric properness for $k=1$
 (Lemma \ref{lem4.10}) under 
 the   assumption
of (local) $L^{\infty}$ bound,
 
 \vspace{1mm}

$(\gamma)$ (local) metric properness for $k \geq 1$ 
under the  additional two assumptions 
of (local) $L^{\infty}$ bound and (locally) 
metrical properness for $k=1$
(Lemma \ref{lem4.7}), and

\vspace{1mm}

$(\delta)$ local $L^{\infty}$ bound
(Proposition \ref{prop4.9}).


\begin{rem}
Let $B \subset M$ be a small open ball. There exists a
Riemannian metric $g$ whose
 scalar curvature is positive except $B$
  \cite{kazhdan and warner}.
 One may assume that the lift $\tilde{B} \subset X$ 
 satisfies  $\gamma(\tilde{B}) \cap \tilde{B} =\phi$
  for all $\gamma \ne id \in \Gamma:= \pi_1(M)$.
 Let us set $\bar{B} := \cup_{\gamma \in \Gamma} \gamma(\tilde{B})$.
 Assume that $\tilde{\mu}$ could be metrically
 non-proper, and choose $\tilde{\mu}(x_i) =y_i$ such that $||y_i|| \leq c < \infty$ while
 $||x_i|| \to \infty$.
 Let $\varphi$ be a
 cut-off function  with $\varphi|\bar{B} \equiv 1$ and
 zero outside a small neighborhood of $\bar{B}$.
Both two families  $\tilde{\mu}| \{(1- \varphi) x_i\}_i$ and $\tilde{\mu}| \{ \varphi x_i\}_i$
 must be proper (see   Lemma  \ref{lem4.4} below), and
 so
 $\{x_i\}_i$ should be unbounded near the boundary of some $\gamma(\bar{B})$
 and $\gamma \in \Gamma$.
 \end{rem}

\subsection{Positive scalar curvature metric}
Let us verify the metric  properness of  the covering monopole map, 
in the case when  the base manifold $M$
admits a metric of positive scalar curvature and 
the AHS complex over $X = \tilde{M}$ 
is closed.

Let us fix a   reducible solution $(A_0, 0)$ 
to the SW equations over $M$
and choose the pair as the base point of the covering monopole map.

  \begin{lem} \label{lem4.4}
Suppose the AHS complex has closed range over $X =\tilde{M}$.
If $M$  admits a Riemannian metric of positive scalar curvature,
then 
the covering monopole map:
\begin{align*}
\tilde{\mu}: & \  L^2_1((X,g);  \tilde{S}^+
\oplus   \Lambda^1 \otimes  i {\mathbb R})  \cap \text{ Ker } d^*
\to  \\
  & \qquad \qquad 
   L^2((X,g);  \tilde{S}^-  \oplus  \Lambda^2_+ \otimes  i {\mathbb R})
\oplus  H^1_{(2)}(X)   & \\
& \quad ( \phi, a) \mapsto (  F_{\tilde{A}_0 , 0}(\phi,a) ,  [a])
\end{align*}
 is  metrically proper, where $H^1_{(2)}(X)$ is the   first $L^2$ cohomology group.
\end{lem}
\begin{proof} 
Let us set $\tilde{\mu}(\phi, a) = (\varphi, b, h)$, and
 denote $A = \tilde{A}_0+a$. 

 {\bf Step 1:}
 We have the pointwise equalities:
\begin{align*}
(F_A \phi , \phi) & = (F_A^+ \phi , \phi)= ((F^+_A - \sigma(\phi) ) \phi  + \sigma( \phi)\phi , \phi) \\
& = (b \phi, \phi)   +  \frac{|\phi|^4}{2}
\end{align*}
since
$F_A \phi= F_A^+ \phi$ holds.

Suppose $M$ admits a Riemannian metric of positive scalar curvature.
Then,  from the Weitzenb\"ock formula:
$$D_A^2 (\phi) = \nabla^*_A \nabla_A(\phi) + \frac{\kappa}{4} \phi + \frac{F_A}{2} \phi,$$
it follows that the estimates: 
$$||D_A(\phi)||^2_{L^2}  + ||b||_{L^2}||\phi||_{L^4}^2 
\geq  \delta ||\phi||^2_{L^2} + \frac{1}{4}||\phi||^4_{L^4}
 \geq \frac{1}{4}||\phi||^4_{L^4} $$
hold  for some positive $\delta >0$.
 In particular, there is $c =c(||\varphi||_{L^2},  ||b||_{L^2})$
 such that the bound:
 $$||\phi||_{L^4} \leq c$$
  holds. 
Using  another estimate:
  $$||\varphi||^2_{L^2}  + ||b||_{L^2}||\phi||_{L^4}^2 
\geq  \delta ||\phi||^2_{L^2} + \frac{1}{4}||\phi||^4_{L^4}
 \geq  \delta ||\phi||^2_{L^2} $$
  we obtain the $L^2$ estimate:
$$||\phi||_{L^2}  \leq \ c'(||\varphi||_{L^2},  ||b||_{L^2}, \delta).$$

{\bf Step 2:}
 From the equality $d^+(a) = b+ \sigma(\phi)$,
 it follows that:
 \begin{align*}
  ||a||_{L^2_1} & \leq c(||d^+(a)||_{L^2} + ||a_{harm}||_{L^2}) \\
 & \leq c(||b||_{L^2} + ||\phi||^2_{L^4}  +||h||_{L^2}).
 \end{align*}
 Combinig this with step $1$, we obtain:
 $$||a||_{L^2_1} \leq c(||\varphi||_{L^2},  ||b||_{L^2}, ||h||_{L^2}, \delta).$$

 {\bf Step 3:}
 It follows
 from the embedding $L^2_1 \subset L^4$ in Lemma \ref{lem3.8}
  that:
 $$||a \phi||_{L^2} \leq ||a||_{L^4}||\phi||_{L^4} 
 \leq 
 c(||\varphi||_{L^2},  ||b||_{L^2}, ||h||_{L^2}, \delta).$$
 Then,
  $$||D_{\tilde{A}_0} (\phi)||_{L^2} \leq ||D_A(\phi)||_{L^2} + ||a \phi||_{L^2} < \infty.$$
  It follows from the elliptic estimate that 
    in $L^2_1$, we have the bound:
   $$||\phi||_{L^2_1} \leq c(||\varphi||_{L^2},  ||b||_{L^2}, ||h||_{L^2}, \delta).$$           
    \end{proof}
 
 
 \begin{rem}
To induce higher Sobolev estimates, it is sufficient
 to obtain 
 the estimate of $||a\phi||_{L^2_1}$.
  Certainly we can obtain the  estimate
  $$||a\phi||_{L^2_k} \leq C_k ||a||_{L^2_k}||\phi||_{L^2_k} $$
  for $k \geq 3$, but it is not applied at $k=1$.
  This forces us to use $L^{\infty}$ estimates later,
  which leads to $L^p$  analysis.
  \end{rem}
 
\begin{exm}
Immediate examples of closed four-manifolds with positive scalar curvature metrics
will be $S^4$ or $\Sigma_g \times S^2$ with their metrics $h + \epsilon g$
with small $\epsilon >0$, where $h$ and $g$ are both
the standard metrics.
For the latter case the AHS complex over their universal covering spaces 
has closed range by Lemma \ref{lem3.6}.
\end{exm}

 \subsection{Regularity under $L^{\infty}$ bounds}
 Let us take  a solution $(A_0,\psi_0)$
  to the SW equations
 over $M$.
 One may assume that the  solution  is smooth.
 We  consider the covering monopole map
 with their lift $(\tilde{A}_0, \tilde{\psi}_0)$ 
 as the base point.

 Assume  the AHS complex has closed range,  and consider 
 the monopole map:
 \begin{align*}
\tilde{\mu}: & \  L^2_k((X,g);  \tilde{S}^+
\oplus   \Lambda^1 \otimes  i {\mathbb R})  \cap \text{ Ker } d^*
\to  \\
  & \qquad \qquad 
   L^2_{k-1}((X,g);  \tilde{S}^-  \oplus  \Lambda^2_+ \otimes  i {\mathbb R})
\oplus  H^1(X)   & \\
& \quad ( \phi, a) \mapsto (  F_{\tilde{A}_0 , \tilde{\psi}_0}(\phi,a) ,  [a]).
\end{align*}
It follows from  Lemma \ref{lem4.7}, Proposition \ref{prop4.9} and Lemma \ref{lem4.10} below
that 
$
\tilde{\mu}$
 is metrically proper,
  under the conditions of $k \geq 3$ and
  closedness
 of the AHS complex.

 \begin{lem}\label{lem4.7}
Suppose that for any positive numbers 
$s,r >0$, there is $c =c(s,r, K, A_0, \psi_0) >0$ such that 
for any $k \geq 3$ and an element 
$(\phi,a) \in L^2_k((X,g);  \tilde{S}^+
\oplus   \Lambda^1 \otimes  i {\mathbb R})  \cap \text{ Ker } d^*$ with
 $r := ||\tilde{\mu}(a, \phi)||_{L^2_{k-1}}$,
  the following conditions  hold:
 \[
 s: =||(\phi,a)||_{L^{\infty}} < \infty, \quad
 ||(\phi,a)||_{L^{2}_1} \leq  c .
 \]

Then,  there is a positive constant 
$c_k =c_k(s,r, K, A_0, \psi_0)$
such that 
the estimate
$||(\phi,a)||_{L^2_k} \leq c_k$ holds.
\end{lem}

  \begin{proof}
 {\bf Step 1:}
 Let us  set
$\tilde{\mu}(\phi,a) = (\varphi, b, h)$. 
 We check  that the multiplication
 $a \cdot \phi $
 is in $L^2_1(X)$ and bounded by  $3sc$.
 First, 
  $a \cdot \phi \in  L^2(X)$ holds by the $L^{\infty}$ bound:
 $$||a \cdot \phi||_{L^2} \leq ||a||_{L^{\infty}} ||\phi||_{L^2} \leq 
 s c.$$
 Note the equality 
 $\nabla(a \cdot
 \phi) = \nabla(a)\phi + a \nabla(\phi)$. 
 Then,
 $$||\nabla(a\phi)||_{L^2} \leq 
 ||\phi||_{L^{\infty}} ||\nabla(a)||_{L^2} + ||a||_{L^{\infty}} ||\nabla(\phi)||_{L^2}
 \leq 2sc.$$
Hence, we have 
$||a \cdot \phi||_{L^2_1} \leq 3sc $.
  Since $\psi_0$ is smooth over $M$ which is compact,
  we have 
   $||a \cdot \tilde{\psi}_0||_{L^2_1} \leq C(\psi_0)||a||_{L^2_1} \leq C(\psi_0) c$.
   Then, 
    $D_{\tilde{A}_0}(\phi)  \in L^2_1(X)$ holds,
because
the left-hand side of
 $$\varphi = D_A(\phi) + a \cdot \tilde{\psi}_0 =
 D_{\tilde{A}_0}(\phi) + a \cdot \phi + a \cdot \tilde{\psi}_0$$
 has  an $L^2_{k-1}(X)$ norm less than $r$.
  Hence, the bound:
\[
||\phi ||_{L^2_2} \leq 
C(r + C(\psi_0)c + 3sc +c)=: c_2'
\]
 holds by the elliptic estimate
$||\phi||_{L^2_2} \leq C(||D_{\tilde{A}_0}(\phi) ||_{L^2_1}+||\phi||_{L^2_1})$.

{\bf Step 2:}
Then, $\phi \in L^4_1(X)$ holds,
because the embedding $L^2_2 (X) \subset L^4_1(X)$ holds by Lemma \ref{lem3.8}.
Let us denote $\sigma(\tilde{\psi}_0,  \phi) = :\sigma(\phi) +  l(\tilde{\psi}_0,  \phi)$
(see Lemma \ref{lem2.2}).
We obtain  the estimates:
\begin{align*}
& ||\sigma(\phi)||_{L^2} \leq
 ||\phi||^2_{L^4} \leq ||\phi||^2_{L^2_1}  \leq c^2 \  \text{ and } \\
 &  ||l (\tilde{\psi}_0,  \phi)||_{L^2} \leq C(\psi_0)
||\phi||_{L^2}  \leq C(\psi_0)c.
\end{align*}

For the derivatives, we have
\begin{align*}
& ||\nabla \sigma(\phi)||_{L^2} \leq
 ||\nabla(\phi)||_{L^4} ||\phi||_{L^4} \leq ||\phi||_{L^4_1} ||\phi||_{L^4}  \leq 
 C ||\phi||_{L^2_2} ||\phi||_{L^2_1}
 \leq 
 Cc_2'c  \\
& ||\nabla l(\tilde{\psi}_0,  \phi)||_{L^2} \leq C(\psi_0)
(||\phi||_{L^2} + ||\nabla \phi||_{L^2}) \leq C(\psi_0)
||\phi||_{L^2_1} \leq C(\psi_0)c.
\end{align*}
Hence, we have
$||\sigma(\tilde{\psi}_0,  \phi)||_{L^2_1} \leq C(\psi_0)cc_2'$.
Then, the estimate
\[
||d^+(a) ||_{L^2_1(X)} \leq  r+C(\psi_0)cc_2'
\]
 holds, because
$b = d^+(a) - \sigma(\tilde{\psi}_0, \phi) $ has $ L^2_{k-1}(X)$ norm that is  less than $r$.
Because the AHS complex has closed range, 
it follows from the open mapping theorem that
there is a positive constant $C >0$
such that the bound:
\[
||a||_{L^2_2} \leq C(  ||d^+(a)||_{L^2_1}
+ ||[a]||_{L^2})
\]
holds for any $a \in L^2_2(X) \cap \text{Ker } d^*$.
Thus,  the estimate holds:
\[
||a ||_{ L^2_2} \leq C( r+C(\psi_0)cc_2' + r) =: c_2.
\]

{\bf Step 3:}
Let us verify  
 $a \cdot \phi \in L^2_2(X)$.
  Note the equality: 
 $$\nabla^2(a \cdot \phi) = \nabla^2(a) \cdot \phi + 2\nabla(a) \cdot \nabla(\phi) 
 + a \cdot \nabla^2(\phi).$$
The  $L^2$ norms of  the first 
and last terms on the right-hand side
are both bounded 
as $||\nabla^2(a) \cdot \phi ||_{L^2},  
|| a \cdot \nabla^2(\phi)||_{L^2} \leq s c_2$.
 For the middle term, we have the estimates:
 \begin{align*}
 ||\nabla(a) \cdot \nabla(\phi)||_{L^2} & \leq ||\nabla(a)||_{L^4}||\nabla(\phi)||_{L^4}
 \leq C||\nabla(a)||_{L^2_1} ||\nabla(\phi)||_{L^2_1} \\
 & \leq C||a||_{L^2_2}||\phi||_{L^2_2} \leq Cc_2^2,
 \end{align*}
 where we used  Lemma \ref{lem3.8}.
Hence, we have
 \[
 ||a \cdot \phi ||_{L^2_2} 
\leq   ||a \cdot \phi ||_{L^2_1}  + ||\nabla^2(a \cdot \phi)||_{L^2}
\leq 3sc+2sc_2+ Cc_2^2 .
 \]
The remainder of the argument is parallel to Step $1$.
We have 
  $
  ||D_{\tilde{A}_0}(\phi)||_{L^2_2}
  \leq r+ C_2(\psi_0) c_2+3sc+2sc_2+ Cc_2^2 .
  $
  Then, the estimate holds:
  \begin{align*}
  ||\phi||_{L^2_3} & \leq C_3( ||D_{\tilde{A}_0}(\phi)||_{L^2_2}
  + ||\phi||_{L^2_2}) \\
&  \leq
  C_3( r+ C_2(\psi_0) c_2+3sc+2sc_2+ Cc_2^2 + c_2') = : c_3'.
  \end{align*}
   By applying similar estimates in Step $2$,
we have 
 $\sigma(\tilde{\psi}_0, \phi) \in L^2_2(X)$,
and then,  we obtain 
   $d^+(a) \in L^2_2(X)$. Then, 
we have $||a ||_{L^2_3(X)} \leq c_3$.

{\bf Step 4:}
We have verified $a, \phi \in L^2_3(X)$. 
Now we can use a simpler argument that 
the multiplication:
$$L^2_3(X) \times L^2_3(X) \to L^2_3(X)$$
is continuous in Lemma \ref{lem3.13} such that we can see the inclusion
$a \cdot \phi \in L^2_3(X)$
 immediately by the estimates:
$$ ||a \cdot \phi||^2_{L^2_3(X)}   \leq   C
||a||_{L^2_3(X)}^2||\phi||_{L^2_3(X)}^2 \leq c_3^2.$$
Then, we repeat the latter part of Step $1$, 
and Step $2$. Then, we obtain the inclusions
$\phi \in L^2_4(X)$ and
 $a \in L^2_4(X)$.

The remainder of the argument is parallel, 
and we obtain  the 
$L^2_k$ bound of $(a, \phi) $ by a constant $c_k$.
\end{proof}


\begin{rem}
The above proof also verifies that 
 one can restrict the functional spaces to
 $L^2_1(K; \tilde{S}^+)_0 
\oplus   L^2_1(X; \Lambda^1 \otimes  i {\mathbb R})_0  \cap \text{ Ker } d^*$,
and still obtain the same conclusion 
such that regularity on 
 local metric properness holds.
 \end{rem}
\vspace{3mm}

\subsection{$L^{\infty}$ estimates}
Let us take  a solution $(A_0,\psi_0)$
 to the SW equations
 over $M$, and consider the covering monopole map
 with the base $(A_0, \psi_0)$.

We take an element:
 $$(\phi,a) \in L^2_k((K',g);  \tilde{S}^+)
\oplus   L^2_k((X,g);  
 \Lambda^1 \otimes  i {\mathbb R})  
 \cap \text{ Ker } d^*$$ and set
$\tilde{\mu}(\phi,a) = (\varphi, b, h)$ and $r = || (\varphi, b, h)||_{L^2_{k-1}}$,
where $K' \subset X$ is a compact subset.

\begin{prop}\label{prop4.9}
Suppose that the AHS complex has closed range. 
Then, we have the $L^{\infty} $ estimate
$(\phi, a) \in L^{\infty}$
 in terms of $r =  || (\varphi, b, h)||_{L^2_{k-1}}$ and  $K'$
 as
 \[
 ||(\phi, a)||_{L^{\infty}(X)} \leq C(r, K')
 \]
 for some constant $(r,K')$ that depends only on $r$ and $K'$.
\end{prop}
\begin{proof} 
We  verify the conclusion when the base solution is reducible $(A_0, 0)$ 
at Step  $2$. 
The general case is verified at Step $3$.

{\bf Step 1:}
We claim that  the uniform estimate:
$$||a||_{L^8_1} \leq C ( ||d^+(a)||_{L^{p}(X)} + r) $$
holds 
for at least one of $p =2, 4,$ or $8$.

It follows from Lemma \ref{lem3.11} $(\beta)$ that 
the inequality: 
$$||a||_{L^8_1(X)} \leq c \  (\max \{ ||d^+a||_{L^8(X)} , \ 
  ||a||_{L^4_1(X)}  \ \} +  ||a||_{harm}).$$
By the same Lemma again, we have another inequality:
\begin{align*}
||a||_{L^4_1(X)} &  \leq c \  (\max \{ ||d^+a||_{L^4(X)} , \ 
  ||a||_{L^2_1(X)}  \ \} +  ||a||_{harm}) \\
&   \leq C(\max \{||d^+a||_{L^4(X)} ,
 ||d^+(a)||_{L^2}\}+||a||_{harm}),
\end{align*}
where we have used the assumption of
 closedness of the AHS complex  
for $k=1$ at the second inequality.
Thus,  we verify  the claim by combining of these two estimates.

{\bf Step 2:}
It follows from Step $1$ with the Sobolev estimate
that the uniform estimate
$$||a||_{L^{\infty}} \leq
c ||a||_{L^8_1} \leq C ( ||d^+(a)||_{L^{p}(X)} +||a_{harm}||+r) $$
holds for at least one of $p =2,4,$ or $8$.

 The Weitzenb\"ock formula gives:
$$D_A^*D_A= \nabla^*_{A} \nabla_A + \frac{1}{4}s + \frac{1}{2}F^+_A,$$
where $s$ is the scalar curvature.

Now  suppose the base point is reducible $(A_0,0)$ over $M$.
Let $\Delta$ be the Laplacian on 
the functions. 
We have the point wise estimate
$$\Delta|\phi|^2 \leq  \ <2 D^*_A \varphi - \frac{s}{2} \phi - (b + \sigma(\phi))\phi, \phi>$$
(see \cite{bauer and furuta} $p12$).
Then: 
\begin{align*}
\Delta|\phi|^2  & +  \frac{s}{2} |\phi|^2  +\frac{1}{2}|\phi|^4
 = <2 D^*_{\tilde{A}_0} \varphi  , \phi> + <2a \varphi, \phi> -<b\phi, \phi> \\
 & \leq 2(|| D^*_{\tilde{A}_0} \varphi||_{L^{\infty}} +||a||_{L^{\infty}} ||\varphi||_{L^{\infty}} )|\phi| 
 + ||b||_{L^{\infty}} |\phi|^2
 \end{align*}
holds by use of the equality
 $\sigma(\phi)\phi = \frac{|\phi|^2}{2}\phi$.
 
 By the assumption, there is  a compact subset $K' \subset X$ 
 such that  $\phi$ has compact support inside $K'$.
 Note  that an a priori estimate:
 $$||\phi||_{L^p(X)} = ||\phi ||_{L^p(K')} \leq C ||\phi ||_{L^{\infty}(K')} =C ||\phi||_{L^{\infty}(X)}$$
 holds for some constant $C=C_{K'}$.
Combining this estimate
with the equality $d^+a= b+\sigma(\phi)$,
we obtain:
\begin{align*}
||a||_{L^{\infty}} & \leq c(||a_{harm} || + ||b||_{L^p} +||\phi ||^2_{L^{\infty}} +r) \\
& \leq c(||a_{harm} || + ||b||_{L^2_2} +||\phi||^2_{L^{\infty}}+r) \\
& \leq c(||a_{harm} || + ||b||_{L^2_{k-1}} +||\phi||^2_{L^{\infty}}+r)
\end{align*}
by Lemma \ref{lem3.8}.
For $\varphi$, we have the estimates:
$$||\varphi||_{L^{\infty}} \leq C||\varphi||_{L^8_1(X)} \leq C' ||\varphi||_{L^2_3(X)}
\leq C' ||\varphi||_{L^2_{k-1}(X)}.$$
again by Lemma \ref{lem3.8}.
At the maximum of $|\phi|^2$, 
the value of $\Delta|\phi|^2$ is non-negative, 
and so we obtain:
$$||\phi||^4_{L^{\infty}} \leq c(||a_{harm}|| , ||b||_{L^2_{k-1}},   ||\varphi||_{L^2_{k-1}},
r)
(||\phi||_{L^{\infty}}+ ||\phi||_{L^{\infty}}^2+||\phi||_{L^{\infty}}^3) .$$
Thus,  we have $L^{\infty}$ estimates of the pair $(\phi,a)$
by $( ||\varphi||_{L^2_{k-1}}, ||b||_{L^2_{k-1}}, ||h||)$.

{\bf Step 3:}
Let us induce the
 $L^{\infty}$ bound for the case of general base $[A_0, \psi_0]$.
We will follow Steps $1$ and $2$.

It follows from Step $1$ 
that the uniform estimates:
 $$||a||_{L^{\infty}} \leq c ||a||_{L^8_1} \leq c(||d^+a||_{L^p} +||a_{harm}|| +r)$$
hold for at least one of $p =2,4,$ or $8$.

We set $\phi_0 := \tilde{\psi}_0 + \phi$.
 We note the bound:
 $$-c + ||\phi||_{L^{\infty}} \leq ||\phi_0||_{L^{\infty}} \leq c + ||\phi||_{L^{\infty}}$$
  holds since $\tilde{\psi}_0$ is $\Gamma$-invariant.
 The Weitzenb\"ock formula 
 gives  the point wise estimate:
\begin{align*}
\Delta|\phi_0|^2 &  \leq  \ 
 <2 D^*_A D_A \phi_0 - \frac{s}{2} \phi_0 - (b + \sigma({\phi}_0))\phi_0, \phi_0> \\
 & =  <2 D^*_A (\varphi )
 - \frac{s}{2} \phi_0 - (b + \sigma({\phi}_0))\phi_0, \phi_0> .
 \end{align*}
 
Now we have the estimates:
\begin{align*}
\Delta|\phi_0|^2  & +  \frac{s}{2} |\phi_0|^2  +\frac{|\phi_0|^4}{2}  \\
& \leq 
  \ 2|D_A(\varphi)| \ | \phi_0| +|<b \phi_0, \phi_0>|  \\
  & \leq 
  \ 2 ( ||D_{\tilde{A}_0}(\varphi)||_{L^{\infty} } 
  + ||a||_{L^{\infty}}  ||\varphi||_{L^{\infty}} ) \ | \phi_0| +||b||_{L^{\infty}}  \ |\phi_0|^2  \\
 & \leq \ 2 ( ||D_{\tilde{A}_0}(\varphi)||_{L^2_3} 
  + ||a||_{L^{\infty}} r ) \ | \phi_0| +r  \ |\phi_0|^2  \\
 & \leq \ 2 r( c
  + ||a||_{L^{\infty}} ) \ | \phi_0| +r  \ |\phi_0|^2 ,
   \end{align*}
 where we have used the estimates
 $||b||_{L^{\infty}} \leq c||b||_{L^8_1} \leq c||b||_{L^2_3} \leq r$, 
 and the elliptic estimate.

 By the assumption, there is  a compact subset $K' \subset X$ and a constant $C=C_{K'}$
 such that  $\phi$ has compact support inside $K'$.
 Note the  estimates
 $||\phi||_{L^p(X)} \leq C ||\phi||_{L^{\infty}(X)}$ and
 $$||\sigma(\tilde{\psi}_0, \phi)||_{L^p} \leq C( ||\phi||_{L^p} +||\phi||^2_{L^{2p}})
 \leq C( ||\phi||_{L^{\infty}} +||\phi||_{L^{\infty}}^2).$$
  Combining Step $1$ with the equality $d^+a= b+\sigma(\tilde{\psi}_0, \phi)$,
we obtain:
\begin{align*}
||a||_{L^{\infty}} & \leq c( ||b||_{L^p} +||\phi ||_{L^{\infty}}+||\phi ||^2_{L^{\infty}} +r) \\
& \leq c( ||b||_{L^2_2} +||\phi ||_{L^{\infty}}+||\phi||^2_{L^{\infty}}+r) \\
& \leq c( ||b||_{L^2_{k-1}} +||\phi ||_{L^{\infty}}+||\phi||^2_{L^{\infty}}+r)  \\
& \leq c'( ||\phi ||_{L^{\infty}}+||\phi||^2_{L^{\infty}}+r)  \\
& \leq c'( ||\phi_0 ||_{L^{\infty}}+||\phi_0||^2_{L^{\infty}}+r+1).  \\
\end{align*}

We now  combine the above estimates.
At the maximum of $|\phi_0|^2$, 
the value of $\Delta|\phi_0|^2$
 is non-negative, so we obtain:
\begin{align*}
||\phi_0||^4_{L^{\infty}} & \leq 
4 r \{ (  c
  + ||a||_{L^{\infty}} ) \ || \phi_0||_{L^{\infty}} +  \ ||\phi_0||_{L^{\infty}}^2\} \\
  &   = 4r(c+ || \phi_0||_{L^{\infty}} )|| \phi_0||_{L^{\infty}}  +  ||a||_{L^{\infty}} || \phi_0||_{L^{\infty}} \\
  & \leq  \{ 4r(c+ || \phi_0||_{L^{\infty}} )+
 c' ( ||\phi_0 ||_{L^{\infty}}+||\phi_0||^2_{L^{\infty}}+r+1) \}|| \phi_0||_{L^{\infty}} \\
 & \leq \{ 4r || \phi_0||_{L^{\infty}} +
 c' ( ||\phi_0 ||_{L^{\infty}}+||\phi_0||^2_{L^{\infty}}+r+1) \}|| \phi_0||_{L^{\infty}} .
  \end{align*}
 Therefore,
 we obtain  the $L^{\infty}$ estimate of the pair $(\phi,a)$
in terms of  $( ||\varphi||_{L^2_{k-1}}, ||b||_{L^2_{k-1}}, ||h||)$.
\end{proof}

\subsection{$L^2_1$ bounds}
Let $(A_0,\psi_0)$ be
 a smooth solution  to the SW equations
 over $M$. 
Let  $\tilde{\mu}$ be the 
  covering monopole map
 with the base $(A_0, \psi_0)$.
 Recall the notation that $\tilde{A}_0$ and $\tilde{\psi}_0$ 
 are both the lift on the universal covering space
  $X =\tilde{M}$.
 
 \begin{lem}\label{lem4.10}
Suppose  the AHS complex 
has  closed range over $X $, and consider the  restriction of the monopole map:
\begin{align*}
\tilde{\mu}:  L^2_1(K; \tilde{S}^+)_0  \ \oplus  \ & L^2_1(X; \Lambda^1)    \cap \text{ Ker } d^*  \\
&  \to 
 L^2(K;\tilde{S}^-)_0
 \oplus L^2(X; \Lambda^2_+;X) \oplus  H^1(X)
& \\
&(\phi, a) \mapsto (D_{\tilde{A}_0, \tilde{\psi}_0}(a, \phi), d^+(a) - \sigma(\tilde{\psi}_0, \phi), [a])
\end{align*}
for a compact subset $K \subset X$,
where $H^1_{(2)}(X)$ is the   first $L^2$ cohomology group.

For any positive numbers  $s,r >0$, 
there is a positive 
constant $C =C(r,s, K, A_0, \psi_0) >0$
such that if an element $(a, \phi) $ in the domain satisfies
the estimates
\[
\tilde{\mu}(a, \phi) ||_{L^2} \leq r, \quad
||(a, \phi)||_{L^{\infty}} \leq s,
\]
then, 
the estimate
$
||(a, \phi)||_{L^2_1} \leq C
$
holds.
\end{lem}

\begin{proof}
We follow the argument in Lemma \ref{lem4.4}.
Let us  denote
$\tilde{\mu}(\phi, a) = (\varphi, b, h)$ by
$ \phi \in L^2_1(K)_0$.
 
 Note the local Sobolev estimates:
 $$||\phi||_{L^2(K)_0} \leq C_K ||\phi||_{L^3(K)_0} \leq C'_K ||\phi||_{L^4(K)_0}.$$
 
 {\bf Step 1:}
 First, we suppose the base point is  reducible $(A_0,0)$.
 Following Step $1$ in Lemma \ref{lem4.4},
 we have the estimates:
\begin{align*}
 ||D_A(\phi)||^2_{L^2(K)}  &  + 
   ||b||_{L^2(X)}||\phi||_{L^4(K)}^2 
 \geq 
- \delta ||\phi||^2_{L^2(K)} + \frac{1}{4}||\phi||^4_{L^4(K)} \\
&  \geq   -C_K^2 \delta ||\phi||^2_{L^4(K)} 
 + \frac{1}{4} ||\phi||^4_{L^4(K)}
 \end{align*}
 In particular, there is $c =c(||\varphi||_{L^2},  ||b||_{L^2}, 
 \delta, K)$
 such that the bound
 $||\phi||_{L^4(K)_0} \leq c$
  holds, and so $||\phi||_{L^2(K)_0} \leq c'$.

The rest of the argument is the same as Steps $2$
and $3$  in  Lemma \ref{lem4.4} for this case.


  {\bf Step 2:}
  Let us consider the general case,
  and choose a solution $(A_0, \psi_0)$ to the Seiberg-Witten equations over $M$.
   The $L^{\infty}$ norm of the lift 
   $||\tilde{\psi}_0||_{L^{\infty}} \leq c$ is finite,
   because 
   $\psi_0$ is smooth and $M$ is compact.
     Let us consider the equality
  $d^+(a) = b + \sigma(\tilde{\psi}_0, \phi)$.
  Recall that  the support of $\phi $ is contained in 
   $ K$, 
  and so
  the equality $d^+(a) =b$ holds on $K^c$.

  Let us consider the equality
  $d^+(a) = b + \sigma(\tilde{\psi}_0, \phi)$.
 We have the estimates
  \begin{align*}
  ||d^+(a)||_{L^2(X)}^2 & = ||d^+(a)||_{L^2(K)}^2 +
  ||d^+(a)||_{L^2(K^c)}^2 \\
  & = ||b + \sigma(\tilde{\psi}_0, \phi)||_{L^2(K)}^2 + 
  ||b||^2_{L^2(K^c)} \\
  & \leq  (
   ||b ||_{L^2(K)} +|| \sigma(\tilde{\psi}_0, \phi)||_{L^2(K)}
   )^2 + ||b||^2_{L^2(K^c)} \\
   & \leq  4(
   ||b ||_{L^2(K)}^2 +|| \sigma(\tilde{\psi}_0, \phi)||_{L^2(K)}^2
   ) + ||b||^2_{L^2(K^c)} \\
      & \leq  4
   ||b ||^2_{L^2(X)} + C_K || \sigma(\tilde{\psi}_0, 
   \phi)||_{L^{\infty}(K)}^2 \\
 &
    \leq
     4
   ||b ||^2_{L^2(X)} + C'_K (|| \tilde{\psi}_0||_{L^{\infty}}^2
   ||\phi||_{L^{\infty}(K)}^2 +   ||\phi||_{L^{\infty}(K)}^4) \\
   & \leq C(r,s, K, \psi_0).
  \end{align*}
Hence,  we obtain the $L^2_1$ bound
 \[
 ||a||_{L^2_1(X)} \leq c(||d^+(a)||_{L^2(X)} +||a||_{harm})
  \leq C(r,s, K, \psi_0).
  \]
  For $\phi$,  we have the estimates
   \begin{align*}
  & ||D_{\tilde{A}_0}(\phi)||_{L^2}^2 \leq 
 4 ||D_A(\phi)||_{L^2}^2 + 4||a \cdot \phi||_{L^2}^2 \\
 & = 4||\varphi - a\tilde{\psi}_0||_{L^2}^2+
 4||\phi||^2_{L^{\infty}}||a||_{L^2}^2 
  \\
&  \leq 16(
 ||\varphi||_{L^2}^2 + ||\tilde{\psi}_0||_{L^{\infty}}^2
 ||a||_{L^2}^2) 
+4||\phi||_{L^{\infty}}||a||_{L^2}^2 \\
&  \leq C(r,s, K, \psi_0).
 \end{align*}
 (see the equalities above  Lemma \ref{lem2.2}).
 By combining  the elliptic estimate
 $||\phi||_{L^2_1} \leq C_{A_0, K} ( ||\phi||_{L^2} +
 ||D_{\tilde{A}_0}(\phi)||_{L^2})$ 
 with the bound
 $||\phi||_{L^2(K)} \leq C_K||\phi||_{L^{\infty} }$,
 we obtain the $L^2_1$ bound:
 \[
 ||\phi||_{L^2_1} \leq 
  C(r,s, K, \psi_0,A_0).
  \]
  \end{proof}
  
  \begin{rem}
  $(1)$ The proof in Step $1$ implies that 
 the map  $\tilde{\mu}$ above is metrically proper
 without the $L^{\infty}$ condition,
 if we use a reducible base solution.

  $(2)$
  Note that we have restricted that
  spinors in the domain  of the map
  are compactly supported on $K$.
We compare this condition with Lemma \ref{lem4.4},
where we have not required such condition, but instead,
we have assumed that the scalar curvature is positive.

\end{rem}


\subsection{Effect on smallness 
of local norms on $1$-forms}
  We now consider the effect of
  local Sobolev  norms  effect on the global norm.
 It is a characteristic of non-compactness 
 that a situation
 can happen where the 
 local norm is small,
  but the total norm is quite large. Below,
we induce a bound on the Sobolev norm under smallness of local norms
on $1$-forms.

Let $K \subset X$ be a fundamental domain.
We take an element
 $(\phi,a)$ and set
$\tilde{\mu}(\phi,a) = (\varphi, b, h)$.
Recall that we have assumed $k \geq 3$ 
(see the third paragraph of Section \ref{sec4}).
 
\begin{lem}\label{lem4.11}
Let  $k \geq 4$.
Let us choose 
a  reducible base solution $(A_0, 0)$ over $M$.
Suppose that the AHS complex has closed range,
and the Dirac operator is invertible.

Then there is  a small $\epsilon_0 >0$ and 
a positive 
constant $C >0$ 
such that
if 
 the local bound
 $||a||_{L^2_{k-1}(\gamma(K))} < \epsilon_0$
holds  for  any $\gamma \in \Gamma$, then
 the pair $(\phi,a)$ satisfies  the
bound:
$$||(\phi, a) ||_{L^{2}_{k}(X)}
\ \leq \ C (r+r^2),
$$
where 
$r : =  || (\varphi, b, h)||_{L^2_{k-1}(X)}$.

\end{lem}
\begin{proof} 
 Let us  denote $A= \tilde{A}_0+a$. 
 Let us check that the estimate:
\[
||a \cdot \phi||_{L^2_{k-1}(X)} \leq C \epsilon_0 
 ||\phi||_{L^2_{k-1}(X)}
\]
 holds.
It follows from Lemma \ref{lem3.13} that 
the estimate
$||a \cdot \phi||_{L^2_{k-1}(\gamma(K))} \leq 
C ||a||_{L^2_{k-1}(\gamma(K))} \cdot ||\phi||_{L^2_{k-1}(\gamma(K))}$ holds for some positive constant $C >0$.
Hence,
the 
following estimate holds:
\begin{align*}
||a & \cdot  \phi||^2_{L^2_{k-1}(X)}   = 
\sum_{\gamma \in \Gamma} \ 
||a \cdot \phi||^2_{L^2_{k-1}(\gamma(K))} \\
& \leq C^2 
\sum_{\gamma \in \Gamma} \ 
||a||^2_{L^2_{k-1}(\gamma(K))} \cdot||\phi||^2_{L^2_{k-1}(\gamma(K))}  \\
& \leq C^2 \epsilon_0^2 \sum_{\gamma \in \Gamma} \
 ||\phi||^2_{L^2_{k-1}(\gamma(K))} 
= C^2 \epsilon_0^2 
 ||\phi||^2_{L^2_{k-1}(X)} .
\end{align*}
Then, we obtain the estimates:
\begin{align*}
 ||\phi||_{L^2_k(X)} & \leq C' ||D_{\tilde{A}_0} (\phi)||_{L^2_{k-1}(X)} 
 \leq  C' ( ||D_{A} (\phi)||_{L^2_{k-1}(X)}  + ||a \cdot \phi||_{L^2_{k-1}(X)}) \\
 & \leq C'(r + C \epsilon_0  ||\phi||_{L^2_{k-1}(X)}).
 \end{align*}
In particular, if $\epsilon_0 >0$ is sufficiently small, then we have:
$$||\phi||_{L^2_k(X)}  \leq C''r.$$

Next, it follows from  Lemma \ref{lem3.8} 
that:
$$
||\sigma(\phi)||_{L^2_{k-1}(X)} 
\leq  C' ||\phi||^2_{L^4_{k-1}(X)} \leq C||\phi||^2_{L^2_k(X)}.
$$
Hence, we obtain:
$$
||d^+(a)||_{L^2_{k-1}(X)}  \leq r +
 ||\sigma(\phi)||_{L^2_{k-1}(X)} \leq  r + 
 C||\phi||^2_{L^2_k(X)} \\
 \leq r+C r^2$$
 as $d^+(a) + \sigma(\phi) $ has the
 $L^2_{k-1}$ norm lower than $r$.
It follows from the assumption that:
$$||a||_{L^2_k(X)}
 \leq C(||d^+(a)||_{L^2_{k-1}(X)} + ||a_{harm}||) 
\leq Cr(1+r).$$
\end{proof}


\section{Approximation by finite dimensional spaces}

\subsection{Fredholm map}
Let $H'$ and $H$ be two Hilbert spaces and 
consider a Fredholm map:
$$F = l+c: H' \to H$$  whose linear part 
$l$ is Fredholm and  where
$c$ is   compact on each bounded set.
 For our purpose later, 
 we restrict the domain by  the Sobolev space.
 A method of finite-dimensional approximation 
  has been developed 
 for a metrically proper and Fredholm map
  \cite{schwarz}, \cite{bauer and furuta}.
 It is  applied to   the monopole map 
  when the underlying space $X$ is  a compact
 four-manifold.
 
 Below, we  introduce an equivariant version of   
 this type of   approximation on a non-linear map
over  the covering space $X= \tilde{M}$ of a compact 
four-manifold
with the action of the fundamental group $\Gamma = \pi_1(M)$.

Let $E \to X$ be a vector bundle.
Let us say that a  smooth  map $c: H' =L^2_k(X; E) \to H$ is 
{\em locally compact} on each bounded set if 
its  restriction  $c|L^2_k(K,E)_0 \cap D$ 
is compact, 
where $K \subset X$ is a compact subset and $D \subset H'$ is a bounded set.

 \subsection{Technical estimates}
 We apply the results in this subsection 
  to construct a finite-dimensional approximation method
 in the next subsection. In particular, Lemma \ref{lem5.1} below
  is applied  to the Dirac operator 
 in the case of the covering monopole map.

 Let:
$$D: L^2_k(X;E') \to L^2_{k-1}(X; E)$$
be a first-order  elliptic  differential operator that
 is $\Gamma$-invariant.
Let $K \subset X$ be a compact subset, and consider the restriction:
$$D: L^2_k(K;E)_0 \to L^2_{k-1}(K; F)_0.$$

\begin{lem}\label{lem5.1}
Suppose that $D: L^2_k(X;E') \to L^2_{k-1}(X; E)$
has closed range.

Then $D$ 
is surjective
if and only if 
$D^*: L^2_k(X:E) \to L^2_{k-1}(X: E')$
is injective with closed range.
\end{lem}
\begin{proof}
{\bf Step 1:}
Let us check that the formal adjoint
$$D^*: L^2_k(X;E) \to L^2_{k-1}(X;E')$$
is injective with closed range.

If $w \in L^2_k(X;E)$ with $D^*(w)=0$  holds, then:
$$<D^*(w), u>_{L^2_{k-1}} =<w, D( u)>_{L^2_{k-1}} 
=0$$ follow for all $u \in L^2_{k}(X: E')$.
This implies $w=0$ by the surjectivity of $D$.

For any $w \in L^2_k$, there is $u' \in L^2_{k+1}$ such that $w=D(u')$ holds.
Then: 
\begin{align*}
||D^*(w)||_{L^2_{k-1}}^2 & = <D^*(w),D^*(w)>_{L^2_{k-1}} \\
& = <D^*D(u'), D^*D(u')>_{L^2_{k-1}} \geq C||u'||_{L^2_{k+1}}^2 \\
& \geq C||D(u')||_{L^2_k}^2 =C||w||_{L^2_k}^2.
\end{align*}
Hence $D^*$ also has closed range.

{\bf Step 2:}
Conversely suppose $D: L^2_k(X;E') \to L^2_{k-1}(X; E)$
has closed range but is not surjective.
Then, there is $0 \ne u  \in L^2_{k-1}(X;E)$ with:
$$<D(w), u>_{L^2_{k-1}} =<w, D^*(u)>_{L^2_{k-1}} =
0$$  for any $w \in L^2_k(X; E')$.
This implies $u \in L^2_k(X; E)$ with $D^*(u)=0$
as $D^*$ is elliptic.

Combining this with Step $1$, 
it follows under the closedness of $D$
that $D$ is surjective, if and only if $D^*$ is injective.
\end{proof}

\begin{rem}
In general, $D: L^2_k(K;E')_0 \to L^2_{k-1}(K; E)_0$
is not necessarily  surjective, even if
$D: L^2_k(X;E') \to L^2_{k-1}(X; E)$
is surjective. Later we will verify that it is asymptotically surjective in some sense.
\end{rem}

Lemma \ref{alternative} below tells us that, under some condition, 
the vectors $w_i$ 
distribute in some high spectral region in the co-kernel of $D$.

\begin{lem}\label{alternative}
Suppose   $D: L^2_k(K;E')_0 \to L^2_{k-1}(K; E)_0$ has closed range.

Moreover, assume that a sequence $w_i \in
L^2_k(X;E)$ with $ ||w_i||_{L^2_{k}}=1$ satisfies a condition:
$$\lim_{i \to \infty} \
\sup_{v \in B \cap \text{ im } D} |<w_i, v>_{L^2_{k-1}} | \ =  \ 0,$$
where $B \subset L^2_{k-1}(X;E)$ is the unit ball.

Then $w_i - P_0(w_i)$ 
converges to $ 0$ in $L^2_{k-1} $, where $P_0$ is the spectral projection to the harmonic 
space of $DD^*$.

In particular $w_i \to 0$ holds in $L^2_{k-1} $, if $D$ is surjective with closed range.
\end{lem}
\begin{proof}
{\bf Step 1:}
Assume  there is a sequence $w_i  \in L^2_k(X;E)$
with $||w_i||_{L^2_{k}}=1$ and:
$$ \sup_{v \in B \cap \text{ im } D} |<w_i, v>_{L^2_{k-1}} |
 < \epsilon_i \to 0.$$
 Then,  any  $f \in L^2_k(X;E)$ with $||D(f)||_{L^2_{k-1}} =1$
 satisfies the bounds:
   $$<D^*(w_i), f>_{L^2_{k-1}} = <w_i, D(f)>_{L^2_{k-1}}  < \epsilon_i$$
    by Lemma \ref{self-adjoint}.

    {\bf Step 2:}    Let $P$ be the spectral projection of $DD^*$ on $L^2(X; E)$.
We claim that  it is sufficient to confirm
the convergence:
\begin{align}\label{1}
\lim_{i \to \infty} \ 
    ||P_{[\lambda^2, \mu^2]}(w_i) ||_{L^2_{k-1}}  < \epsilon_i \to 0
    \end{align}
    for any $0 < \lambda < \mu$.
    Then, it follows from the equality
 $
    ||D^* w||_{L^2_{k-1}}^2 
     = <w,DD^*w>_{L^2_{k-1}} $
    that 
    the estimate:
    \begin{align}\label{2}
    ||w||_{L^2_k}^2 &= ||D^* w||_{L^2_{k-1}}^2 +  ||w||^2_{L^2} 
    \  \geq \  \mu^2 ||w||_{L^2_{k-1} }^2+ ||w||^2_{L^2}
    \end{align}
   holds  for any element $w \in \text{ im } P_{[\mu^2, \infty)}$.
    This implies that 
    the  $L^2_{k-1}$ norm of the projection to high spectra
    of $w_i$ must be small if $||w_i||_{L^2_k}=1$ holds.
    Then, combining this with the two properties (\ref{1}) and (\ref{2}), 
    we obtain the convergence:
    \begin{align}\label{4'}
    \lim_{i \to \infty} \ 
    ||P_{[\lambda^2,  \infty)}(w_i) ||_{L^2_{k-1}}  < \epsilon_i \to 0
    \end{align}
for any $\lambda >0$.
By the diagonal method, there is a decreasing sequence
$0 < \lambda_i \to 0$ such that convergence holds:
 \begin{align}\label{4}
    \lim_{i \to \infty} \ 
    ||P_{[\lambda_i^2,  \infty)}(w_i) ||_{L^2_{k-1}}  \to 0.
    \end{align}

Noting that $DD^*$ has a gap in its spectrum around zero,
we choose $\lambda^2 >0$ in this gap. 
Then
$P_{[\lambda_i^2,  \infty)}(w_i) = w_i -P_0(w_i)$,
and so its $L^2_{k-1}$-norm goes to zero by 
(\ref{4'}).
This implies the conclusion.

{\bf Step 3:}
Let us verify the claim in Step $2$.
We suppose the contrary, 
and assume that
 there is a constant with the uniform bound:
$$\lim_{i \to \infty} \ 
    ||P_{[\lambda^2, \mu^2]}(w_i) ||_{L^2_{k-1}} \geq \epsilon_0 >0.$$
We set  $f =  D^* P_{[ \lambda^2, \mu^2]}(w_i)$.
Then we have the bound:
\[
||Df||_{L^2_{k-1}} \leq \mu^2||w_i||_{L^2_{k-1}} \leq \mu^2||w_i||_{L^2_{k}} \leq \mu^2.
\]
Then, we have the estimates:
    \begin{align*}
    <w_i, Df>_{L^2_{k-1}} 
    & = <w_i, DD^*P_{[ \lambda^2, \mu^2]}(w_i)>_{L^2_{k-1}} \\
     &=<P_{[ \lambda^2, \mu^2]}(w_i)
     , DD^*P_{[ \lambda^2, \mu^2]}(w_i)>_{L^2_{k-1}} \\
&       \geq \lambda^2 ||P_{[ \lambda^2, \mu^2]}(w_i)||^2_{L^2_{k-1}}
       \geq \lambda^2 \epsilon_0.
      \end{align*}
   This contradicts to the assumption.
       Thus,  combined with Step $2$, we obtain the first statement.
       
    The last statement follows by Lemma \ref{lem5.1}.
\end{proof}

Let us consider the restriction:
$$D : L^2_k(K;E')_0 \to L^2_{k-1}(K;E)_0$$
and set:
$$\partial_K := D(L^2_k(K;E')_0)^{\perp} \cap L^2_{k-1}(K; E)_0.$$

\begin{lem}\label{high-e.v.}
Suppose $D: L^2_k(X;E') \to L^2_{k-1}(X;E)$ is surjective.
Assume that a sequence $u_i \in L^2_k(K;E')_0$
with $||u_i||_{L^2_k} =1$ satisfies:
$$||P_K(u_i) -u_i||_{L^2_{k-1}} \to 0,$$
where 
$P_K: L^2_{k-1}(K; E)_0 \to \partial_K$ is the orthogonal projection.

Then $||u_i||_{L^2_{k-1}} \to 0$ holds.  Moreover
$$||P_{[0, \lambda]}u_i||_{L^2_k} \to 0$$ holds for any $\lambda >0$,
where $P_{[0, \lambda]}$ is in  Lemma \ref{alternative}.
\end{lem}

\begin{proof}
Take any $v \in L^2_k(K;E)_0$. Then we have the estimate:
\begin{align*}
|<v, D^* u_i>_{L^2_{k-1}}| & =
|<Dv, u_i>_{L^2_{k-1}}|  = |<Dv, (1-P_K)u_i>_{L^2_{k-1}}| \\
& \leq ||(1-P_K)u_i||_{L^2_{k-1}}  ||v||_{L^2_k} \to 0.
\end{align*}
Hence, $D^*u_i$ weakly converge to zero.

Assume that  a subsequence of $\{ ||u_i||_{L^2_{k-1}}\}$  is uniformly bounded from below.
For simplicity of  notation, we 
 assume $||u_i||_{L^2_{k-1}} \geq \epsilon >0$.
Then,  from Rellich's Lemma 
 a subsequence of $\{u_i\}_i$  converges to 
a non-zero element $u \in L^2_{k-1}(K; E)_0$ with $D^*u=0$.
This cannot happen by Lemma \ref{lem5.1}
 since $D$ is assumed to be surjective.
Hence $||u_i||_{L^2_{k-1}} \to 0$ holds. 

Then one must see the property
$||P_{[0, \lambda]}u_i||_{L^2_k} \to 0$ for any $\lambda >0$, 
since we have the estimates
$||P_{[0, \lambda]}u_i||_{L^2_k}  \leq \lambda ||P_{[0, \lambda]}u_i||_{L^2_{k-1}}  \to 0$.
\end{proof}

\subsection{Finite dimensional approximations}
To apply a method of finite-dimensional approximation, 
we need to induce a kind of properness on  the image of the projection.

Let $F = l+c: H' \to H$ be a metrically proper map between Hilbert spaces.
Then there is a proper and increasing 
 function $g :[0, \infty) \to [0, \infty)$  
such that the following lower bound holds:
\begin{align}\label{5}
g(||F(m) ||  ) \geq  ||m||.
\end{align}

 Later, we analyze a family of maps
 of the form $F_i : W_i' \to W_i$, 
 where $ W_i'$ and $W_i$ are both 
 finite-dimensional linear spaces.
 We also say that the family of maps is 
 {\em metrically proper}, 
 in the sense  that there are positive numbers 
 $r_i,s_i \to \infty$ such that the inclusion:
 $$F_i^{-1}(B_{s_i} \cap W_i)  \ \subset \ B_{r_i} \cap W_i'$$
holds, where $B_{s_i}, B_{r_i}$ are the open balls
with radii  $s_i$ and $r_i$
respectively.

\begin{lem}\label{met-pro}
Let $F = l+c: H' \to H$ be a metrically proper map.
 Suppose  $l$ is surjective and $c$ is compact on each bounded set.
 Then  for any $r >0$ and $\delta_0 >0$,
   there is a finite dimensional linear subspace $W'_0 \subset H'$
 such that for any linear subspace $W_0' \subset W' \subset H'$:
 $$\text{ pr } \circ F: B_r  \cap W' \to  W$$
 also satisfies the bound:
 $$ f(||\text{ pr } \circ F(m) || ) \geq  ||m||  $$
 for any $m \in B_r \cap W'$, where 
 $W= l(W')$,  $\text{pr}$ is 
 the orthogonal projection to $W$,
   $f(x): =g(x+ \delta_0)$, and $g$ is in (\ref{5}).
 
 Moreover, the following estimate holds:
 $$\sup_{m \in B_r \cap W'} 
 ||F (m)- \text{ pr } \circ F(m)|| \  \leq  \ \delta_0.$$
   \end{lem}
 \begin{proof}
 Take any positive constant $ \delta_0 >0 $.
 Let $C  \subset H$ be the closure of the image 
 $c(B_r)$, which is compact. 
 Then,  there is a finite number of points 
 $w_1, \dots, w_m \in c(B_r)$ such that their $\delta_0$ neighborhoods cover
 $C$. 
 Choose $w_i' \in H'$ such that $l(w_i') =w_i$ hold for $1 \leq  i \leq m$,
 and let $W'_0$ be the linear span of these $w_i'$.

 The restriction
  $\text{ pr } \circ F: B_r  \cap W'_0 \to  W_0$
  satisfies the equality:
  $$\text{ pr } \circ F = l +  \text{ pr } \circ c,$$
  where $W_0 =l(W_0')$.
  Notice the equalities
  $\text{ pr } \circ F(w_i') =F(w_i')$ for $1 \leq i \leq m$.
  Then for any $m \in B_r \cap W'_0$, there is some $w_i' $ with $||c(m) -c(w_i')|| \leq \delta_0$, and
  the estimate $||F(m) - \text{ pr } \circ F(m)|| \leq \delta_0$
  holds.
  
  Since $g$ is increasing, we obtain the estimates:
  $$ g(||\text{ pr } \circ F(m) ||  +\delta_0) \ \geq \ 
   g(||F(m) || ) \ \geq \  ||m||  .$$ 
   The function $f(x) = g(x+\delta_0)$ satisfies the desired property.
   
   For any other linear subspace $W_0' \subset W' \subset H'$, 
   the same property holds for  $\text{ pr } \circ F: B_r  \cap W' \to  W$ with $W =l(W')$.
   \end{proof}

\begin{rem}
Note that 
if $l$ is not injective, then $l^{-1}(W')$ is already infinite dimensional, 
in the case of infinite covering monopole map, 
because the  kernel is infinite-dimensional.
\end{rem}

\begin{rem}
In the case of  the covering monopole map we analyzed, 
the domain  is not the full Sobolev space, 
but its closed linear subspace
 $L^2_k(K; \tilde{S}^+)_0   \oplus   L^2_k(X; \Lambda^1)    \cap \text{ Ker } d^*  $.
 Moreover, the target space is the sum of the Sobolev space with the first $L^2$ cohomology group.
 Regardless, the content in Section $5$ works for this case,
 as the linearized map splits into the sum of the Dirac operator with $d^+$ and the harmonic projection.
 \end{rem}

\subsubsection{Compactly supported Sobolev space}\label{cpt}
Let:
$$F =l+c : L^2_k(X;E') \to L^2_{k-1}(X;E)$$
be a smooth map between Sobolev spaces,
where $l$ is a first-order differential operator
 and $c$ is  pointwise  and 
 is  locally compact   on each bounded set.

Local compactness on each bounded set
 means that the restriction 
on $L^2_k(K;E')_0 \cap B$ is compact where 
$B \subset L^2_k(X;E')$
is a bounded set and $K \subset X$ is a compact subset.

 In \ref{cpt},
 we assume that
 $l$ has closed range, and that 
  the restriction:
$$F: L^2_k(K;E')_0 \to L^2_{k-1}(K;E)_0$$
is metrically proper.

 Consider the splitting:
 $$L^2_{k-1}(K;E)_0 =  l(L^2_{k}(K;E')_0) 
  \oplus  \partial_K,$$
  where $\partial_K $ is the orthogonal complement of
  $l(L^2_{k}(K;E')_0)$. We
  denote
   $pr_K: L^2_{k-1}(K;E)_0 \to l(L^2_{k}(K;E)_0)$ as the orthogonal projection.
  For any closed linear subspace $W \subset L^2_{k-1}(K;E)_0$,
  we also denote:
  $$\text{pr}_W: L^2_{k-1}(K;E)_0 \to W$$ 
  as the orthogonal projection.

Let $S \subset L^2_k(K;E')_0$ be the unit sphere, and consider the closure of the image
as:
$$\overline{c( S)}  \  \subset \
l(L^2_k(K;E')_0) \oplus \partial_K  =L^2_{k-1}(K;E)_0.$$
We say that $c$ is quadratic if 
$c(av) = a^2 c(v)$ holds for any $a \in \mathbb{R}$ and 
$v \in L^2_k(K;E')_0$.

\begin{lem}\label{regular}
Assume moreover  that 
$c$ is quadratic.

 If the $w_1 \ne 0$ component in $l(L^2_k(K;E')_0)$
  is non-zero for
  any element $w =(w_1,w_2) \in \overline{c( S)} $,
  then
  the composition
 $$\text{pr}_K \circ F: L^2_k(K;E')_0 \to  l(L^2_k(K;E')_0) \subset L^2_{k-1}(K;E)_0$$
 is metrically proper.
  \end{lem}

\begin{proof}
Since $\overline{c( S)} $ is compact, 
$$r = \inf_{w \in \overline{c( S)} } \ ||w_1|| \ > \ 0$$
is positive.
Then for any $u \in L^2_k(K;E')_0$, 
the estimate $||\text{pr}_K(c(u))|| \geq r ||u||_{L^2_k}^2$ holds because
 $c$ is quadratic.
On the other hand, $||l(u)||_{L^2_{k-1}} \leq C ||u||_{L^2_k}$ holds for some $C$.
Hence, we obtain the lower bound:
$$||\text{pr}_K \circ F(u)||_{L^2_{k-1}} \geq  r ||u||_{L^2_k}^2 - C ||u||_{L^2_k}.$$
The conclusion follows. 
\end{proof}

Of course, it is unrealistic  to expect that such assumption could happen.
Therefore, we  state a modified version.

Let $B_r \subset L^2_k(K;E')_0$ be  the open ball with radius $r$.
Take a finite-dimensional linear subspace $U_r  \subset \partial_K $ and 
denote the orthogonal projection by:
 $$P_r: L^2_{k-1}(K;E)_0   \to l(L^2_{k}(K;E')_0) \oplus U_r
 \subset L^2_{k-1}(K;E)_0.$$

\begin{lem}\label{modify}
Suppose $F$ is locally metrically proper.
Moreover, suppose that  $l$ has closed range
and $c$ is locally compact on each bounded set.

Then, for each $r>0$, there is a finite-dimensional linear subspace
$U_r \ \subset \ \partial_K $  such that 
 the composition:
$$P_r \circ F:  B_r  \to l(L^2_{k}(K;E')_0) \oplus U_r$$
is still metrically proper.
\end{lem}

\begin{proof}
The proof is  very much in the same spirit as
Lemma \ref{met-pro}.
We fix $\delta >0$, which is independent of $r>0$.
Then, for this $\delta >0$ and $r >0$, 
we take a sufficiently many but
 finite set of points $\{ p_1, \dots, p_m\} \subset B_r$
 and denote
  the  finite-dimensional linear subspace
  spanned by $c(p_i)$
  as $\tilde{W}_r$. 
   Then,  
we can  assume 
$\tilde{W}_r \subset L^2_{k-1}(K;E)_0$
satisfies the bound $d(\tilde{W}_r, c(B_r)) < \delta$ as
 $c$ is  locally compact on each bounded set.
Hence, any element $w \in \text{pr}_{\partial_K}
 ( \overline{c(B_r)}) $
is at most $\delta$ away from $U_r : = \text{pr}_{\partial_K}(\tilde{W}_r)$,
where $\text{pr}_{\partial_K}$ is the orthogonal projection to $\partial_K$.

Let us consider the linear plane:
$$L_r := l(L^2_{k}(K;E')_0) +\tilde{W}_r
= l(L^2_{k}(K;E')_0) \oplus  U_r.
$$
Then, we obtain the bound:
\begin{align}\label{111}
\delta \geq \sup_{m \in B_r} \ 
||F(m) - \text{pr}_{L_r} \circ F(m)|| ,
\end{align}
where the right-hand side depends on $r>0$,
 but $\delta$ is independent of it.
 Hence, the conclusion follows from the estimate
 (\ref{111})
with metric properness of $F$.
\end{proof}

\begin{cor}\label{fin}
Suppose $F, l,$ and $c$ satisfy the conditions in Lemma \ref{modify}.

There are finite-dimensional linear subspaces $W_r' \subset L^2_k(K;E)_0$ 
and 
$U_r \subset \partial_K$
with a linear map:
$$l' : L^2_k(K;E')_0 \to  l(L^2_{k}(K;E')_0) \oplus U_r$$
such that the following hold:

$(a)$ The composition of $l'$ with the projection $\text{pr}_K$ to the first component 
coincides with $l$ as:
$$l = \text{pr}_K \circ l': W_r'  \to  l(W_r').$$

$(b)$  Let $\text{pr}_{W_r}: L^2_{k-1}(K,F)_0 \to W_r := l'(W_r)$ 
be  the orthogonal projection. Then:
 $$\text{pr}_{W_r} \circ F : W_r' \cap B_r  \to  
 W_r $$
is proper (see above Lemma \ref{met-pro}).

$(c)$ If $l: L^2_k(X;E) \hookrightarrow L^2_{k-1}(X;F)$
is injective, then 
the estimates:
$$||l(v)|| \  \leq \ ||l'(v)|| \ \leq  \ 3||l(v)||$$
hold for any $v \in L^2_k(K;E)_0$.
\vspace{2mm}

The same properties hold if one takes a larger but still 
finite-dimensional linear subspace 
$\tilde{W}_r' $ with 
$W_r' \subset \tilde{W}_r' \subset    L^2_k(K;E')_0  $.
\end{cor}

\begin{proof}
Let  $\{ p_1, \dots, p_m\} \subset B_r$ and $U_r$ be as in the proof of Lemma \ref{modify}
 such that $F(p_i) \in   l(L^2_{k}(K;E')_0) \oplus U_r$ holds.
  Let $W_r'  \subset L^2_k(K;E')_0$ be the finite-dimensional  linear subspace spanned by 
  $\{ p_1, \dots, p_m\} $.
 One may assume the estimate:
 \begin{align}\label{6}
 d( \text{pr}_K (c(B_r)) , l(W_r')) < \delta
 \end{align}
 by adding extra points, if necessarily.

  Let us introduce a linear map as follows.
  Let $f: [0, \infty)^2  \to [0, \infty)$ be a smooth map with:
  $$f(r,s) = 
  \begin{cases}
  s & r \geq s , \\
  2r & 2r \leq s.
  \end{cases}$$
  Then, we define:
  $l' : W_r' \to  l(L^2_{k}(K;E')_0) \oplus U_r$
  by the linear extension of the map:
   $$l'(p_i) = l(p_i) + f(||l(p_i)||, ||   \text{pr}_{U_r} \circ c(p_i)||)
    \frac{   \text{pr}_{U_r} \circ c(p_i)}{ 1+||   \text{pr}_{U_r} \circ c(p_i)||}.$$
  We require a slightly complicated formula for the second term because
     the norm $   ||\text{pr}_K \circ c(p_i)||$
 can grow more than linearly.
Clearly both $(a)$ and $(c)$
 are satisfied.
 
 There is a proper increasing map
 $g: [0, \infty) \to [0, \infty)$ independent of $r$ 
 such that:
 $$\max( \ ||\text{pr}_K \circ F(v)||_{L^2_{k-1}}, 
 || \text{pr}_{U_r} \circ c(v) ||_{L^2_{k-1}}\ ) \geq g(||v||_{L^2_k})$$
 holds for $v \in D_r \cap W_r'$,
 since  $F$ is metrically proper.
 Hence,   $(b)$ follows  by combination  with (\ref{6}).
   \end{proof}

\subsubsection{Asymptotic surjection}\label{as-sur}
In this section, we
  assume  the restriction $F|L^2_k(K;E)_0$ is metrically proper on any compact subset,
  $l: L^2_k(X;E) \to L^2_{k-1}(X;F)$ is surjective, 
  and  $c$ is locally compact on each bounded set.

One can obtain a finite-dimensional approximation of $F$
as Corollary \ref{fin},
but the linear map $l'$  may be quite different from $l$.
In this subsection, we will use a larger compact subset $K \subset L$
in the target space such that $l'$ surely approximates $l$.

We want to use  $P_r \circ F : W_r' \cap B_r  \to   l(L^2_{k}(K;E')_0) \oplus U_r$ 
in Lemma \ref{modify}
as an asymptotic approximation of $F$ instead of  using
$F$ itself. 
As above, we have to use a pair of 
compact subsets to approximate its linearized operator.
Note that the choice of these linear subspaces $W_r'$ or $U_r$ heavily depends on 
the compact subset $K \subset X$. Thus, we denote $P_r^K$ instead of $P_r$ above.

Let us consider two compact subsets $K \subset L \subset X$,
and let  $\partial_L \subset L^2_{k-1}(L; E)_0$ be 
the orthogonal complement of $l(L^2_k(L,E')_0)$.

\begin{lem}\label{proj}
We fix $K$. Then, for any $\epsilon >0$,
 there is another compact subset $L \supset K$
such that the orthogonal projection:
$$P: L^2_{k-1}(K; E)_0 \to \partial_L$$
satisfies the estimate:
$$||P|| \leq \epsilon.$$
\end{lem}

\begin{proof}
Let us choose a compact subset $L \subset X$ such that
it admits a smooth cut-off function
$\varphi: L \to [0,1]$
with the following properties:

$(1)$
 $\varphi|K \equiv 1$;
 
  $(2) $
 $ \varphi|L^c \equiv 0$; and
 
  $(3)$
  $||\nabla(\varphi) ||_{L^2_{k-1}} < \delta$,
 where $\delta >0$ is sufficiently small.
 \\
 There is a constant $C$ such that 
 for any $u \in L^2_{k-1}(K;E)_0$, there is $v \in l^{-1}(u)  \subset  L^2_k(X;E')$
 with $||v||_{L^2_k} \leq C||u||_{L^2_{k-1}}$.
 
It follows from the equality:
 $$ u = \varphi l(v) = l(\varphi v) - [l, \varphi] v$$
that the estimates
$$||P(u)||_{L^2_{k-1}} \leq || [l, \varphi] v||_{L^2_{k-1}} 
\leq \delta ||v||_{L^2_k} \leq C\delta ||u||_{L^2_{k-1}}$$
hold.
\end{proof}

Let $L$ be a  compact subset and 
 apply
Corollary \ref{fin} to $L$ as:
$$l' : L^2_k(L;E')_0 \to  l(L^2_{k}(L;E')_0) \oplus U^L_r$$
with
 $W_r' \subset L^2_k(L;E)_0$
and  a proper map:
 $$\text{pr}_{W_r}  \circ F : W_r' \cap B_r  \to   W_r 
  \subset L^2_{k-1}(L,E)_0.$$

\begin{cor}\label{asymp}
We fix $K$. Then, for any $\epsilon >0$,
 there is another compact subset $L \supset K$
such that the operator norm of  the restriction satisfies the estimate
\begin{align}\label{7}
||(l - l')|L^2_k(K,E')_0|| < \epsilon.
\end{align}

Moreover,
the restriction:
 $$\text{pr}_{W_r} \circ F : W_r' \cap B_r  \cap L^2_k(K;E')_0  \to W_r \subset    l(L^2_{k}(L;E')_0) \oplus U^L_r$$
 satisfies the estimate:
 \begin{align}\label{8}
 \sup_{m \in W_r' \cap B_r  \cap L^2_k(K;E')_0}
 ||\text{pr}_{W_r} \circ F(m)  - F(m)||_{L^2_{k-1}}
  \ < \ \epsilon.
  \end{align}

The same properties hold if one takes a larger
 but still finite-dimensional linear subspace 
$L^2_k(L,E')_0 \supset \tilde{W}_r' \supset W_r'$.

\end{cor}

\begin{proof}
The first estimate (\ref{7}) follows from Corollary \ref{fin} with Lemma \ref{proj}.
Moreover, through  combination with Lemma \ref{met-pro}, we can obtain the second estimate (\ref{8}).
\end{proof}

\subsection{Finitely approximable maps}\label{5.4}
So far,  we have fixed a compact subset $K \subset X$. 
Our final aim is to approximate a non-linear map between  Sobolev spaces over $X$  
by a family of maps between   finite-dimensional linear subspaces
that are included in exhausting
compactly supported
 functional spaces.

 Consider a map 
 $F =l +c : L^2(X;E') \to L^2_{k-1}(X;E)$.
Throughout \ref{5.4}, we  assume that 
$l: L^2_k(X,E') \cong L^2_{k-1}(X,E)$ is a 
linear isomorphism and 
$c$ is is pointwise and locally compact on each bounded
set. We also assume that F is locally metrically proper. 
Note that these conditions satisfy the properties we have assumed in
Subsections \ref{cpt} and \ref{as-sur}.

Recall the locally strong properness of a map in Definition \ref{def1.1}. Note
that the above properties are satisfied,
 if $F$  is locally strongly proper, $l$
is isomorphic, and c is pointwise and locally compact on each bounded
set.

In particular, it follows from Theorem \ref{thm4.1}  that the covering monopole
map satisfies the above properties
 if the AHS complex has closed range
and l is isomorphic. 
This is equivalent to the two properties that, if the AHS
complex has closed range and the Dirac operator is invertible, when a
reducible solution exists and the fundamental group of  $X$  is infinite.

Let us 
 consider a family of maps:
$$F_i: W_i' \to W_i ,$$
where $W_i' \subset H'$  and $W_i \subset H$ 
are both finite dimensional  linear subspaces whose  
respective unions
$$\cup_{i=0}^{\infty} \ W'_i \ \subset  \ H', \qquad
\cup_{i=0}^{\infty} \ W_i \ \subset \ H$$
are dense.
We denote by $B_{r_i}' \subset W_i' $ 
and $B_{s_i} \subset W_i$ 
the open  balls with radii $r_i$ and $s_i$, respectively.

Let us say that the   family $\{F_i\}_i$ is {\em asymptotically proper} on $H'$,
 if  there are two  sequences
 $s_0 <s_1<  \dots \to \infty$ and
 $r_0 <r_1< \dots \to \infty$ such that
 the following embeddings hold:
$$F_i^{-1}(B_{s_i}) \subset B_{r_i}' .$$
 To obtain a
better approximation
as in Corollary \ref{asymp}, 
we can take a larger compact subset.

Let us apply this notion to $H'= L^2_k(X;E')$ and $H=L^2_{k-1}(X;E)$.
Let:
$$F =l+c : L^2_k(X;E') \to L^2_{k-1}(X; E)$$
be a 
 locally strongly proper map, where $l$ is a first-order elliptic differential operator and $c$ is
pointwise and   locally compact on each bounded set.
Suppose  $l$ is  an isomorphism.
Let us restate Corollary \ref{asymp} in terms
of a family of maps.

\begin{cor}\label{K-approx}
There exists an exhaustion 
$\cup_i \ K_i =X$ by compact subsets and 
 families of finite-dimensional linear subspaces:
$$W_i' \subset  L^2_k(K_i; E')_0 \  \text{ and }  \ 
W_i \subset  L^2_{k-1}(K_i; E')_0$$
such that the following holds. 
For any compact subset $K \subset X$,
 the limit of the operator norms of the restriction:
$$\lim_{i \to \infty} \
 ||(l -l_i)|W_i' \cap L^2_k(K,E')_0 || =0$$
holds.
 Moreover, the restriction approximates $F$ as:
 $$ \ \lim_{i \to \infty} \ 
 \sup_{ m \in W_{i}' \cap B_{r_{i}} \cap L^2_k(K,E')_0}
 ||F_i (m)  - F(m)||_{L^2_{k-1}}=0,$$
 where 
$F_i := \text{ pr}_{W_i} \circ F : W_i' \to  W_i$
is an  asymptotically proper family
with linear isomorphisms
$l_i: W_i' \cong W_i$, and 
 $\text{ pr}_{W_i} : H \to W_i $
 is the orthogonal projection.
\end{cor}


Let $K_1 \Subset   \dots  \Subset   K_i \Subset K_{i+1} \subset X$ 
be  an exhaustion of $X$ by  compact subsets
 and $E', E \to X$ be  vector bundles over $X$.
Then we have an increasing  family of Sobolev spaces:
$$L^2_k(K_i; E')_0 \subset L^2_k(K_{i+1}; E')_0 \subset \dots  \subset L^2_k(X;E').$$

Let $F =l +c : L^2(X;E') \to L^2_{k-1}(X;E)$ be a smooth map,
where $l$ is a first-order differential operator,
 and $c$ is the non-linear term which is a pointwise operator.

Let  $F_i$ and 
$ \text{ pr}_{W_i} :  L^2_{k-1}(X;E) \to W_i$  
be  in Corollary \ref{K-approx}, 
which considers a 
 situation
 in which  a compact subset $K$ is fixed.
Now, we use whole families of such approximations over $K_i$, 
and select well-approximated maps as below.

First, let us state a weaker version 
of the approximation.

\begin{lem}
Let $W_i'$ be in Corollary \ref{K-approx}.
For any $v' \in L^2_k(X; E')$,
 there is  an approximation 
$v_i' \in W_i'$ with $v_i' \to v' $ in $L^2_k(X; E')$ such that the convergence:
$$\lim_{i} ||F_i(v'_i) - F(v')||_{L^2_{k-1}} = 0$$
holds, where $F_i = \text{pr}_{W_i} \circ F: W_i'  \to  W_i$. 
\end{lem}
\begin{proof}
Let $v_i' \in L^2_k(K_i; E')_0$ be any approximation with 
$v_i' \to v'$.
By Corollary \ref{K-approx}, 
$\lim_{i} ||F_i(v'_{i_0}) - F(v'_{i_0})||_{L^2_{k-1}} = 0$ holds
for each $i_0$.
Since $F$ is continuous, 
for any $\epsilon >0$, there is $i_0$ such that  
the estimate
$||F(v'_{i_0}) - F(v')|| < \epsilon$ holds.

Note that we can  regard $v'_{i_0} \in W_i'$ for $i \geq i_0$.
Then, we  replace the approximation of $v'$ 
such that  we obtain the desired estimate
 by applying the triangle inequality.
\end{proof}

\vspace{3mm}

Let $B_{r_i}' \subset W_i'$ and $B_{s_i} \subset W_i$
be  the open balls with radii  $r_i$ and $s_i$
 respectively.
\begin{defn}\label{weak-f.appr} \cite{kato4}
Let $F = l+c: H' \to H$ be a smooth map, 
where $l$ is the linear part
and $c$ is  its non-linear term.

We say that $F$ is  weakly finitely approximable, if there is
an increasing family 
of finite-dimensional linear subspaces
$ W'_0   \subset   W'_1 
 \subset   \dots  \subset  W'_i  
  \subset  \dots  \subset H' $, 
 an asymptotically proper family $\{F_i\}_i$, and  linear isomorphisms  $l_i: W_i' \cong W_i$
 such that:

$(1)$ Their  union 
$\cup_{i \geq 0} W'_i \ \subset H'$ is dense;

$(2)$ 
 The inclusion $F_i^{-1}(B_{s_i}) \subset B_{r_i}'$ holds,
where:
$$F_i  := \text{pr}_i \circ F: W'_i \to W_i  =l_i(W_i)$$
and $\text{pr}_i : H \to W_i$ is the orthogonal projection.

$(3)$ For each $i_0$, the norm 
converges:
$$\lim_{i \to \infty} \  \sup_{m \in B_{r_{i_0}}'} \ ||F(m) - F_i(m)||_{L^2_{k-1}} =0; $$

$(4)$
The operator norm of the restriction converges:
$$\lim_{i \to \infty} \  \ ||(l -l_i)|W_{i_0}'|| =0 ; \ \  \text{ and }$$

$(5)$
 The uniform bounds 
$C^{-1} || l || \leq  ||l_i|| \leq C || l ||$ hold on their norms,
where $C$ is independent of $i$.
\end{defn}
Later  we will introduce finite approximability,
below Definition \ref{def6.1}.

\vspace{3mm}

Let us introduce two variations:

\vspace{3mm}

$(\text{\bf A})$  
Suppose both $H'$ and $H$ admit linear  isometric  
actions by a group  $\Gamma$, and 
assume that $F$ is $\Gamma$-equivariant.
Then, we say that 
$F$ is {\em weakly finitely $\Gamma$-approximable},
 if moreover for 
the above family $\{W'_i\}_i$,    the union:
 $$ \cup_i \ 
  \{ \ \gamma(W_i') \cap W_i'  \ \} \ \subset \ H'$$
 is dense
for any $\gamma \in \Gamma$.

Note that 
the above family $\{F_i\}_i$  satisfies
 convergence  for any  $\gamma \in \Gamma$:
$$\lim_{i \to \infty}  \  \sup_{m \in B_{r_{i}}' \cap
 \gamma^{-1}(B_{r_{i}}')
}  \ ||\gamma F_i( m) - F_i( \gamma m)|| =0$$
because the following estimate holds:
\begin{align*}
 ||\gamma F_i( m) - F_i( \gamma m)|| & \leq
  ||\gamma F( m) - \gamma F_i( m)||+
   ||\gamma F( m) - F_i( \gamma m)|| \\
   & =  ||F( m) - F_i( m)||+
   ||F( \gamma m) - F_i( \gamma m)||.
   \end{align*}

Let us take $\gamma \in \Gamma$, 
and consider the $\gamma$ shift of the weakly  finite approximation data:
$$\gamma(W_i'), \quad
\gamma^*(F_i), \quad  \gamma^*(l_i).$$
This shift 
gives anothor weakly  finite approximation of $F$.

\vspace{3mm}

$(\text{\bf B})$  
Suppose $F= l+c:L^2_k(X;E') \to L^2_{k-1}(X; E)$
 consists of a first-order differential operator and $c$ is the non-linear term 
by some pointwise operation.

Let us say that a  weakly finite approximation of $F$ is {\em adapted}
if there is  an exhaustion
$K_1 \subset  \dots  \subset   K_i \subset K_{i+1} \subset  \cdots  \subset X$ 
 by  compact subsets such that
 the following  inclusions  both hold:
 $$W_i' \subset L^2_k(K_i; E')_0, \quad W_i 
 \subset L^2_{k-1}(K_{i}; E)_0.$$

Hereinafter, we always assume that any weakly finite approximation of $F$ is adapted
whenever $F$ is a map between Sobolev spaces.
Note that if a group $\Gamma$ acts on $X$, 
then the  $\Gamma$ orbit 
$\Gamma(K_i) =X$ coincides with $X$ for all sufficiently large $i$.
Hereinafter,
 we  assume this property for any $i$.

\begin{prop}\label{fin-approx}
Let: 
$$F =l+c : H' =  L^2_k(X;E') \to H = L^2_{k-1}(X; E)$$
be a $\Gamma$-equivariant
 locally strongly proper map, where $l$ is a first-order elliptic differential operator and $c$ is
pointwise and locally compact on each bounded set.
Suppose  $l$ is  isomorphic.

Then, there is an adapted  family of
 finite-dimensional linear subspaces $\{W_i'\}_i $
that weakly finitely $\Gamma$-approximates  $F$.
\end{prop}
\begin{proof}
Take an 
exhaustion of $X$ by  compact subsets
 $K_1 \subset  \dots \subset  K_i  \subset K_{i+1} \subset X$.
It follows from Corollary \ref{K-approx} 
 that 
 there are
 finite-dimensional linear subspaces
$W_i' \subset L^2_k(K_i;E')_0$ 
and $W_i \subset L^2_{k-1}(K_{i};E)_0$
with positive numbers $s_i, r_i >0$ 
such that the family of maps:
$$F_i := \text{pr}_{W_i} \circ F:  B_{r_i}'  \to W_i$$
 is asymptotically proper and
 satisfies the inclusion
 $F_i^{-1}(B_{s_i}) \subset B_{r_i}'$. 
 Moreover, the restrictions
 satisfy the
   convergences:
 \begin{align*}
&  \lim_{i \to \infty} \  
 ||F_i   -  F||_{B_{r_{i_0}}'} =0 \ \text{ and }  \\ 
 &
 \lim_{i \to \infty} \ || (l-l_i)|W_{i_0}' || =0.
 \end{align*}
 
 These properties also hold if $W_i'$ is replaced
 by any other 
 finite-dimensional linear subspace $\tilde{W}_i' $
 that contains $W_i'$.
 Thus, we assume that the union:
 $$\cup_{i \geq 1} \ W_i'  \ \subset \ L^2_k(X;E')$$
 is dense, and so
 $\cup_i \ W_i  \subset L^2_{k-1}(X;E)$
  is also dense because $l$
 is assumed to be isomorphic.

Let us consider the $\Gamma$ action.
Let us replace $W_i' $ by the span of $ \Gamma(W_i') \cap L^2_k(K_i;E')_0$
and denote it by  $W_i'' $.
Note that $W_i''$  contains $W_i'$ and is also 
 finite-dimensional. Moreover
the corresponding $F_i$ and $l_i$ still 
give the weakly  finite approximation data.

Take any $\gamma$ and $i$. Then, 
there is some $j \geq i$ such that 
$\gamma^{-1}(W_i') \subset  L^2_k(K_j;E')_0 $.
So
$\gamma^{-1}(W_i')  \subset W_j''$. Hence,
the inclusion:
$$\gamma(W_j'') \cap W_j'' \supset W_i'$$
holds.  This implies that 
the union of the left-hand side with respect to $j$
is dense in $L^2_k(X;E')$.
 This gives a weakly $\Gamma$-finite  approximation of $F$.
\end{proof}

\begin{rem}\label{gamma-inv}
The above argument 
implies that by  adding more linear planes if necessarily,
one may assume $\Gamma$-invariance of the union:
$$\gamma (\  \cup_i  \ W_i \ ) \  = \ \cup_i \ W_i$$
for any $\gamma \in \Gamma$.

In other words, for any $\gamma $ and $i$, there is some $i' \geq i$  such that
the inclusion 
$\gamma(W_i) \subset W_{i'}$ holds.
\end{rem}

\subsection{Sliding end phenomena}
In general,
$F_i|B_{r_i} \cap W'_i$
may not converge to  $F$  in operator topology
if there is a  difference between the images of 
$\text{pr} \circ c(L^2_k(K;E')_0)  \subset L^2_{k-1}(K;E)_0$
and  $l(L^2_k(K;E')_0)$.

\begin{exm}
Let us give a simple example.
Let $H' =H =l^2({\mathbb Z})$ and consider
$F = l+c: H' \to H$,
where $l(\{a_i\}_i) =\{a_{i-1}\}_i$ is the shift and 
$c(\{a_i\}) = \{b_i\}_i$ with $b_i =f(a_i)$.

We set 
$V_{m,n} = \{ \ \{a_i\}_i : a_i =0 \text{ for  } i \leq m \text{ or } n \leq i \}$.
Then, $l: V_{m,n} \cong V_{m+1,n+1}$ and the restriction
$\text{ pr} \circ F - F : V_{m,n} \to V_{m,n+1}$ satisfy 
$(\text{ pr} \circ F - F) \  (\{a_i\}_i) = - f(a_m).$
Therefore,
 $\text{ pr} \circ F$ pushes bubbling $f(a_m)$ off as $m \to - \infty$.
\end{exm}

Let us introduce a sliding end quantity.
Let $K_1 \Subset   \dots  \Subset  K_i \Subset K_{i+1} \Subset X$ 
be an exhaustion of $X$ by  compact subsets, 
and:
$$\text{pr}_i: L^2_{k-1}(K_i; E)_0 \to l(L^2_k(K_i;E')_0)$$ be the orthogonal projections.

We introduce a 
 {\em sliding end quantity} $b(F) \in [0, C_0]$  given by:
$$ b(F) := \    \inf_{\{K_i\}_i} \  \lim_{i \to \infty} \  b(F)_i,$$
where $b(F)_i  =  \sup_{v \in L^2_k(K_i)_0}
 \{ \  ||(1- \text{pr}_i)(c(v))||_{L^2_{k-1}} : ||v||_{L^2_k(K_i)_0} \leq 1 \ \}$.

\section{Infinite-dimensional Bott periodicity}
\label{sec6}
Let $H$ be a Hilbert space.
Higson, Kasparov, and Trout constructed 
the  Clifford $C^*$-algebra $S{\frak C}(H)$ of $H$, 
and verified Bott periodicity 
$\beta: K(C_0(\mathbb R) \cong K(S{\frak C}(H))$
by use of approximations by 
 finite-dimensional linear subspaces.
 If a discrete group $\Gamma$ acts on $H$ linearly and isometrically, 
 then its construction induces the equivariant Bott periodicity:
 $$\beta: K(C_0(\mathbb R \rtimes \Gamma) \cong K(S{\frak C}(H) \rtimes \Gamma),$$
where the crossed product is full.

\begin{rem}
Even though the right-hand side $C^*$-algebra is
generally  quite `huge',
 its $K$-theory has the same size as $K_*(C^*\Gamma)$.
In non commutative geometry, 
it is conjectured (the 
Baum-Connes conjecture)
 and actually verified for many classes of groups
that the $K$-theory of the reduced group $C^*$-algebra $C^*_r \Gamma$ 
 is isomorphic to the
$K$-homology of the classifying space $B \Gamma$
 if $\Gamma$ 
 is torsion free. In particular, $K_*(C^*\Gamma)$
  is isomorphic to $K_*(B \Gamma)$ 
if $\Gamma$  is a torsion-free amenable
group \cite{hk}. Notice that $K_*(X)$  is rationally isomorphic to 
$H_*(X; \mathbb Q)$
for a CW complex $X$.
\end{rem}

\subsection{Asymptotic unitary maps}
Let $l: H' \cong H$ be a linear isomorphism 
between  Hilbert spaces.

\begin{defn}\label{def6.1} \cite{kato4}
Let  $H'$ and $H$   be Hilbert spaces and
$l: H' \cong H$ be a linear isomorphism.
Then, 
$l$ is asymptotically unitary, if for any $\epsilon >0$,
there is a finite-dimensional linear subspace 
$V \subset H'$ such that
the restriction
$$l: V^{\perp} \  \cong \ l(V^{\perp})$$
satisfies the estimate 
$$||(l-\bar{l})| V^{\perp}|| < \epsilon$$
on its operator norm,
 where $\bar{l}$ is the unitary of the polar decomposition of
  $l: H' \cong H$.
\end{defn}

Let $F =l+c: H' \to H$ be  weakly finitely approximable
with $F_i: W_i' \to W_i$ 
(see Definition \ref{weak-f.appr}).
In \cite{kato4}, we have introduced {\em finite approximability} on $F$.

\begin{defn}\label{def6.2}
$(1)$
$F$ is  finitely approximable, if moreover  
$l$ is asymptotically unitary.

$(2)$
 $F$ is strongly finitely approximable if
it is finitely approximable with the same $l_i =l$
and $c_i = \text{pr}_i \circ c$ such that:
$$\lim_{i \to \infty} \ \sup_{m \in B_{r_i}' } \
 ||(1 - \text{pr}_i ) \circ c(m) || =0$$
holds,
 where 
$\text{pr}_i : H \to W_i$ is the orthogonal projection.

$(3)$
$F$ is asymptotically finitely approximable
if there is a stratification by infinite-dimensional Hilbert subspaces:
\[
H_1' \subset H_2' \subset \cdots \subset H'
\]
with $W_i' \subset H_i'$ such that the restriction of the linear part
$l|H_i'$ on $H_i'$ is asymptotically unitary for each $i$.
\end{defn}

\vspace{3mm}

\begin{rem}
In $(3)$ above, when we consider the
$\Gamma$ action, we do not generally
require $\Gamma$-invariance on each $H_i'$.
Note that, by definition, the union $\cup_i H_i' \subset H'$ is dense.
\end{rem}

Suppose $H'$ and $H$ are 
the Sobolev spaces such that
$l: L^2_k(X;E') \cong L^2_{k-1}(X;E)$ 
 is an elliptic operator that gives an isomorphism.
Recall the Sobolev norm 
  introduced  in Section $3$.
We denote by $P$ the spectral projection of $l^* \circ l$,
where $l$ is regarded as an  unbounded operator 
on $L^2(X;E')$ and $l^*$ is the formal adjoint operator.
We mostly regard $l$ as a bounded
operator between Sobolev spaces and, hence,
$l^*$ is the adjoint operator 
between them, unless otherwise stated.

The following lemma \ref{spe-proj} 
 is a key to inducing 
  asymptotic unitarity for an elliptic operator.
\begin{lem} \label{spe-proj}
Let  $l$ be as above.
Then the  operator
$l : H' \cong H$
satisfies the property that for any $\epsilon >0$,
 there is $\lambda_0 >>1$ 
 such that the operator norm of 
  the restriction of $l^* \circ l$
 on $P[\lambda_0, \infty) \subset H'$ 
 satisfies the estimate
 \begin{align}\label{10}
 ||(l -\bar{l})| P[\lambda_0, \infty) || < \epsilon
 \end{align}
 where $\bar{ \quad }$ is the unitary of the polar decomposition.
 
 In particular, the self-adjoint operator:
$$U := \bar{l}^{*} \circ l = \sqrt{l^*l}: H' \cong H'$$
satisfies the estimate:
 \begin{align}\label{11}
 ||(U -\text{ id })| P[\lambda_0, \infty) || < \epsilon.
 \end{align}
\end{lem}
\begin{proof}
The latter statement (\ref{11}) follows from the former (\ref{10}).

Let us verify the former property (\ref{10}).
We set:
$$P_N(c) = \frac{c^N-1}{c-1}$$
for $c >1$.  Notice the equalities
$c P_{N-1}(c) + 1 = P_N(c)$.

If $u $ is 
 an eigenvector  vector  with $l^*l(u) = \lambda^2 u$, then 
 the formulas $||u||^2_{L^2_k} = P_k(\lambda^2) ||u||^2_{L^2}$ hold
for all $k \geq 0$. 
We can check this by induction as: 
\begin{align*}
<u,u>_{L^2_k } & = <l(u), l(u)>_{L^2_{k-1}} +<u,u>_{L^2} \\
& =  <l^*l(u), u>_{L^2_{k-1}} +<u,u>_{L^2} = \lambda^2 <u,u>_{L^2_{k-1}} + ||u||^2_{L^2}.
\end{align*}

In particular, if 
$u \in L^2_k(X;E')$ with $||u||_{L^2_k} =1$
 lies in the image of the spectral projection to $[\lambda_0^2, \infty)$
on $l^* \circ l$, then $||u||_{L^2} $ is sufficiently small for large $\lambda_0 >>1$. 
Then, from:
$$<u,u>_{L^2_k }  = <l(u), l(u)>_{L^2_{k-1}} +<u,u>_{L^2} $$
it follows that $l$ is close to preserve the norms.
\end{proof}

Let $l: L^2_k(X;E') \cong L^2_{k-1}(X;E)$ 
be as above, and  $K \subset X$ be a compact subset.
 Consider the restriction:
$$l: L^2_k(K;E')_0 \to L^2_{k-1}(K;E)_0.$$

\begin{prop}\label{asymp-uni}
$l: L^2_k(K;E')_0 \to L^2_{k-1}(K;E)_0$ is asymptotically unitary.
In particular, 
$U := \bar{l}^{-1}\circ l $
is an asymptotic identity.
\end{prop}
\begin{proof}
{\bf Step 1:}
We restate that 
for any $\epsilon >0$, there is a finite dimensional linear subspace $V \subset L^2_k(K;E')_0$
such that 
the restriction:
$$l : V^{\perp} \cap L^2_k(K;E')_0  \to L^2_{k-1}(K;E)_0$$
satisfies:
$$||(l - \bar{l} )| V^{\perp} \cap L^2_k(K;E')_0 || < \epsilon.$$
In particular,  we obtain the estimate:
 $$||(U -\text{ id })| V^{\perp} \cap L^2_k(K;E')_0 || < \epsilon.$$
We will verify this in step $2$ and $3$.

Note that  an eigenvalue can have  infinite multiplicity on $L^2(X)$,
when $X$ is non-compact.
For such  cases, the above estimate does not hold.

{\bf Step 2:}
Let $P[0, \lambda]: L^2(X;E') \to L^2(X;E')$ be the spectral projection of $l^* \circ l$, 
and $B_K \subset L^2_k(K;E')_0$ be the unit ball.

We claim  that  $P[0, \lambda](B_K) \subset L^2_k(X;E')$ is relatively compact
 for every $\lambda >0$.
 
In fact, the inclusion $L^2_{k+1}(X;E') \to L^2_{k}(X;E')$
 is compact by Rellich's Lemma.
 
 Since the bounded map $P[0, \lambda]:  L^2_k(X;E')   \to  L^2_{k}(X;E') $ extends
 to a bounded map 
 $P[0, \lambda]:  L^2_k(X;E')   \to  L^2_{k+1}(X;E') $, the former map factors through
 the last one. Then the composition is relatively compact.
 This verifies the claim.

{\bf Step 3:}
Let us take an orthonormal basis $\{u_i\}_i \subset L^2_k(K;E')_0$,
 and set
  $u_i^1 = P[0, \lambda](u_i)$ with 
  $u_i^2 = u_i - u_i^1 \in P(\lambda, \infty)(B_K)$.
It follows from Step $2$ that a subsequence of $\{u_i^1\}_i$ converges in $L^2_k$.
In particular, for any $\epsilon >0$, there is
 a finite-dimensional
vector space $V'$ spanned by $\{u^1_{i_1}, \dots , u^1_{i_m}\}$ for some $\{i_1, \dots, i_m\}$
such that:
 $$||(1- \text{pr}_{V'})u^1_i||_{L^2_k}  < \epsilon$$
holds for all $i$.

Let  $V  \subset L^2_k(K;E')_0$ be a finite-dimensional vector space spanned 
by $\{u_{i_1}, \dots, u_{i_m} \}$.
Then, the inclusion:
$$V \subset V' \oplus P[\lambda, \infty)$$
holds. Moreover, for any $i$, there is $u_i' \in L^2_k(X;E)$ with 
$||u^1_i -(u_i')^1||_{L^2_k} < \epsilon$ such that:
$$u_i' \in V' \oplus P[\lambda, \infty)$$
holds.
Then the conclusion follows by Lemma \ref{spe-proj}.
\end{proof}

In particular, if we apply Proposition \ref{asymp-uni} to the monopole map over
a compact four-manifold,  we obtain the following.

\begin{cor}\label{st.fin.appr}
Let $F =l+c: H' \to H$ be the monopole map 
over a compact oriented four-manifold  $M$ with $b^1(M)=0$,
such that the Fredholm index of $l$ is zero. 

Then, $F$ is strongly finitely approximable.
\end{cor}

\begin{proof}
This follows from \cite{bauer and furuta} with Proposition \ref{asymp-uni}.
\end{proof}

\begin{rem}\label{rem-uni}
$(1)$
In the case of the covering monopole map:
 \begin{align*}
\tilde{\mu}: & \  L^2_k((X,g);  \tilde{S}^+) 
\oplus  L^2_k((X,g); \Lambda^1 \otimes  i {\mathbb R})  \cap \text{ Ker } d^*
\to  \\
  & \qquad \qquad 
   L^2_{k-1}((X,g);  \tilde{S}^-  \oplus  \Lambda^2_+ \otimes  i {\mathbb R})
\oplus  H^1_{(2)}(X)   & \\
& \quad ( \phi, a) \to (  F_{\tilde{A}_0 , \tilde{\psi}_0}(\psi,a) ,  [a])
\end{align*}
the target space is the sum of  the Sobolev space  with $H^1_{(2)}$,
where the latter space is infinite-dimensional if not zero.

By Hodge theory, Ker $d^*$ decomposes as 
$d^*(L^2_{k+1}(X; \Lambda^2_+)) \oplus H^1_{(2)}(X)$,
 and so
 $$ \cup_i  \ d^*(L^2_{k+1}(K_i; \Lambda^2_+)_0) \  \oplus H^1_{(2)}(X)$$
 is dense in Ker $d^*$. 
 
 The restriction of the  linearized map 
 on the harmonic part 
is, in fact, isometry. Hence,
  the covering monopole map is also 
asymptotically unitary over compactly supported Sobolev spaces.

$(2)$
In  the case when $l: L^2_k (X) \cong L^2_{k-1}(X)$ 
gives a linear isomorphism, 
we can  use  the Sobolev norms by
$$<u,v>_{L^2_k} = <(l^*l)^ku, v>_{L^2}.$$
Then $l : L^2_k(X) \cong  L^2_{k-1}(X)$ is unitary with respect to this particular norm.
(See the paragraph below the proof of Lemma \ref{self-adjoint}).
\end{rem}

Now consider the case of a covering monopole map.

\begin{cor}\label{as-uni-family}
Assume the conditions  in Proposition \ref{fin-approx}.
 Then $F$ can admit an asymptotic
   $\Gamma$-approximation.
 \end{cor}

\begin{proof} 
Let us take an
exhaustion by compact subsets $\cup_i K_i =X$. 
It follows from Proposition \ref{fin-approx} 
 that 
  there is a family 
of finite-dimensional linear
subspaces $W'_i \subset L^2(K_i,E')$ 
that satisfies the conditions in 
Definition \ref{weak-f.appr}.
We may assume that it is adapted such that $W_i' \subset L^2_k(K_i, E')_0$.
  
  To obtain an asymptotically $\Gamma$-finite approximation, we set
  $H_i': = L^2_k(K_i;E')_0$ with $H': = L^2_k(X;E')$ in Definition \ref{def6.2}.
  Then, from
Proposition \ref{asymp-uni}, the restriction
  $l: L^2_k(K; E')_0  \to L^2_{k-1}(K;E)_0$ is 
  asymptotically unitary on each compact subset $K \subset X$.
\end{proof}

\begin{rem}
Let us describe some  functional analytic aspect of  a differential operator 
acting on Sobolev spaces.
Let  $\tilde{l}_i: L^2_k(K_i;E')_0 \to L^2_{k-1}(K_i; E)_0$
be  the  restriction of $l:L^2_k(X;E') \to L^2_{k-1}(X; E)$
such that the equality
 $\tilde{l}_i^*\tilde{l}_i = \text{pr}_{L^2_k(K_i)_0} \circ \tilde{l}_{i+1}^*\tilde{l}_{i+1}$
holds on $L^2_k(K_i;E')_0$,
where 
$\tilde{l}_i^*: L^2_{k-1}(K_i;E')_0 \to L^2_{k}(K_i; E)_0$
is the adjoint operator between these Hilbert spaces.

We claim that the restrictions of the self-adjoint
operators below satisfy
 the equality
$$\tilde{l}_i^*\tilde{l}_i
 |W'_{i_0} =   \tilde{l}_{i+1}^*\tilde{l}_{i+1} 
 |W'_{i_0}$$
 for $i_0 < i$.
Note that  $l$ is assumed to be
 a first order differential operator.
Let $\varphi_i \in C^{\infty}_c(K_{i+1})$ be a cut-off function
with 
$\varphi_i|K_{i_0} \equiv 1$ and $\varphi_i|K_{i}^c \equiv 0$.
Consider the equalities 
among  the inner product values as:
\begin{align*}
<  \tilde{l}_{i+1}^*\tilde{l}_{i+1}(v), v'> & =<\tilde{l}_{i+1}(v), \tilde{l}_{i+1}(v')> 
= <\tilde{l}_{i}(v), \tilde{l}_{i}(\varphi_i v')> \\
& = <\tilde{l}_i^*\tilde{l}_i (v), \varphi_i v'> = <\varphi_i \tilde{l}_i^*\tilde{l}_i (v), v'>
\end{align*}
for any unit vectors
$v \in W'_{i_0} \subset L^2_k(K_{i_0}; E')_0
$ and 
 $v' \in 
L^2_k(K_{i+1};E')_0 $.
Hence, the equality
$ \tilde{l}_{i+1}^*\tilde{l}_{i+1} = \varphi_i \tilde{l}_i^*\tilde{l}_i$ holds
on $W'_{i_0}$. In particular:
$$<\tilde{l}_{i+1}^*\tilde{l}_{i+1}(v), v''>=0$$
vanishes for any 
$v'' \in L^2_k(K_i;E')_0^{\perp} \cap L^2_k(K_{i+1};E')_0$.
Thus, if we decompose $v' = u+v'' \in L^2_k(K_{i+1};E')_0
$ with $u \in L^2_k(K_i;E')_0$, then the equality:
$$<\tilde{l}_{i+1}^*\tilde{l}_{i+1} (v), v'>= <\tilde{l}_{i+1}^*\tilde{l}_{i+1} (v), u> 
=<\tilde{l}_{i}^*\tilde{l}_{i}(v),u>$$
holds, which verifies the claim.

The above argument  implies the inclusion:
$$\tilde{l}_{j}^*\tilde{l}_{j}(
 L^2_k(K_{i-1}; E')_0) \ \subset \  L^2_{k}(K_{i};E')_0$$
 for any $j \geq i$.
 \end{rem}

\subsection{Induced Clifford $C^*$-algebras}
We recall the construction of the induced Clifford $C^*$-algebras in \cite{kato4}.
Assume that 
$F = l+c: H' \to H$ is 
 finitely approximable as in Definition \ref{def6.2}
 with respect to the data
$ W'_0   
 \subset   \dots  \subset  W'_i   \subset  \dots  \subset H' $
 with open disks
 $B_{r_i}' \subset W_i'$ and  $B_{s_i} \subset W_i$, and 
$F_i =l_i +c_i: W_i' \to W_i$.

 Let $S_r  :=C_0(-r,r) \subset S$ be the set of   continuous functions on $(-r,r)$
 vanishing at infinity, and 
 consider the following $C^*$-subalgebras:
$$   S_{r_i} {\frak C}(B'_{r_i}):=
S_{r_i} \hat{\otimes} C_0(B'_{r_i}, Cl(W'_i)) .$$
Since the inclusion
$F_i^{-1}(B_{s_i}) \subset B_{r_i}'$ holds, 
it induces a $*$-homomorphism:
$$F_i^*: S_{s_i} {\frak C}(B_{s_i})
\to S_{r_i} {\frak C}(B'_{r_i})$$
given by
$ F_i^*(h)(v') := \bar{l}_i^{-1}(h(F_i(v')))$.
Denote its image by:
$$S_{r_i} {\frak C}_{F_i}(B'_{r_i}) :=
F_i^*(S_{s_i} {\frak C}(B_{s_i})),$$
which is a $C^*$-subalgebra with the norm
$||\quad||_{S_{r_i} {\frak C}_{F_i}}$.

Let us say that a family
 of elements
$\alpha_i \in S_{r_i}{\frak C}_{F_i}(B_{r_i}' )$
 is $F$-{\em compatible} if
there is an element $u_{i_0} \in S_{s_{i_0}}{\frak C}(B_{s_{i_0}} )$
such that:
$$\alpha_i = F_i^*(u_i) \in S_{r_i}{\frak C}_{F_i}(B_{r_i}' )$$
holds for any $i \geq i_0$,
where $u_i = \beta(u_{i_0} ) \in S_{s_{i}}{\frak C}(B_{s_{i}} )
$ with the standard Bott map $\beta$ introduced in \cite{hkt}.
 For an $F$-compatible  sequence 
 $\alpha = \{\alpha_i\}_{i \geq i_0}$, 
  the limit:
$$|| \ \{\alpha_i\}_i \ || :=  \lim_{j \to \infty}
\lim_{i  \to \infty} \
 ||\alpha_i |B_{r_{j}}'||$$
exists because
  both $F_i$ and $l_i$ converge weakly.
 Moreover,
both $F_i^*$ and $\beta$ are
$*$-homomorphisms between $C^*$-algebras and so
 both are norm-decreasing.

\begin{defn} \label{induced Clifford}
Let $F$ be finitely approximable.
The induced Clifford $C^*$-algebra is given by:
$$S{\frak C}_F(H') = \overline{  \{ \ \{ \alpha_i\}_i ;  \ 
F\text{-compatible sequences } \} },
$$
which is obtained by the norm closure of all $F$-compatible sequences,
where the norm is  the one above.
\end{defn}

\subsection{Induced maps on Clifford $C^*$ algebras}
When $H =E$ is finite-dimensional, $S{\frak C}(E)$ 
is given by:
 $$C_0(\mathbb R) \hat{\otimes} C_0(E, Cl(E)),$$
where $Cl(E)$ is the complex Clifford algebra of $E$.
Let  $E'$ and $E$ 
be two finite-dimensional Euclidean spaces, and: 
$$F = l+c : E' \to E$$ be a proper map, 
where $l$ is its linear part.
Assume $l: E' \cong E$ is an isomorphism, 
and let $\bar{l} := l \sqrt{l^*l}^{-1} : E' \cong E$ 
be the unitary  
corresponding to  $l$.
There is a natural pull-back 
$F^*: S{\frak C}(E) \to S{\frak C}(E')$
which is induced from:
$$F^*: C_0(E, Cl(E)) \to C_0(E', Cl(E'))$$
by $u \to v' \to \bar{l}^*(u(F(v')))$, where $\bar{l}: Cl(E') \cong Cl(E)$
is the canonical extension of $\bar{l}$ between the Clifford algebras.
When $F = l+c : H' \to H$ 
is a map between infinite-dimensional Hilbert spaces,
we typically cannot obtain such a general induced map as $F^*$ between 
$S{\frak C}(H)$, in general.

\begin{lem}\label{kato4-st} \cite{kato4}
Let $F = l+c: H' \to H$ be a strongly finitely approximable map.
Then, it induces a $*$-homomorphism:
$$F^* : S{\frak C}(H) \to S{\frak C}(H').$$
\end{lem}
Let us apply $K$-theory.
The above $F^*$  induces a homomorphism:
$$F^* : \mathbb Z \cong K(S{\frak C}(H)) \to K(S{\frak C}(H')) \cong \mathbb Z.$$
It is given by the multiplication of an integer 
that  we call the {\em degree} of $F$ as follows:
$$F^* = \text{ deg } F \times : \  \mathbb Z \to (\text{ deg } F) \ \mathbb Z.$$

\begin{rem}\label{fred}
We can replace 
the condition of linear isomorphism of $l$ with a 
zero Fredholm index (Remark $5.4$ in \cite{kato4}).
\end{rem}

For  finitely approximable $F$, we 
 constructed a variant  $S{\frak C}_F(H)$ of $S{\frak C}(H)$ in \cite{kato4}.
 In fact,
  its construction can be straightforwardly generalized to apply and obtain the $C^*$-algebra.
 
 \begin{lem}
 Let $F$ be asymptotically finitely approximable
  as in Definition \ref{def6.2}.
 Then, there is a $C^*$-algebra  $S{\frak C}_F(H')$.
 \end{lem}
 
 If $H'=E'$ and $H=E$ are both  finite-dimensional, 
 then an asymptotically finitely approximable map is finitely approximable, and
 $S{\frak C}_F(E')$ 
 is given by the image of the induced map:
 $$F^*: S{\frak C}(E) \cong S{\frak C}_F(E') = F^*(  S{\frak C}(E )) \subset 
S{\frak C}(E' ) $$
in the standard sense in basic algebraic topology.

The following property has been verified 
for the class of  finitely approximable maps in \cite{kato4}.
 However, the proof is  parallel to the
case for a broader class of asymptotically finitely approximable maps.

\vspace{3mm}

\begin{prop}\label{kato4-fin.app}
Let $F = l+c: H' \to H$ be a  finitely approximable map.
Then, $F$  induces a $*$-homomorphism:
$$F^* : S{\frak C}(H) \to S{\frak C}_F(H').$$

If a discrete group $\Gamma$ acts on both $H'$ and $H$
linearly and isometrically, then $F^*$ 
is $\Gamma$-equivariant.
\end{prop}
In particular,  $F$  induces a homomorphism:
$$F^* : K_*(S{\frak C}(H)\rtimes \Gamma) \cong 
K_{*+1}( C^*(\Gamma))
 \to K(S{\frak C}_F(H')\rtimes \Gamma),$$
where the isomorphism above is given by \cite{hkt}.

For convenience, 
we now quickly describe how to construct the induced
$*$-homomorphism. Take an element  $v \in S{\frak C}(H)$
 and its approximation 
 $v= \lim_{i \to \infty} v_i $ with $v_i  \in  S_{s_i}{\frak C}(B_{s_i} ) 
 = C_0(-s_i,s_i) \hat{\otimes} C_0(B_{s_i}, Cl(W_i))$.

Consider the induced $*$-homomorphism:
\[
F_i^*: 
S_{s_i}{\frak C}(B_{s_i} ) \to S_{r_i}{\frak C}(B_{r_i}' ) 
\]
and denote its image by
$S_{r_i}{\frak C}_{F_i}(B'_{r_i} ) : =F_i^*(S_{s_i}{\frak C}(B_{s_i} ) )$.

Let $u_i = \beta(v_{i_0})  \in S_{s_i}{\frak C}(B_{s_i} ) 
$
 be the image of the standard Bott map for some $i_0$.
 Then the family  $\{F_i^*(u_i)\}_{i \geq i_0}$ 
 determines an element
 in $S{\frak C}_F(H')$, which 
  gives a $*$-homomorphism:
  \[
F^*: 
S_{s_{i_0}}{\frak C}(B_{s_{i_0}} ) \to S{\frak C}_F(H')
\]
since both $F_i^*$ and $\beta$ 
are $*$-homomorphisms. Note that the composition of the two $*$-homomorphisms:
\[
\begin{CD}
S_{s_{i_0}}{\frak C}(B_{s_{i_0}} ) @>{\beta}>> S_{s_{i'_0}}{\frak C}(B_{s_{i'_0}} )  @>{F^*}>>  S{\frak C}_F(H') 
\end{CD}
\]
coincides with 
$F^*: 
S_{s_{i_0}}{\frak C}(B_{s_{i_0}} ) \to S{\frak C}_F(H')$.

Take two sufficiently large
$i_0' \geq i_0 >>1$ 
such that the estimate
$||\beta(v_{i_0}) - v_{i_0}'|| < \epsilon$ holds for a small $\epsilon >0$. Set
 $u_i' = \beta(v_{i_0'}) \in S_{s_{i}}{\frak C}(B_{s_{i}} )$
  for $i \geq i_0'$.
   Since $F^*$ is norm-decreasing, the estimate
   $||F_i^*(u_i) - F_i^*(u_i')||< \epsilon$ holds for any $i \geq i_0'$.
   Hence, the estimate
    $||F^*(v_{i_0}) - F^*(v_{i_0}')||< \epsilon$ holds.

Thus,  we obtain the assignment 
$v \to \lim_{i_0 \to \infty} F^*(v_{i_0})$, 
which gives a $\Gamma$-equivariant $*$-homomorphism:
\[
F^*: 
S{\frak C}(H) \to S{\frak C}_F(H'),
\]
where $\{v_i\}_i$ is
any approximation of $v$.

\vspace{3mm}

\section{$K$-theoretic higher degree} 
Let $H'$ and $H$ be two  Hilbert spaces on which $\Gamma$ acts linearly and isometrically, and
let  $F = l+c: H' \to H$ be a $\Gamma$ equivariant smooth map
such that $l: H' \cong H$ gives a linear isomorphism.

Assume that 
$F$ is asymptotically $\Gamma$-finitely 
approximable such that  there is
a  family of finite dimensional linear subspaces
$ W'_0   \subset   W'_1 
 \subset   \dots  \subset  W'_i   \subset  \dots  \subset H' $
 with dense union,
  $F_i  : W'_i \to W_i =l_i(W_i')$
with the inclusions $F_i^{-1}(B_{s_i}) \subset B_{r_i}'$, 
and two 
convergences to both $F$ and $l$
(see Definition \ref{def6.2} and  Definition \ref{weak-f.appr}).

Our basic idea is to pull back functions on 
$W_i =l_i(W'_i)$ by $F_{i}$ and combine them.
Consider the induced $*$-homomorphism
$F_i^* : S{\frak C}(W_i)  \to   S {\frak C}(W_i')$ by
$$F_i^*(f \hat{\otimes} h) (v) = f \hat{\otimes} \bar{l}_i^{-1}(h(F_i(v)),$$
where
$ S{\frak C}(W_i) = C_0(\mathbb R) \hat{\otimes}  C_0(W_i, Cl(W_i))$.

We shall give the equivariant degree of the covering monopole map as 
a homomorphism between the K-theory of  $C^*$ algebras.

\begin{thm}\label{cov-degree}
Let $F =l+c: H' \to H$ be the covering monopole map. 
Assume 
 that the linearized operator is  an isomorphism.
 Furthermore, assume that the AHS complex has closed range over the universal covering space.

Then, $F$  induces  the   equivariant $*$-homomorphism
$$F^* :  S{\frak C}(H)   \to  S{\frak C}_F(H').$$

In particular it   induces a map  on $K$-theory as:
  $$F^*: K_*(C^*(\Gamma)) \to 
  K_{*}(S{\frak C}_F (H')\rtimes \Gamma).$$
   \end{thm}
   We call this as the {\em higher degree 
of the covering monopole map}.

\begin{proof}
It follows from Theorem \ref{thm4.1} that the covering monopole map is locally strongly proper.

Then, by 
 Proposition \ref{fin-approx}, it is weakly $\Gamma$-finitely approximable.
 By  Corollary \ref{as-uni-family}, $F$ is actually asymptotically 
$\Gamma$-finitely approximable.

Then, the conclusion follows from Proposition \ref{kato4-fin.app}.
\end{proof}

\vspace{3mm}

Finally we describe the case of the monopole map over a compact manifold.
\begin{prop}
Let $F: H' \to H$ be the monopole map over
a compact oriented four-manifold $M$ with $b^1(M)=0$ and $b^+(M) \geq 1$.
Then, $F$ induces a $*$-homomorphism
$$F^*: S{\frak C}(H) \to S{\frak C}(H').$$

Moreover, 
the induced map:
$$F^*: K(S{\frak C}(H)) \cong {\mathbb Z} \to K(S{\frak C}(H')) \cong {\mathbb Z}$$
is given by   multiplication 
by  the degree $0$ SW invariant.
\end{prop}
\begin{proof}
Suppose the index is non-zero. 
Then,  we simply put the map as  zero.

Suppose  the index is equal to zero.
If $l$ is linearly isomorphic,  then 
 the conclusion follows from  Corollary \ref{st.fin.appr} with Lemma \ref{kato4-st}.
 If the Fredholm index of $l$ is zero, then the same conclusion follows from Remark \ref{fred}.
 The $K$-theoretic degree of $F$ coincides with 
 the degree $0$ SW invariant by a theorem in  \cite{bauer and furuta}.
  \end{proof}

\vspace{1cm}

Tsuyoshi Kato

Department of Mathematics

Faculty of Science

Kyoto University

Kyoto 606-8502
Japan

\end{document}